\documentclass[oneside, letterpaper]{amsart}
\usepackage{amsmath, amssymb}
\usepackage{palatino, euler}
\usepackage{enumitem}
\usepackage{xy}
\xyoption{all}
\usepackage[letterpaper,left=1in,right=1in,top=1in,bottom=1in]{geometry}

\usepackage{setspace}
\usepackage{hyperref}
\usepackage{cite}

\usepackage[usenames, dvipsnames]{xcolor}

\usepackage{todonotes}

\newtheorem{tm}{Theorem}[section]
\newtheorem{pr}[tm]{Proposition}
\newtheorem{lm}[tm]{Lemma}
\newtheorem{lemma}[tm]{Lemma}
\newtheorem{co}[tm]{Corollary}

\theoremstyle{definition}
\newtheorem{df}[tm]{Definition}
\newtheorem{asm}[tm]{Assumption}

\theoremstyle{remark}

\newtheorem{remark}[tm]{Remark}
\newtheorem{example}[tm]{Example}
\newtheorem{construction}[tm]{Construction}

\newcommand{\sh}[1]{{\mathcal{#1}}}
\newcommand{\cat}[1]{{\mathbf{#1}}}
\newcommand{\ho}{\operatorname{ho}}
\newcommand{\ms}{\mathfrak{a}}

\newcommand{\isom}{\cong}
\newcommand{\iso}{\isom}
\newcommand{\weqto}{\overset{\weq}{\longrightarrow}}

\newcommand{\weq}{\simeq}
\newcommand{\id}{\mathrm{id}}

\newcommand{\Aone}{{\mathbb{A}^{\!1}}}
\newcommand{\PAone}{P,{\mathbb{A}^{\!1}}}

\newcommand{\KMW}{\mathrm{K}^{\mathrm{MW}}}
\newcommand{\GW}{\mathrm{GW}}

\newcommand{\proj}{\mathbb{P}}
\newcommand{\Laone}{{\mathrm{L}_{\Aone}}}
\newcommand{\LAone}{\Laone}

\newcommand{\Lnis}{\mathrm{L}_{\textrm{Nis}}}

\newcommand{\Nis}{\mathrm{Nis}}

\newcommand{\sHf}{\mathcal{H}}
\newcommand{\fl}{\textrm{fl}}

\DeclareMathOperator*{\colim}{colim}

\DeclareMathOperator*{\hocolim}{hocolim}
\DeclareMathOperator{\hofib}{hofib}
\DeclareMathOperator{\hocofib}{hocofib}

\newcommand{\rH}{\operatorname{H}}

\newcommand{\Z}{\mathbb{Z}}

\newcommand{\Q}{\mathbb{Q}}

\newcommand{\C}{\mathbb{C}}
\newcommand{\G}{\mathbb{G}}
\newcommand{\Gm}{\mathbb{G}_m}
\newcommand{\ZZ}{\Z}

\newcommand{\CC}{\C}

\newcommand{\Map}{\operatorname{Map}} % Internal mapping objects
\newcommand{\sMap}{\operatorname{SMap}} % added later, simplicial mapping object
\newcommand{\Hom}{\operatorname{Hom}} % Abelian group of homomorphisms
\newcommand{\Fun}{\operatorname{Fun}}  % Set of functors
\newcommand{\End}{\operatorname{End}} % Ring of endomorphisms
\newcommand{\Set}{\cat{Set}}

\newcommand{\sSet}{\cat{sSet}}
\newcommand{\sPre}{\cat{sPre}}
\newcommand{\Sm}{\cat{Sm}}
\newcommand{\Spaces}{\cat{sPre}(\Sm_k)} 
\newcommand{\Spt}{\cat{Spt}} 
\newcommand{\soneSpt}{\Spt(\Sm_k)}

\newcommand{\Spec}{\operatorname{Spec}}

\newcommand{\Ker}{\operatorname{Ker}}
\newcommand{\Image}{\operatorname{Image}}
\newcommand{\ssp}{\Sigma^\infty}
\newcommand{\tensor}{\otimes}
\newcommand{\pt}{\ast}

\newcommand{\op}{\operatorname{op}}
\newcommand{\bd}{\partial}
\newcommand{\sgn}{\operatorname{sign}}
\newcommand{\pre}{\mathrm{pre}}
\newcommand{\ssimp}{\operatorname{Simp}}
\newcommand{\sk}{\operatorname{sk}}
\newcommand{\Sinf}{\Sigma^{\infty}}

\newcommand{\Loc}{\mathrm{L}}
\newcommand{\constr}[2]{\left({#1 \wedge #2_+}\right)\circ \Sigma \Delta_+}

\newcommand{\sPreK}{\Spaces} %% Comment out (or delete) this line to get back the old behaviour of \sPre--now renamed \sPreK--in sections 7
\newcommand{\hidden}[1]{\footnote{Hidden:  #1}}
\renewcommand{\hidden}[1]{}

\begin{document}
\title{The Simplicial EHP Sequence in $\Aone$--Algebraic Topology}

\author{Kirsten Wickelgren}
\address{School of Mathematics, Georgia Institute of Technology, Atlanta~GA, USA}
\email{wickelgren@post.harvard.edu}
\author{Ben Williams}
\address{Department of Mathematics, University of British Columbia, Vancouver~BC, Canada}
\email{tbjw@math.ubc.ca}

\date{Tuesday, June 15, 2018.}
\subjclass[2010]{Primary 55Q40, 55Q25, 14F42 Secondary 55S35, 19D45.}
\begin{abstract}
We give a tool for understanding simplicial desuspension in $\Aone$-algebraic topology: we show that $X \to \Omega (S^1 \wedge X )\to \Omega (S^1 \wedge X \wedge X)$ is a fiber sequence up to homotopy in $2$-localized $\Aone$ algebraic topology for $X = (S^1)^m \wedge \G_m^{\wedge q}$ with $m>1$. It follows that there is an EHP sequence spectral sequence $$\ZZ_{(2)} \otimes \pi_{n+1+i}^{\Aone} (S^{2n+2m+1} \wedge (\G_m)^{\wedge 2q}) \Rightarrow \ZZ_{(2)} \otimes \pi_{i}^{\Aone, s} (S^{m} \wedge (\G_m)^{\wedge q}).$$ 
\end{abstract}

\maketitle

%%%%%%%%%%%%%% List of Sections %%%%%%%%%%%%%%

\tableofcontents

\section{Introduction}\label{Section:introduction}

Let $\Sigma$ denote the suspension functor from pointed simplicial sets (or topological spaces) to itself, defined
$\Sigma X := S^1 \wedge X$. For some maps $f: \Sigma Y \to \Sigma X$, there is a $g: Y \to X$ such that $f = \Sigma
g$. In this case, $f$ is said to desuspend and $g$ is called a desuspension of $f$. Under certain conditions, the
obstruction to desuspending $f$ is a generalized Hopf invariant, as is proven by the existence of the EHP
sequence 
\begin{equation}\label{X_to_OmegaSigmaX_to_OmegaSigmaX2_sequence}
X \to \Omega \Sigma X \to \Omega \Sigma X^{\wedge 2} \end{equation} 
of James \cite{james1955} \cite{James56a} \cite{James56b} \cite{James57} and Toda
\cite{toda1956} \cite{toda1962} which induces a long exact sequence in homotopy groups in a range, see for example
\cite{Toda1952} or \cite[XII Theorem 2.2]{whitehead2012}. Namely, $f$ desuspends if and only if the generalized Hopf
invariant $$H(f): Y \to \Omega \Sigma Y \stackrel{\Omega f}{\to} \Omega \Sigma X \to \Omega \Sigma X^{\wedge 2} $$ is
null. Because calculations can become easier after applying suspension, it is useful to have such a systematic tool for
studying desuspension.

By work of James  \cite{James56a} \cite{James56b}, it is known that when $X$ is an odd dimensional sphere, \eqref{X_to_OmegaSigmaX_to_OmegaSigmaX2_sequence} is a fiber sequence, and when $X$ is an even dimensional sphere, \eqref{X_to_OmegaSigmaX_to_OmegaSigmaX2_sequence} is a fiber sequence after localizing at $2$. In particular, for any sphere, \eqref{X_to_OmegaSigmaX_to_OmegaSigmaX2_sequence} is a $2$-local fiber sequence. Since the suspension of a sphere is again a sphere, the corresponding fiber sequences for all spheres form an exact couple, thereby defining the EHP spectral sequence \cite{Mahowald82}. The EHP spectral sequence is a tool for calculating unstable homotopy groups of spheres. See for example, the extensive calculations of Toda in \cite{toda1962}. 

We provide the analogous tools for $\Sigma X = S^1 \wedge X$ in $\Aone$-algebraic topology, identifying the obstruction to $S^1$-desuspension of a map whose codomain is any sphere with a generalized Hopf invariant, and relating $S^1$-stable homotopy groups of spheres to unstable homotopy groups, after $2$-localization, by the corresponding EHP special sequence. We leave the $p$-localized sequence for future work.  

Place ourselves in the setting of $\Aone$-algebraic topology over a field \cite{morel1998} \cite{morel2012}; let $\Sm_k$
denote the category of smooth schemes over a perfect field $k$, and consider the simplicial model category
$\sPre(\Sm_k)$ of simplicial presheaves on $\Sm_k$ with the $\Aone$ injective local model structure, which will be
recalled in Section \ref{Section_overview_A1_homotopy_theory}. This model structure can be localized at a set of primes
$P$ (see \cite{hornbostel2006} and Section \ref{sec:localization}) giving rise to the notation of a $P$-local fiber
sequence up to homotopy. See Definition  \ref{Def:fiber_sequence_up_to_homotopy}. Define the notation \[S^{n + q \alpha}
= (S^1)^{\wedge n} \wedge (\G_m)^{\wedge q}.\]
Let $\Omega(-)$ denote the pointed $\Aone$-mapping space $\Map_{\sPre(\Sm_k)_{\ast}} (S^1, \Laone -)$, where $\Laone$ denotes $\Aone$ fibrant replacement. 

\begin{tm}\label{EHP-fiber-sequence_introduction}
Let $X = S^{n + q \alpha}$ with $n>1$. There is a $2$-local $\Aone$-fiber sequence up to homotopy $$X \to \Omega \Sigma X \to \Omega \Sigma X^{\wedge 2}.$$
\end{tm}

Let $\pi^{\Aone}_{i}$ denote the $i$th $\Aone$ homotopy sheaf, and more generally define $\pi^{\Aone}_{i + v \alpha}(X)$ to be the sheaf associated to the presheaf taking a smooth $k$-scheme $U$ to the $\Aone$-homotopy classes of maps from $S^{i+ j \alpha} \wedge U_+$ to $X$. The stable $\Aone$ homotopy groups are defined as the colimit $\pi_{i + v \alpha}^{s,\Aone}(X) = \colim_{r \to \infty} \pi_{i + r + v\alpha}^\Aone (\Sigma^r X) $. 

\begin{tm} (Simplicial EHP sequence) \label{tm_EHP_sequence_introduction}
Choose $n,q $ and $v$ in $\ZZ_{\geq 0}$ with $n \geq 2$. 
\begin{itemize}
\item There is a spectral sequence $(E^{r}_{i,j}, d_r: E^{r}_{i,j} \to E^{r}_{i-1,j-r}) \Rightarrow \ZZ_{(2)} \otimes \pi_{i-n + (v-q)\alpha}^{s,\Aone} =  \ZZ_{(2)} \otimes \pi^{\Aone, s}_{i + v \alpha} S^{n + q \alpha}$ with $E^1_{i,j} =   \ZZ_{(2)} \otimes \pi^{\Aone}_{j+1+i + v \alpha} (S^{2j+2n+1+ 2q \alpha})$ if $i \geq 2n-1 +j$ and otherwise $E^1_{i,j} = 0$.
\item Choose $n' > n$. There is a spectral sequence $(E^{r}_{i,j}, d_r: E^{r}_{i,j} \to E^{r}_{i-1,j-r}) \Rightarrow  \ZZ_{(2)} \otimes \pi^{\Aone}_{i+ v\alpha} S^{ n' + q \alpha}$ with $E^1_{i,j} =   \ZZ_{(2)} \otimes \pi^{\Aone}_{j+1+i + v \alpha} (S^{2j+2n+1 + 2q \alpha})$ if $i \geq 2n-1 +j$ and $j< n' - n$,  and $E^1_{i,j} = 0$ otherwise.
\end{itemize}
\end{tm}

Theorem \ref{tm_EHP_sequence_introduction} follows directly from Theorem \ref{EHP-fiber-sequence_introduction}. Theorem \ref{EHP-fiber-sequence_introduction} is a summary of a more refined theorem, giving conditions under which \eqref{X_to_OmegaSigmaX_to_OmegaSigmaX2_sequence} is a fiber sequence without $2$-localization. To state this theorem, let $\GW(k)$ denote the Grothendieck-Witt group of $k$, and consider the element of $\GW(k)$ given by $ -\langle -1 \rangle = -(1+ \rho \eta)$, where $\eta$ is the motivic Hopf map and $\rho=[-1]$ in the notation of \cite[Definition 3.1]{morel2012}. Let $\KMW$ denote Milnor-Witt $K$-theory defined \cite[Definition 3.1]{morel2012}.  For a set of primes $P$, write $\ZZ_P$ for the ring $\ZZ$ with formal multiplicative inverses adjoined for all primes not in $P$.

\begin{tm}\label{EHP-fiber-sequence_without_localization_introduction}
Let $X = S^{n + q \alpha}$ with $n>1$, and let $e = (-1)^{n+q} \langle -1 \rangle^q$. Let $P$ be a set of primes. The sequence $$ X \to \Omega
\Sigma X \to \Omega \Sigma X^{\wedge 2} $$ is a $P$-local $\Aone$-fiber sequence up to homotopy if  $1 + m(1 + e)$  are units in $\GW(k) \otimes \ZZ_P$ for all positive integers $m$. 
\end{tm}

\begin{co}\label{co_fiber_sequence_specific-P_introduction} In the setting of Theorem \ref{EHP-fiber-sequence_without_localization_introduction}, the sequence \begin{itemize}
\item is always a $2$-local $\Aone$-fiber sequence up to homotopy. 
\item  is an $\Aone$-fiber sequence up to homotopy when $e=-1$ or when $n+q$ is odd and the field $k$ is not formally real.
\end{itemize}

In particular, the sequence is an $\Aone$-fiber sequence up to homotopy
\begin{itemize}
\item when $n$ is odd and $q$ is even.
\item when $n+q$ is odd and $k = \CC$, or more generally, when $n+q$ is odd
  and $k$ is any field such that $2 \eta = 0$ in $\KMW_{\ast}$.
\end{itemize}
\end{co}

Although the statement of Theorem \ref{EHP-fiber-sequence_without_localization_introduction} is a direct analogue of the corresponding theorem in algebraic topology, the proof given here is not a straightforward generalization of a proof in algebraic topology. The difficulty is that $\Aone$-fiber sequences are problematic and $\Aone$-homotopy groups are not necessarily finitely generated. Standard tools like the Serre spectral sequence are not currently available.

If a theorem holds for every simplicial set in a functorial manner, it may globalize in the following sense. First, one may be able to obtain in $\sPre$ a na\"ive analogue by starting with simplicial presheaves instead of simplicial sets, performing corresponding operations, producing corresponding maps in $\sPre$. If the theorem in algebraic topology says that some map is always a weak equivalence (respectively weak equivalence through a range), it may be immediate that the corresponding map is a global weak equivalences (respectively global weak equivalence through a range). If the $\Aone$-invariant analogues of the operations considered in the theorem are obtained by applying $L_{\Aone}$ to the na\"ive analogue (defined by applying the operation in simplicial set to the sections over each $U \in \Sm_k$), then the theorem holds in $\Aone$-algebraic topology.

This is the case of the Hilton-Milnor splitting shown below:

\begin{tm}\label{Hilton_Milnor_intro}
There is a natural isomorphism $$\Sigma \Omega \Sigma X \to \Sigma \vee_{n=1}^{\infty} X^{\wedge n} $$ in the $\Aone$-homotopy category.
\end{tm}

This is also the case for the statement that for any simplicial presheaf $X$, the sequence \eqref{X_to_OmegaSigmaX_to_OmegaSigmaX2_sequence} is a fiber sequence up to homotopy in the range $ i \leq 3n - 2$, meaning \begin{align*} & \pi^{\Aone}_{3n-2} X \to \pi^{\Aone}_{3n-1} \Sigma X \to \pi^{\Aone}_{3n-1} \Sigma (X^{\wedge 2}) \to  \ldots \\ \to &\pi^{\Aone}_i X \to \pi^{\Aone}_{i+1} \Sigma X \to \pi^{\Aone}_{i+1} \Sigma (X^{\wedge 2}) \to \pi^{\Aone}_{i-1} X \to \ldots  \end{align*} is exact. This fact is shown in joint work with A. Asok and J. Fasel \cite{AFWW_SimpSuspSeq}.

This is not the case for Theorem \ref{EHP-fiber-sequence_introduction} and Theorem \ref{EHP-fiber-sequence_without_localization_introduction}, i.e. these theorems are not proven by globalizing a corresponding result in algebraic topology, where the sequence \eqref{X_to_OmegaSigmaX_to_OmegaSigmaX2_sequence} fails to be exact for $X = S^n \vee S^n$. See Example \ref{ex:SnSnexample}. 

Here is a sketch of the proof of Theorem \ref{EHP-fiber-sequence_introduction}; its purpose is to help the reader understand the proof given in this paper, and also to explain the similarities with, and differences from the situation in classical algebraic topology. Let $J(X)$ denote the free monoid on a pointed object $X$ in simplicial presheaves on $\Sm_k$, where $\Sm_k$ denotes smooth schemes over $k$. 

In algebraic topology, the free monoid on a pointed object is canonically homotopy equivalent to the loops of the suspension.  It was understood by Fabien Morel that the same result holds in $\Aone$-algebraic topology. Indeed, a result of Morel implies that $\Laone J(X)$ is simplicially equivalent to $\Omega \Laone \Sigma X$, for $X$ pointed, fibrant and connected. (The phrase ``simplicially equivalent" means weakly equivalent in the injective Nisnevich local model structure. Here, ``fibrant" means with respect to this model structure as well.) We show the versions of this result that we need in Section \ref{Hilton-Milnor_Snaith_section}. By globalizing a construction from algebraic topology \cite[VII \S 2]{whitehead2012}, there is a sequence $$X \to J(X) \to J(X^{\wedge 2}),$$ where $X \to J(X)$ is the canonical map induced from the adjunction between $\Sigma$ and $\Omega$, and $J(X) \to J(X^{\wedge 2})$ is the James-Hopf map i.e., the above maps exist in $\Aone$-algebraic topology and the composite map $X \to J(X^{\wedge 2})$ is nullhomotopic (simplicially). Thus there is an induced map in the homotopy category from $X$ to the  $P$-localized $\Aone$-homotopy fiber of $J(X) \to J(X^{\wedge 2})$, where $P$ is a set of primes. Use the notation $h: X \to F$ for this map. Theorems \ref{EHP-fiber-sequence_introduction} and \ref{EHP-fiber-sequence_without_localization_introduction} say that for $X$ a sphere, $h$ is a $P$-localized $\Aone$-homotopy equivalence for appropriate $P$, and it is proved as follows. 

By Theorem \ref{Hilton_Milnor_intro}, there is a map of $S^1$-spectra $b: \Sigma^{\infty} J(X) \to \Sigma^{\infty}
X$. Using the tensor structure of spectra over spaces, it follows that there is a map of  $S^1$-spectra $$c: \Sinf J(X)
\to \Sinf ( X \times J(X^{\wedge 2}))$$ which fits into the commutative diagram $$\xymatrix{\Sigma^{\infty}  J(X)
  \ar[rd] \ar[rr] &&\Sigma^{\infty}  X \times J(X^{\wedge 2}) \ar[ld] \\ &\Sigma^{\infty}  J(X^{\wedge 2}) &}. $$ (See
Section \ref{subsection:stable_weak_equivalence}, and, for general $X$, see Section \ref{subsection_functoriality}, in
particular, the discussion following Construction \ref{cons:added_in_process2}.)

The two spaces $J(X)$ and $ X \times J(X^{\wedge 2}) $ are the same size in the sense that stably they are both weakly equivalent to $\Sigma^{\infty} \vee_{n=1}^{\infty} X^{\wedge n}$. To see this, note that $\Sigma^{\infty} J(X) \cong \Sinf \vee_{n=1}^{\infty} X^{\wedge n}$ by Theorem \ref{Hilton_Milnor_intro}; $$\Sinf (X \times J(X^{\wedge 2})) \cong \Sinf X \vee  \Sinf J(X^{\wedge 2}) \vee \Sinf (X \wedge  J(X^{\wedge 2}) )$$ because the product of two spaces is stably equivalent to the wedge of their smash with their wedge, i.e. $\Sinf (X \times Y) \cong \Sinf (X \vee Y \vee X \wedge Y).$ By Theorem \ref{Hilton_Milnor_intro}, we have stable weak equivalences $J(X^{\wedge 2}) \cong \vee_{n=1}^{\infty} X^{\wedge 2n}$ and $X \wedge J(X^{\wedge 2}) \cong \vee_{n=1}^{\infty} X^{\wedge 2n +1}$. These equivalences, when combined with the previous, show that stably $X \times J(X^{\wedge 2}) \cong \vee_{n=1}^{\infty} X^{\wedge n}$. It is not always the case, however, that the stable map $c: \Sinf J(X) \to \Sinf ( X \times J(X^{\wedge 2}))$ constructed above is a weak equivalence, see Example \ref{ex:SnSnexample}. In algebraic topology, this map is a weak equivalence for $X$ an odd sphere, and an equivalence after inverting $2$ for $X$ an even sphere. We show an analogous fact in $\Aone$-algebraic topology, in the following way. By the Hilton-Milnor theorem, the map $c$ can be viewed as a ``matrix," which itself is the product of matrices corresponding to the diagonal of $J(X)$ and a combination of $b$ with the James-Hopf map $J(X) \to J(X^{\wedge 2})$. Nick Kuhn's calculations of the stable decomposition of the diagonal of $J(X)$ (see \cite{kuhn2001}) and the stable decomposition of the James-Hopf map (see \cite[\S 6]{Kuhn_87}) in algebraic topology globalize to give the matrix entries of $c$ in terms of sums of permutations of smash powers of $X$. Morel computes that the swap map $X \wedge X \to X \wedge X$ is $e$, and more importantly, any permutation $\sigma$ on $X^{\wedge m}$ is equivalent to $e^{\sgn(\sigma)}$ in the homotopy category (see \cite[Lemma 3.43]{morel2012}). Since $X$ is a co-$H$ space, N. Kuhn's results imply that the matrix entries of $c$ are diagonal, and when combined with Morel's result, we calculate the $n$th such entry to be \begin{itemize}\item $1+ \frac{1}{2}(\frac{(2n)!}{2^n n!}-1) (e+1)$ for $n$ even.
\item $(1+ \frac{1}{2}(\frac{(2(n-1))!}{2^{n-1}(n-1)!}-1) (e+1))(\frac{n+1}{2} + \frac{n-1}{2} e)$ for $n$ odd.
\end{itemize} Note that $(2n)!/(2^n n!) = 1(3)(5)\cdots(2n-1)$ is an odd integer, so that the $n$th diagonal term of this matrix is of the form $1 +  m(e+1)$, with $m$ an integer, for $n$ even, and a product of two such terms for $n$ odd. Note that $e^2=1$ in the homotopy category, because $e$ is the class of the swap. It follows that the product of two terms of the form $(1 + m(e+1))$ is also of this form because $(e+1)^2 = 2(1 + e)$. Also note that for any positive integer $m$, we have that $$((m+1) + m e)((m+1) - m e) = 2m+1,$$ whence $(m+1) + m e$ is a unit after localizing at $2$. It follows that $c$ is a weak equivalence after $2$-localization. More generally, $c$ is a weak equivalence after $P$-localization whenever all the terms $(m+1) + m e$ are units in $\GW(k) \otimes \ZZ_P$. See Proposition \ref{pr:stableIsomorphism}. This produces the corresponding hypothesis in Theorem \ref{EHP-fiber-sequence_without_localization_introduction}. We can furthermore characterize exactly when $(m+1) + m e$ is a unit in $\GW(k)$ for all $m$: either $e=-1$ or the field is not formally real and $e = -\langle -1 \rangle$. See Corollary \ref{co:GWUnits}.

We are then in the situation where we have two $P$-localized $\Aone$-fiber sequences \begin{equation}\label{intro_F_to_JX_to_JX2}F  \to J(X) \to J(X^{\wedge 2})\end{equation} \begin{equation}\label{into_Xtimes_JX2_fiber_sequence} X \to X \times J(X^{\wedge 2}) \to  J(X^{\wedge 2}),\end{equation} and a stable equivalence between the total spaces which respects the maps to the base. We would like to ``cancel off" the base $J(X^{\wedge 2})$ to conclude that there is an equivalence between the fibers. This is indeed what we do, however, there are two major obstacles to overcome with this approach. 

The first is that the standard tool to measure the size of a fiber of a fibration in terms of the base and total space is the Serre spectral sequence, and at present there no Serre spectral sequence for $\Aone$-fiber sequences. The desired such sequence would use a homology theory like $\rH^{\Aone}_{\ast}$ (see \cite[Definition 6.29]{morel2012}) because of the need for analogues of the Hurewicz theorem as in \cite[Chapter 6.3]{morel2012} to conclude a weak equivalence between the fibers. We use $S^1$-stable $\Aone$ homotopy groups on the obvious analogue of the Serre spectral sequence defined by lifting the skeletal filtration on the base to express the total space as a filtered limit of cofibrations, and then making an exact couple by applying $\pi_i^{s,P, \Aone}$. This gives a spectral sequence even for a global fibration, but it is not clear that it can be controlled. We provide some of the desired control in Section \ref{subsection_functoriality}. Assume for simplicity that the base is reduced in the sense that its $0$-skeleton is a single point, as is the case for $J(X^{\wedge 2})$. The $E^1$-page can be then identified with $\pi_i^{s,P,\Aone}$ applied to a wedge indexed by the non-degenerate simplices of the base of the fibration. This wedge construction takes $P$-local $\Aone$-weak equivalences of the fiber (respectively $P$-local $\Aone$-weak equivalences in a range) to $P$-$\Aone$ weak equivalences (respectively in a range). See Lemmas \ref{wedge_over_PreSheaf_sets_preserves_Aone_w-e}, \ref{connectivity_wedge_over_PreSheaf_sets}, and \ref{connectivity_wedge_over_PreSheaf_sets_plus_section}. We then show that this identification of the $E^1$-page is natural with respect to maps, and even natural with respect to the stable map $c$ discussed above. See Lemma \ref{E1_functoriality}. This identification of the $E^1$-page does not behave well with respect to weak equivalences of the base, as it involves the specific simplices of the base. It is sufficient here because the map on the base is the identity. We do not understand the $E^2$ page.

We then have a map of spectral sequences from the spectral sequence associated to \eqref{intro_F_to_JX_to_JX2} to the spectral sequence associated to \eqref{into_Xtimes_JX2_fiber_sequence}. We wish to use this map of spectral sequences to show that the stable weak equivalences of the base and total space imply a stable $\Aone$ equivalence of the fibers, after appropriately localizing. 

Then comes the second difficulty. There are infinitely many non-vanishing stable homotopy groups of the fibers in question, and these groups themselves are not necessarily finitely generated abelian groups. We need to show that there is an isomorphism of these $E^1$-pages, but to do this, we need to allow for the possibility that all terms of both spectral sequences are non-zero non-finitely generated groups. We give an inductive argument to do this in Proposition \ref{bsmashf_we_implies_ba_we}, and immediately following the proposition there is a verbal description of what happened.  

The strategy of this proof of the motivic EHP sequence is modeled on the proof of the EHP sequence given in Michael Hopkins's stable homotopy course at Harvard University in the fall of 2012. Hopkins credits this proof to James \cite{james1955} \cite{James56b} together with some ideas of Ganea \cite{Ganea1968}. In this argument, the original Serre spectral sequence is used; there is no need to work in spectra, as calculations in (co)homology suffice. Since the (co)homology of spheres in algebraic topology is concentrated in two degrees, there is no analogue of Proposition \ref{bsmashf_we_implies_ba_we}.

It is also possible to compute the first differential in the EHP sequence of Theorem \ref{tm_EHP_sequence_introduction}, and this computation will be made available in a joint paper with Asok, Fasel and the present authors \cite{AFWW_SimpSuspSeq}.

Computations of unstable motivic homotopy groups of spheres can be applied to classical problems in the theory of projective modules, for example to the problem of determining when algebraic vector bundles decompose as a direct sum of algebraic vector bundles of smaller rank. See \cite[Chapter 8]{morel2012}, \cite{Asok_Fasel_coh_class14}, and \cite{Asok_Fasel14}.

In a different direction, it can be shown that there is an $\Aone$ weak equivalence $\Sigma (\proj^1 -\{0,1,\infty\}) \cong \Sigma (\mathbb{G}_m \vee \mathbb{G}_m)$ between the $S^1$ suspensions of $\proj^1 -\{0,1,\infty\}$ and $\mathbb{G}_m \vee \mathbb{G}_m$. By comparing the actions of the absolute Galois group on geometric \'etale fundamental groups, it can be shown that this weak equivalence does not desuspend \cite{VBACWickelgren}. Because the action of the absolute Galois group on $\pi_1^{\textrm{\'et}}(\proj^1_{\overline{\mathbb{Q}}} -\{0,1,\infty\})$ is both tied to interesting mathematics \cite{Ihara1991} and obstructs desuspension, it is potentially also of interest to have systematic tools like those provided by the EHP sequence to study the obstructions to desuspension.

\subsection{Remark on the field}

Throughout this paper, we will use $k$ for a field such that the unstable connectivity results of Morel apply in the
$\Aone$ homotopy theory over $k$. Specifically, we rely on \cite[Theorems 5.46 and 6.1]{morel2012}. At present, these
results require $k$ to be perfect, but this may well be unnecessary. The requirement that $k$ should be infinite that
arises in the use of \cite[Lemma 1.15]{morel2012} can be circumvented by use of \cite{hogadi2018}.

\subsection{Organization} The organization of this paper is as follows: Theorem \ref{tm_EHP_sequence_introduction} is proven in Section \ref{Section_Aone_simplicial_EHP_fiber_sequence} as Theorems \ref{EHP_spectral_sequence} and \ref{truncated_EHP_spectral_sequence}. Theorem \ref{EHP-fiber-sequence_introduction} is proven in Section \ref{Section_Aone_simplicial_EHP_fiber_sequence} as Theorem \ref{p-local_fiber_sequence_J(X)}. The core of these arguments is the cancelation property of Section \ref{subsection_cancelation_property}. The substitute for the Serre spectral sequence is developed in Section \ref{Section:Fiber_JH_map}. In Section \ref{Section:stable_isomorphism}, the motivic James-Hopf map and the diagonal of the James construction are computed stably as matrices with entries in $\GW(k)$. Section \ref{Hilton-Milnor_Snaith_section} proves the Hilton-Milnor splitting. Section \ref{section:GW} gives results on the Grothendieck--Witt group that are needed to understand when the matrices computated in Section \ref{Section:stable_isomorphism} are invertible. Section \ref{sec:localization} provides needed results on localizations of $\Spaces$ and $\Spt(\Sm_k)$, and Section \ref{Section_overview_A1_homotopy_theory} introduces the needed notation and background on $\Aone$-homotopy theory.

\subsection{Acknowledgements} We wish to strongly thank Aravind Asok and Jean Fasel for their generosity in sharing with
us unpublished notes about the James construction in $\Aone$ algebraic topology.
The first author wishes to thank Michael Hopkins for his stable homotopy course in the Fall of 2012, and the entire homotopy theory community in Cambridge Massachusetts for their
energy, enthusiasm, and perspicacity. We are also pleased to thank Emily Riehl for help using simplicial model
categories, and Aravind Asok, Jean Fasel, David Gepner, Daniel Isaksen, and Kyle Ormbsy for useful discussions. We also thank an anonymous referee for useful comments, in particular for pointing out a gap in a previous version of the proof of Proposition \ref{bsmashf_we_implies_ba_we}.

The first author was supported by an American Institute of Mathematics five year fellowship and NSF grants DMS-1406380 and DMS-1552730. Some of
this work was done while the first author was in
residence at MSRI during the Spring 2014 Algebraic Topology semester, supported by NSF
grant 0932078 000. Further work was done while both authors were in residence at the Institut Mittag-Leffler during the special program ``Algebro-geometric and homotopical methods.'' We thank MSRI and Institut Mittag-Leffler for the pleasant and productive visits we enjoyed with them.

\section{Overview of \texorpdfstring{$\Aone$}{A1} homotopy theory}\label{Section_overview_A1_homotopy_theory}

In the sequel, we will have to draw on many results regarding $\Aone$ homotopy. We collect those results in this
section for ease of reference. We make no claim that any of these results are original.

As stated in the introduction, we assume $k$ is a field so that the Unstable Connectivity Theorem of Morel applies over $k$.

Let $\Sm_k$ denote a small category equivalent to the category of smooth, finite type $k$-schemes. The category
$\Spaces$ is the category of simplicial presheaves on $\Sm_k$, and $\Spaces_\pt$ the category of pointed simplicial
presheaves. The category $\Sm_k$ is considered embedded in $\Spaces$ via the Yoneda embedding. The terminal object of
$\Sm_k$ and $\Spaces$ is therefore $\Spec k$, which is also denoted by $k$ and $\pt$ depending on the context.

The notation $\Map(X,Y)$ denotes the internal mapping object where it appears, generally in $\Spaces$. Many categories
appearing in the sequel are simplicially enriched, and in them $\sMap(X,Y)$ will denote a simplicial mapping
object. Where there is a model structure, $\ms$, present we will use the notation $[X,Y]_\ms$ to denote the set of maps
in the homotopy category from $X$ to $Y$. The notation $[X,Y]$ will be used when $\ms$ is clear from the context.

If $K$ is a simplicial set, then we write $K_i$ for the set of $i$--simplices in $K$.

\subsection{Model Structures}\label{subsection:model_structures}

This paper makes use of two families of model structure on the category $\Spaces$ and its descendants. In the first
place, the local injective model structure of \cite{jardine1987-a}---introduced there as the `global' model
structure---and the local flasque model structure of \cite{isaksen2005}. Our use of these terms follows
\cite{isaksen2005}. These model structures are Quillen equivalent. Each gives rise to descendent model structures by $\Aone$-- or $P$--
localization or by stabilization. The flasque model structures are employed only to prove technical results regarding
spectra; when `flasque' is not specified, it is to be understood that the injective structures are meant.

The weak equivalences in the injective local and the flasque local model structures are the local weak
equivalences---those maps that induce isomorphisms on homotopy sheaves, properly defined: \cite{jardine1987-a}. In the
seminal work \cite{morel1998}, these maps are called `simplicial weak equivalences' in order to emphasise their
non-algebraic character.

Both the injective and the flasque local model structures are left Bousfield localizations of global model structures on
$\Spaces$; a global model structure being one where the weak equivalences are those maps $\phi: X \to Y$ that induce weak
equivalences  $\phi(U): X(U) \to Y(U)$ for all objects $U$ of $\Sm_k$. Both the global injective and the global flasque
model structures are left proper, simplicial, cellular, see \cite{hornbostel2006}, and combinatorial so that left Bousfield localizations of either
at any set of morphisms exist and are again left proper, simplicial and cellular. In the injective model structures
all objects are cofibrant, and therefore these model structures are \textit{tractable} in the sense of
\cite{barwick2010}.

We shall need a cartesian model category structure on $\Spaces$ from time to time. The category $\Spaces$ is cartesian
closed as a category in its own right, and it is well known---and proved in \cite[Application IV]{barwick2010}---that
the injective local model structure is a symmetric monoidal model category in the sense of \cite[Chapter 4]{hovey1999}
with the cartesian product providing the monoidal operation. One then may wish to establish that some model structure
$\ms$ obtained as a left Bousfield localization of this structure inherits the structure of a monoidal model
category.

\begin{pr} \label{pro:symm_monoidal}
  Suppose $\ms$ is a localization of the injective model structure on $\Spaces$ such that $\ms$ is a simplicial monoidal model
  category with respect to the cartesian product. Let $A$ denote a set of morphisms in $\Spaces$ such that for all
  objects $U$ of $\Sm_k$, if $f: X \to Y$ is in $A$, there is a morphism in $A$ isomorphic to $f \times \id_U : X \times
  U \to Y \times U$. Then the localization of $\ms$ at $A$ inherits the monoidal model category structure of $\ms$.
\end{pr}
\begin{proof}
  This is an application of \cite[Proposition 4.47]{barwick2010}. Here we assume that we are working within some
  universe $\mathbf X$. The role of $\mathbf V$ is played by $\sSet$. The model category $\ms$ is left proper because
  all objects are cofibrant. Then the hypotheses of \cite[Proposition 4.47]{barwick2010} are that $\ms$ is symmetric
  monoidal, and that there exists a set of homotopy generators of $\ms$, here taken to be the representable objects $U$,
  such that if $Z$ is $A$-local and $U$ is representable, the internal mapping object $\Map(U, Z)$ is again
  $A$-local. By adjunction, this follows from our hypotheses.
\end{proof}

\begin{co}\label{smash_preserve_ms_we} 
  Let $\ms$ be a symmetric monoidal model structure on $\Spaces$ where the monoidal operation is given by the cartesian
  product. Let $\ms_\pt$ denote the pointed analogue. If  $X$ is an object of $\Spaces$, then the functor
  $X \times \cdot$ preserves $\ms$ weak equivalence. If $X$ is an object of $\Spaces_\pt$, then the functor
  $X \wedge \cdot$ preserves $\ms_\pt$ weak equivalences.
\end{co}

\begin{proof}
  Let $f: Z \to Y$ be a $\ms$ weak equivalence. Because $\ms_{\pt}$ is a simplicial model category in which all objects
  are cofibrant and monomorphisms are cofibrations, it follows from \cite[Corollary 14.3.2]{Riehl} that
  $\id_{X} \vee f: X \vee Z \to X \vee Y$ is a $\ms$ weak equivalence.

  By Proposition \ref{pro:symm_monoidal}, for any object $X$, the functor $X \times \cdot$ preserves trivial
  cofibrations; by Ken Brown's lemma, it therefore preserves weak equivalences between cofibrant objects, but all
  objects are cofibrant.

  Note that $\id_{X} \vee f$, $\id_{\pt}$, $X \times f$, and $X \wedge f$ determine a map of push-out squares as in the
  commutative diagram \[ \xymatrix{
    & \pt \ar[rr]\ar'[d][dd] & & X \wedge Z \ar[dd]\\
    X \vee Z \ar[ur]\ar[rr]\ar[dd] & & X \times Z \ar[ur]\ar[dd] & \\
    & \ast \ar'[r][rr] & &  X \wedge Y \\
    X \vee Y \ar[ur]\ar[rr] & & X \times Y \ar[ur]& }
\]
Furthermore, $X \vee Z \to X \times Z$ and $X \vee Y \to X \times Y$ are cofibrations because they are monomorphisms. It
now follows from \cite[Corollary 14.3.2]{Riehl} that $X \wedge f: X \wedge Z \to X \wedge Y$ is a $\ms$ weak equivalence
as claimed.
\end{proof}

\subsection{Homotopy Sheaves}

If $X$ is an object of $\Spaces$ or $\Spaces_\pt$, we write $\Lnis X$ for a functorial fibrant replacement in the local
injective model structure; this also serves as a fibrant replacement in the flasque model structure. We write $\LAone$
for a functorial fibrant replacement in the $\Aone$ model structures.

Since the purpose of this paper is to establish some identities regarding $\Aone$ homotopy sheaves, it behoves us to
define what a homotopy sheaf means in the sequel.

The following definitions date at least to \cite{jardine1987-a}.
\begin{df}
  If $X$ is an object of $\Spaces$, then we define $\pi_0^\pre(X)$ as the presheaf
  \[ U \mapsto \pi_0( |X(U)|) \]
  where $U $ is an object of $\Sm_k$, and where $|X(U)|$ indicates a geometric realization of $X(U)$. We define $\pi_0(X)$ as the associated
  Nisnevich sheaf to $\pi_0^\pre(X)$.
\end{df}

\begin{pr} \label{pr:pi0asCoeq}
  If $X$ is an object of $\Spaces$, then
  $\pi_0(X)$ is the sheaf associated to the presheaf coequalizer:
  \[ U \mapsto \mathrm{coeq} \left(\xymatrix{  X(U)_1\ar@<3pt>^{d_1}[r] \ar@<-3pt>_{d_0}[r] & X(U)_0 }\right).  \] 
\end{pr}

\begin{df}
  If $X$ is an object of $\Spaces_\pt$, with basepoint $x_0 \to X$, then we define $\pi_i^\pre(X, x_0)$ for $i \ge 1$ as the presheaf
  \[ U \mapsto \pi_i( |X(U)|, x_0) \]
  where $U $ is an object of $\Sm_k$, and where $x_0$, in an abuse of notation, indicates the basepoint of $|X(U)|$
  induced by $x_0 \to X$. We define $\pi_i(X, x_0)$ as the associated Nisnevich sheaf. The basepoint $x_0$ will
  generally be understood and omitted.
\end{df}

The reader is reminded that $X(U)$ may have connected components that do not appear in the global sections,
$X(\pt)$. In this case, the sheaves of groups $\pi_i(X, x_0)$ as defined above are insufficient to describe the homotopy theory of $X$.

It is the case that the functor $\pi_i(\cdot)$ takes simplicial weak equivalences to isomorphisms, and $\pi_0^\Aone$ takes
$\Aone$ weak equivalences to isomorphisms.

If $X$ is an object of $\Spaces_\pt$, we reserve the notation $\Omega^i X$ for the derived loop space $\Map_\pt( S^n, \Lnis X)$. In particular,
$\Omega^0 X \iso \Lnis X$.

We rely on the following result throughout.
\begin{pr}
  Equip $\Spaces_\pt$ with the Nisnevich local model structure. If $X$ is an object of $\Spaces_\pt$, and if $i \ge 0$, then $\pi_i(X)$ is the sheaf associated to the presheaf
  \[ U \mapsto [\Sigma^i(U_+), X]_\pt. \]
\end{pr}
\begin{proof}
  By applying a functorial fibrant replacement functor if necessary, we may assume that $X(U)$ is a fibrant simplicial
  set for all objects $U$ of $\Sm_k$. In the following sequence of isomorphisms, all homotopy groups considered are
  simplicial homotopy groups of fibrant simplicial sets
  \begin{align*}
    [ S^i \wedge U_+, X]_\pt & \iso \pi_0(\sMap_\pt(S^i \wedge U_+, X)) \\
    & \iso \pi_0 (\sMap_\pt(S^i , \Map_\pt(U_+, X))) \\
    & \iso \pi_0 (\sMap_\pt(S^i, \Map(U, X))) \\
    & \iso \pi_0 (\sSet_\pt(S^i, \Map(U,X)(\pt))) \\
    & \iso \pi_0(\sSet_\pt(S^i, X(U))) \\
    & \iso \pi_i(X(U)),
  \end{align*}
  as required.
\end{proof}

\begin{co} \label{co:piipi0}
  If $X$ is an object of $\Spaces_\pt$, and if $i \ge 0$, then $\pi_i(X) \iso \pi_0(\Omega^i X)$.
\end{co}
\begin{proof}
  The result follows from the proposition and the adjunction
  \[ [\Sigma^i(U_+) , X ]_\pt \iso [U_+, \Omega^i  X]_\pt. \]
\end{proof}

Since taking global sections, $X \mapsto X(\ast)$, is taking a stalk, we also have the following corollary.
\begin{co}
  If $X$ is an object of $\Spaces_\pt$, and if $i \ge 0$, then
  \[  \pi_i(X)(\ast) \iso [S^i, X]. \]
\end{co}

If $i, j \ge 0$ we define 
\[ S^{i + j \alpha} = S^i \wedge \Gm^{\wedge j},\] 
where $\Gm$ is pointed at the rational point $1$. If $j \ge 0$ and $X$
  is an object of $\Spaces_+$, we define $\pi^\Aone_{i + j \alpha}(X)$ as $\pi_i(\Map_\pt(\Gm^{\wedge j}, L_\Aone X))$; it is
  isomorphic to the sheaf associated to the presheaf $U \mapsto [ S^{i+ j \alpha} \wedge U_+, X]_\Aone $. Taking global
  sections, we have
\[ \pi_{i + j \alpha}^\Aone(X) (k) \iso [S^{i+ j \alpha}, X]_\Aone. \]

\subsection{Compact Objects and Flasque Model Structures}

We say that an object $X$ of $\Spaces_\pt$ is \textit{compact} if 
\[ \colim_i \Map_\pt(X, F_i) \iso \Map_\pt( X, \colim_i F_i) \]
whenever $F_i$ is a filtered system in $\Spaces$, and similarly for $\Spaces_\pt$. An argument similar to that of
\cite[Lemma 9.13]{dugger2005} shows that pointed smooth schemes are compact, and it is easy to see that finite constant
simplicial presheaves are compact. If $A$, $B$ are pointed compact objects, then $A \wedge B$ is compact, and all finite colimits
of compact objects are again compact.

We shall frequently make use of the following. 

\begin{pr} \label{pr:sequentialHocolim}
Let
  $I$ be a filtered small category and $X$ an $I$-shaped diagram in $\Spaces$ (resp.~$\Spaces_\pt$). Then the natural
  map $\hocolim X \to \colim X$ is a global weak equivalence.
\end{pr}
\begin{proof}
  We describe the case of $\Spaces$, that of $\Spaces_\pt$ is similar.

  By construction, \cite[Chapter 18]{hirschhorn2003}, $(\hocolim X)(U) \iso \hocolim (X(U))$ for all $U \in \Sm_k$,
  and similarly for $\colim$. The result then follows from the classical fact that the natural map $\hocolim X(U) \to \colim
  X(U)$ is a weak equivalence, \cite[XII.3.5]{bousfield1972}.
\end{proof}

\begin{pr} \label{pr:colimpi0}
  If $X_0 \to X_1 \to \dots$ is a sequential diagram in $\Spaces$, then the natural map of sheaves
  \[ \colim_i \pi_0( X_i) \to \pi_0 (\colim_i X_i)\]
  is an isomorphism.
\end{pr}
\begin{proof}
  By Proposition \ref{pr:sequentialHocolim}, we may replace $\{X_i\}$ by a naturally weakly equivalent diagram without
  changing the homotopy type of $\colim_i X_i$. The group $\colim_i \pi_0(X_i)$ is also unchanged by such a
  procedure. We can therefore assume that $X_i(U)$ is a fibrant simplicial set for all objects $U$ of $\Sm_k$.

  Since, according to Proposition \ref{pr:pi0asCoeq}, $\pi_0(Y)$ is the sheaf associated to a coequalizer of presheaves,
  provided $Y$ takes values in fibrant simplicial sets, the result follows by commuting colimits. 
\end{proof}

The injective local model structure on $\Spaces_\pt$ suffers from a technical drawback when one wishes to calculate with
filtered colimits, which is that filtered colimits of fibrant objects are not
necessarily fibrant themselves.  This is the problem that motivates the construction of the flasque model structures of
\cite{isaksen2005}, and one can see the presence of flasque or flasque-like conditions appearing often throughout the
literature when calculations with filtered colimits are being carried out, see \cite{jardine2000-a},
\cite{dugger2005}, \cite{morel2005}.

We therefore consider two flasque model structures on $\Spaces$: the local flasque structure in which the weak equivalences are
the simplicial weak equivalences, and the $\Aone$ flasque structure in which the weak equivalences are the $\Aone$ weak
equivalences. These model structures apply also to $\Spaces_\pt$. These model structures are simplicial, proper and cellular, and
the $\Aone$ structures are left Bousfield localizations of the local model structure. There is a square of Quillen
adjunctions
\begin{equation} \label{eq:ms4} \xymatrix{ \text{ Injective Local} \ar^\weq[r] \ar[d] & \text{Flasque Local} \ar[d] \\ \text{Injective $\Aone$}
  \ar^\weq[r] & \text{Flasque $\Aone$} } \end{equation}
where the arrows indicate the left adjoints, and each arrow is the identity functor on $\Spaces$. The horizontal arrows
represent Quillen equivalences. A similar diagram obtains in $\Spaces_\pt$. 

Not all objects are cofibrant in $\Spaces$ or $\Spaces_\pt$ in the flasque model structures, in contrast to the case
of the injective structures. Since the $\Aone$ flasque structures are left Bousfield localizations of the local flasque
structures, the cofibrant objects in one model structure agree with the cofibrant objects in the other. The results of
\cite{isaksen2005}, specifically Lemmas 3.13, 6.2, show that all pointed simplicial sets and all quotients $X/Y$ of
monomorphisms $Y \to X$ in $\Sm_k$ are flasque cofibrant in $\Spaces_\pt$. This includes all smooth schemes pointed at a
rational point. Lemma 3.14 of \cite{isaksen2005} shows that finite smash products of flasque
cofibrant objects are again flasque cofibrant in $\Spaces_\pt$. 

\begin{pr} \label{pr:colimMappingSpaces}
  If $F_i$ is a filtered diagram of objects of $\Spaces_\pt$, and if $X$ is a compact and flasque cofibrant object of $\Spaces_\pt$, then there is 
a zigzag of  local (resp.~$\Aone$) weak equivalences:
  \begin{equation} \label{eq:1}  \colim_i \Map_\pt( X, RF_i ) \to \Map_\pt( X, R \colim_i  R F_i) \leftarrow \Map_\pt (X, R \colim_i F_i),\end{equation}
  where $R$ denotes an injective local (resp.~injective $\Aone$) functorial fibrant replacement.
\end{pr}
\begin{proof}
  By Proposition \ref{pr:sequentialHocolim}, the local (resp.~$\Aone$) homotopy type of a filtered colimit is invariant under
  termwise replacement by locally (resp.~$\Aone$) equivalent objects.
  
  Filtered colimits of flasque fibrant objects are again flasque fibrant, see \cite{isaksen2005}.

  The objects $R F_i$ are flasque fibrant, so the colimit $\colim_i R F_i$ is flasque fibrant, as is $R \colim_i R
  F_i$. There is a global weak equivalence $\colim_i RF_i \weq R \colim_i RF_i$. Since $R$ preserves weak equivalences,
  we also have $R \colim_i R F_i \weq R \colim_i F_i$. Since the object $X$ is flasque cofibrant, the functor
  $\Map_\pt(X, \cdot)$ preserves trivial flasque fibrations, and by Ken Brown's lemma, weak equivalences between flasque
  fibrant objects. The map $\Map_\pt(X , \colim_i R F_i) \to \Map_\pt(X, R \colim_i F_i)$ is therefore a weak
  equivalence. The result now follows from the compactness of $X$.
\end{proof}

\begin{co} \label{co:homotopysheafcolimit}
  If $F_r$ is a filtered system of objects of $\Spaces_\pt$, and if $i, j \ge 0$ are integers, then there are natural
  isomorphisms of sheaves
  \[ \pi_i(\colim_r F_r) \iso \colim_r \pi_i(F_r) \]
  and
  \[ \pi_{i + j \alpha}^\Aone( \colim_r F_r) \iso \colim_r \pi_{i + j \alpha}^\Aone(F_r). \]
\end{co}
\begin{proof}
  Combine Corollary \ref{co:piipi0} and Propositions \ref{pr:colimpi0}, \ref{pr:colimMappingSpaces}, noting that the
  objects $S^{i+j\alpha}$ are compact and flasque cofibrant.
\end{proof}

We warn the reader that $\pi_{i+j \alpha}^\Aone( \Omega^r X)$ differs from $\pi_{i+r+j\alpha}^\Aone(X)$ in
general, \cite[Theorem 6.46]{morel2012}.

\subsection{Spectra}\label{subsection:Spectra}

We take \cite{hovey2001} as our main reference for the theory of spectra in model structures such as those we consider
here. We shall require only na\"ive spectra, rather than symmetric spectra. For us a spectrum, $E$, shall
be an $S^1$--spectrum, consisting of a sequence $\{E_i\}_{i=0}^\infty$ of objects of $\Spaces_\pt$, equipped with
bonding maps $\sigma: \Sigma E_i \to E_{i+1}$. The maps of spectra $E \to E'$ being defined as levelwise maps $E_i \to
E'_i$ which furthermore commute with the bonding maps, we have a category of presheaves of spectra, which we denote by
$\Spt(\Sm_k)$.

Just as we have two notions of weak equivalence on $\Spaces_\pt$, the local and the $\Aone$, we shall have two kinds of
weak equivalence between objects of $\Spt(\Sm_k)$, the stable and the $\Aone$.

There is a set, $I$ in the notation of \cite{isaksen2005}, of generating cofibrations for which the domains and
codomains all posses the property that we call ``compact'', which \cite{isaksen2005} calls ``$\omega$--small'' and
which is stronger than the property that \cite{hovey2001} calls ``finitely presented''. Moreover, both model structures are localizations of an
objectwise flasque model structure having a set, $J$ in the notation of \cite{isaksen2005}, which again consists of
maps having finitely-presented domains and codomains. By the arguments of \cite[Section 4]{hovey2001}, these model
structures are \textit{almost finitely generated}.

The theory of \cite[Section 3]{hovey2001} establishes a stable model structure on $\Spt(\Sm_k)$ based on any cellular,
left proper model structure, $\ms$, on $\Spaces_\pt$. In particular, this applies when $\ms$ is a left Bousfield
localization of the global injective or global flasque model structure, and therefore when it is one of the four
structures of \eqref{eq:ms4}. The results of \cite[Section 5]{hovey2001} ensure that we have
Quillen adjunctions and equivalences between these model structures:
\begin{equation} \label{eq:ms5} \xymatrix{ \text{ Stable Injective Local} \ar^\weq[r] \ar[d] & \text{Stable Flasque Local} \ar[d] \\ \text{Stable Injective $\Aone$}
  \ar^\weq[r] & \text{Stable Flasque $\Aone$}. } \end{equation}
Since the functors of \eqref{eq:ms4} are the identity functors, the same is true of the functors of \eqref{eq:ms5}; only the model structure varies.

We write \textit{stable weak equivalence} for the weak equivalences of the stable injective local and stable flasque
local model structures, and \textit{stable $\Aone$ equivalence} for the weak equivalences of the stable injective $\Aone$
and the stable flasque $\Aone$ model structures. In keeping with our convention, we write $\Laone E$ to denote a fibrant
replacement of $E$ in the stable $\Aone$ model structures.

Since the underlying unstable model structures are proper, we may apply fibrant-replacement functors levelwise to
objects in $\Spt(\Sm_k)$ to obtain maps of spectra: $E \to RE$ given by $E_i \to RE_i$, the fibrant replacement in any
one of the four unstable model structures under consideration. There is also a spectrum-level infinite loop space
functor, $\Theta^\infty$ that takes a spectrum $E$ to the spectrum having $i$-th space
\[ (\Theta^\infty E)_i =  \colim_{k \to \infty} \Map_\pt(S^k, E_{i+k}). \]
 
\begin{pr}
  A map $f: E \to E'$ of $\Spt(\Sm_k)$ is a stable weak equivalences (resp.~a stable $\Aone$ equivalence) if and only if 
  \[ \Theta^\infty (Rf)_i : (\Theta^\infty RE)_i \to (\Theta^\infty RE')_i \]
  is a weak equivalence for all $i$, where $R$ represents the flasque local fibrant replacement functor (resp.~flasque $\Aone$ fibrant replacement functor).
\end{pr}
\begin{proof}
  This is a special case of \cite[Theorem 4.12]{hovey2001}. The ancillary hypotheses given there, that sequential colimits in commute with finite products and that $\Map_\pt(S^1, \cdot)$ commutes with sequential limits, are satisfied in $\Spaces_\pt$.
\end{proof}

One can verify that a spectrum $E$ is weakly equivalent to the spectrum one obtains from $E$ by replacing each space
$E_i$ by the connected component of the basepoint in $E_i$. We may therefore assume that $E_i$ is connected, meaning we
do not have to worry about the problem of non-globally-defined components.

For any integer $i$, there is an adjunction of categories
\begin{equation} \label{eq:SigmaEvAdjunction} \xymatrix{ \Sigma^{\infty-i} : \Spaces_\pt \ar@<2pt>[r] & \ar@<2pt>[l] \Spt(\Sm_k) : \operatorname{Ev}_i }\end{equation}
where the spectrum $\Sigma^{\infty-i} X$ is the spectrum the $j$-th space of which is $\Sigma^{j-i} X$ if $j\ge i$, and $\pt$
otherwise, and where the bonding maps are the evident ones. The right adjoint $\operatorname{Ev}_i$ takes $E$ to
$E_i$. 

\begin{pr} \label{pr:unstableStableAdjunction}
  Suppose $\ms$ is a left Bousfield localization of either the global injective or the global flasque model structure on
  $\Spaces$. Then the adjoint functors of \eqref{eq:SigmaEvAdjunction} form a Quillen pair between the pointed
  model structure on $\Spaces_\pt$ and the stable model structure on $\Spt(\Sm_k)$ induced by $\ms$.
\end{pr}
\begin{proof}
  This follows from Definition 1.2, Proposition 1.15 and Definition 3.3 of \cite{hovey2001}.
\end{proof}

 The left-derived functor of $\mathrm{Ev}_0$ in the flasque model structures are the functors
\[ \Omega^\infty: \{ E_n \}_n \mapsto \colim_k \Omega^k R E_{n+k} \]
where $R$ is either $\Loc^\fl_\Nis$ or $\Loc^\fl_\Aone$, as appropriate. The smash product on $\Spaces_\pt$ extends to an action of
$\Spaces_\pt$ on $\Spt(\Sm_k)$, and the functors $\Sigma^{\infty - i}$ preserve this structure, and in particular 
are simplicial functors, \cite[Section 6]{hovey2001}. 

For an object $E$ of $\Spt(\Sm_k)$, we define the stable homotopy sheaves $\pi_{i}^s(E)$ as the colimit 
\[ \pi_i^s(E) = \colim_{r \to \infty} \pi_{i+r} (E_r). \]
Since $\pi_{i+r}(E_r)$ is the sheaf associated to the presheaf $U \mapsto \pi_{i+r}(E_r(U))$, and sheafification
commutes with colimits, it follows that $\pi_i^s(E)$ may also be described as the sheaf associated to the presheaf
\[ U \mapsto \pi_i^s(E(U)) \]
where the stable homotopy group is the ordinary stable homotopy group of simplicial spectra. This definition of the sheaf $\pi^s_i$ is used in \cite{morel2005}.

We similarly define the stable $\Aone$ homotopy sheaves $\pi_{i + j \alpha}^{s,\Aone}(E)$ as the colimit
\[ \pi_{i + j \alpha}^{s,\Aone}(E) = \colim_{r \to \infty} \pi_{i + r + j\alpha}^\Aone (E_r). \]
When $X$ is an object of $\Spaces_\ast$, will use the notation $\pi_{\ast}^{s}(X)$ and $\pi_{\ast}^{s, \Aone}(X)$ for $\pi_{\ast}^{s}(\ssp X)$ and
$\pi_{\ast}^{s, \Aone}(\ssp X)$ respectively.

\begin{pr}\label{pi_s_detect_we_Aone_and_simplicial} Let $f: E \to E'$ be a map in $\Spt(\Sm_k)$. Then
  \begin{enumerate} 
  \item  $f$ is a stable weak equivalence if and only if $\pi^s_i(f)$ is an isomorphism for all $i$;
  \item  $f$ is an $\Aone$ stable weak equivalence if and only if $\pi^{s, \Aone}_i(f)$ is an isomorphism for all $i$.
  \end{enumerate}
\end{pr}
\begin{proof}
  The map $f: E \to E'$ is a stable weak equivalence if and only if the maps $(\Theta^\infty \Loc^\fl_\Nis E)_i \to (\Theta^\infty \Loc^\fl_\Nis
  E')_i$ are simplicial weak equivalences for all $i$. The space $(\Theta^\infty \Loc^\fl_\Nis E)_j$ is
  \[ \colim_{r \to \infty} \Omega^r (\Loc^\fl_\Nis E_{j+r}) \]
  and its $i+j$-th homotopy sheaf is, by Corollary \ref{co:homotopysheafcolimit}, 
  \[ \pi_{i+j} ( \colim_{r \to \infty} \Omega^r (\Loc^\fl_\Nis E_{j+r})) \iso \colim_{j+r \to \infty} \pi_{i+j+r}(E_{j+r}) \iso
  \pi_i^s(E). \] The result for $\pi_i^s$ follows.

  For $\Aone$ equivalence, the same argument applies \textit{mutatis mutandis}. Writing $\Loc^\fl_\Aone$ for the flasque
  $\Aone$ fibrant replacement functor, we see that the $i+j$-th homotopy sheaf
  of the $j$-th level of the $\Aone$ stable fibrant replacement $\Theta^\infty (\Loc^\fl_\Aone E)$ is $ \pi_{i+j}(\colim_{r \to
    \infty} \Omega^r (\Loc^\fl_\Aone E_{j+r})) $ which simplifies to $\pi_i^{s} (\Loc^\fl_\Aone E) = \pi_i^{s, \Aone}(E)$
\end{proof}
This proposition says that the definition of stable weak equivalence used in this paper agrees with that of \cite{morel2005}.

\begin{pr}
  For any object $E$ of $\Spt(\Sm_k)$ and any nonnegative integers $i$, $i'$ and $j$,
  \begin{enumerate}
  \item The sheaf associated to the presheaf
    \[ U \mapsto [ \Sigma^{\infty - i'}(S^{i } \wedge U_+) , E ] \] is $\pi_{i-i'}^s(E)$
  \item The sheaf associated to the presheaf
    \[ U \mapsto [ \Sigma^{\infty - i'}( S^{i + j \alpha} \wedge U_+) , E ]_\Aone \]
    is $\pi_{i-i'+j\alpha}^{s, \Aone}(E)$.
  \end{enumerate}
\end{pr}
\begin{proof}
  We prove the first statement.

    The given presheaf, by adjunction, is 
  \[ U \mapsto [ S^i \wedge U_+, \operatorname{Ev}_{i'} E ], \]
  where the functor $\operatorname{Ev}_{i'}$ is a derived functor in the flasque stable model structure. By reference to
  \cite{hovey2001}, we write this presheaf more explicitly as
  \[ U \mapsto [ S^i \wedge U_+, \colim_{r \to \infty} \Omega^r \Loc^\fl_\Nis E_{i' + r}] \]
  which is associated to the sheaf 
  \[ \pi_i(\colim_{r \to \infty} \Omega^r \Loc^\fl_\Nis E_{i'+r}) \iso \colim_{r \to \infty} \pi_{i-i'+(i' +r)} (E_{i'+r}) \iso
  \pi_{i-i'}^s(E),\]
  as asserted.

  The proof of the second statement is similar, with the proviso that $\Loc^\fl_\Nis$ is replaced by $\Loc^\fl_\Aone$, and one concludes
  that the sheaf being represented it $\pi_{i-i'}^s( \Map_\pt(\Gm^{\wedge j}, \Loc^\fl_\Aone E)) \iso \pi_{i-i' + j\alpha}^{s, \Aone}(E)$.
\end{proof}

\begin{co} \label{co:homapsAsGlobalSections}
  For any $i , j \ge 0$, and any object $E$ of $\Spt(\Sm_k)$, taking global sections gives
  \[ \pi^s_i(E)(\pt) = [ \ssp S^i, E ] \]
  and
  \[ \pi_{i+j\alpha}^{s, \Aone}(E)(\pt)= [\ssp S^{i+j\alpha}, E]_\Aone .\]
\end{co}

\begin{pr} \label{pr:stableHoCommuteFilteredColimits}
  Suppose $\{E_n\}$ is a filtered system of objects in $\Spt(\Sm_k)$. Then the natural maps
  \[ \colim_n \pi_{i}^{s}(E_i) \to \pi_{i}^{s}( \colim_n E_i) \]
  and
  \[ \colim_n \pi_{i+j\alpha}^{s,\Aone}(E_i) \to \pi_{i+j\alpha}^{s, \Aone}( \colim_n E_i) \]
  are isomorphisms.
\end{pr}
\begin{proof}
  These follow from Corollary \ref{co:homotopysheafcolimit} and the observations that taking colimits commute and
  that colimits of spectra are calculated termwise.
\end{proof}

\subsection{Long Exact Sequences of Homotopy Sheaves}

We will use the term \textit{cofiber sequence} only in a limited sense: a cofiber sequence in a pointed model category
$\cat{M}$ is a sequence of maps $X \to Y \to Z$ such that $X \to Y$ is a cofibration of cofibrant spaces and $Z$ is a
categorical pushout of $\pt \leftarrow X \to Y$. A \textit{fiber sequence} is dual. 

The image of a cofiber sequence in $\ho\cat{M}$ may also be called a cofiber sequence, as in \cite[Chapter 6]{hovey1999}. The notion of \textit{fiber sequence} is dual.

The derived functors of left-Quillen functors preserve cofiber sequences, and dually the derived functors of
right-Quillen functors preserve fiber sequences.

Suppose $\ms$ is a model structure on $\Spaces_\pt$, obtained as a left Bousfield localization of the flasque- or
injective-local model structure. Consider $\Spt(\Sm_k)$, endowed with the stable model structure derived from $\ms$,
\cite[Section 3]{hovey2001}. By \cite[Theorem 3.9]{hovey2001}, the homotopy category $\ho_\ms \Spt(\Sm_k)$ is a triangulated category in the sense of \cite[Chapter
7.1]{hovey1999}. Write $\pi^{s,\ms}_{i+ j \alpha} (E)$ for the sheaf associated to the presheaf $U \mapsto [ S^i \wedge
\Gm^j \wedge U_+ , E]_{s \ms}$, where the set of maps is calculated in $\ho_\ms(\Spt(\Sm_k))$. Then we have the following
result as an immediate corollary of sheafifying Lemma 7.1.10 of \cite{hovey1999}.

\begin{pr} \label{pr:generalStableCofiberLES}
  If $X \to Y \to Z$ is a cofiber sequence in $\ho_\ms(\Spt(\Sm_k))$, then the induced sequence of homotopy sheaves
\[ \to \pi_{i+j\alpha}^{s,\ms}(X) \to   \pi_{i+j\alpha}^{s,\ms}(Y)  \to  \pi_{i+j\alpha}^{s,\ms}(Z) \to
\pi_{i-1+j\alpha}^{s,\ms}(X) \to  \]
is an exact sequence of sheaves of abelian groups.
\end{pr}

Examples include $\pi_i^s$, $\pi_{i+j\alpha}^{s, \Aone}$ as well as $\pi_i^{s,P}$ and $\pi_{i+j\alpha}^{s, P, \Aone}$ of
Section \ref{sec:localization}.

\subsection{\texorpdfstring{$\Aone$}{A1} Unstable and Stable}

We say that a spectrum $E$ is $\Aone$--$n$--connected if $\pi_i^{s, \Aone}(E) = 0$ for all $i \le n$. From the above
definition of $\pi_i^{s, \Aone}$, combined with  \cite[Theorem 6.38]{morel2012} saying that $\Laone$ does not decrease the connectivity of connected
objects, and that $\Laone$ commutes with $\Omega$ for simply-connected objects, we deduce the following lemma:

\begin{lm} \label{lem:sspPreserveA1conn}
  If $X$ is an $\Aone$--$n$--connected object of $\Spaces_\pt$, then $\ssp X$ is $\Aone$--$n$--connected.
\end{lm}

Recall that a map $f : X \to Y$ of connected objects of $\Spaces_\pt$ is said to be $n$--connected if the homotopy fiber
is $(n-1)$-connected, and $\Aone$--$n$-connected if the $\Aone$-homotopy fiber is $\Aone$--$(n-1)$-connected.

By use of \cite[Theorem 6.53, Lemma 6.54]{morel2012} and the $\Aone$--connectivity theorem, we deduce that if $X \to Y$ is
$n$-connected with $n \ge 1$ and if moreover $\pi_1(Y)$ is strongly $\Aone$ invariant, then $X \to Y$ is
$\Aone$--$n$-connected. These conditions hold when $X$ is simply connected, or when $n \geq 2$ and $X$ is $\Aone$ local.

The following result is due to Asok--Fasel, \cite{Asok_Fase_Euler_class13}

\begin{pr}[The Blakers--Massey Theorem of Asok--Fasel] \label{pr:bmtaf}
  Suppose $f: X \to Y$ is an $\Aone$--$n$-connected map of connected objects in $\Spaces_\pt$ and $X$ is $\Aone$--$m$-connected, with $n, m
  \ge 1$, then the morphism $\hofib_\Aone f \to \Omega \LAone \hocofib f$ is $m+n$-connected.
\end{pr}
\begin{proof}
  We rely on a homotopy excision result, a consequence of the Blakers--Massey theorem, that says that the result of this
  proposition holds in the setting of classical topology, \cite[VII Theorem 7.12]{whitehead2012}.

  We may replace $f: X \to Y$ by an equivalent $\Aone$-fibration of $\Aone$-fibrant objects without changing the $\Aone$
  homotopy type of $\hofib_\Aone f$ or of $\hocofib f$.

  The $\Aone$ homotopy fiber of $f$ therefore agrees with the ordinary fiber and therefore also with the
  simplicial homotopy fiber.

  The classical homotopy excision result, applied at points, now says that the map $\hofib f
  \to \Omega \hocofib f$ is simplicially $(m+n)$-connected. Since $m+n \ge 2$, and $\hofib f=\hofib_\Aone f$ is $\Aone$-local, it
  follows that $\pi_1(\Omega \hocofib f)$ is strongly $\Aone$-invariant and then by
  \cite[Theorem 6.56]{morel2012} it follows that
  \[ \hofib_\Aone f \weq \LAone \hofib_\Aone f \to \LAone \Omega \hocofib f \]
  is $(m+n)$-connected.

  The connectivity hypotheses imply that $\pi_1(Y) \iso
  \pi_1^\Aone(Y)$ is trivial, and therefore by the van Kampen theorem, that $\hocofib f$ is simply connected. This
  implies by \cite[Theorem 6.46]{morel2012} that $\LAone \Omega \hocofib f \weq \Omega \LAone \hocofib$. This completes
  the proof.
\end{proof}

\begin{co} \label{co:weakEquiv}
  Suppose $f: X \to Y$ is a map of $\Aone$ simply connected objects in $\Spaces_\pt$ such that the homotopy cofiber $\hocofib f$ is
  $\Aone$ contractible. Then $f$ is an $\Aone$ weak equivalence.
\end{co}
\begin{proof}
  We show by induction that $\hofib_\Aone f$ is arbitrarily highly connected. Since $X$ and $Y$ are simply connected,
  $\hofib_\Aone f$ is $0$-connected, so $f$ is $1$-connected.

  Suppose we know that $\hofib_\Aone f$ is $d$-connected, then applying Proposition \ref{pr:bmtaf} with $n = d+1$ and $m = 1$, we
  deduce that $\hofib_\Aone f \to \Omega \hocofib f \weq \pt$ is $\Aone$--$(d+2)$-connected, so that $\pi_{d+1}(
  \hofib_\Aone f)$ is trivial.
\end{proof}

\begin{co} \label{co:destabilization}
  Suppose $f: X \to Y$ is a map of $\Aone$ simply connected objects in $\Spaces_\pt$ such that $\ssp f: \ssp X \to \ssp Y$ is an
  $\Aone$-weak equivalence, then $f$ is an $\Aone$ weak equivalence.
\end{co}

\begin{proof}
  We may replace $f$ by a fibration of $\Aone$-fibrant objects.
  
  The map $f$ is necessarily $1$-connected, and from the proposition we deduce that
  $\pi_1(\hocofib f) \weq \pi_0(\hofib f)$, which is trivial. Since $\ssp$ is a left Quillen functor, it preserves
  cofiber sequences in the derived category, and we deduce that  $\ssp \hocofib f$ is $\Aone$ contractible. Since $\hocofib f$ is simply
  connected, the $\Aone$ Hurewicz theorem implies that $\hocofib f$ is $\Aone$ contractible.

  An appeal to Corollary \ref{co:weakEquiv} now completes the argument.
\end{proof}

\subsection{Points}

The site $\cat{Sh}_\Nis(\Sm_k)$ is well known to have enough points. Let $Q$ be a conservative set of points of
$\cat{Sh}_\Nis(\Sm_k)$. For each element $q \in Q$, there is an adjunction of categories
 \[ \xymatrix{ q^* : \cat{Sh}_\Nis(\Sm_k) \ar@<2pt>[r] & \ar@<2pt>[l]   \cat{Set} : q_* },\]
where $q^*$, as well as preserving all colimits, preserves finite limits.

There is a Quillen adjunction
 \[ \xymatrix{ q^* : \Spaces \ar@<2pt>[r] & \ar@<2pt>[l]   \sSet : q_* }\]
from the injective local model structure on $\Spaces$ to the usual model structure on $\sSet$. This extends in the
obvious way to the pointed model categories, and to the categories of spectra
 \[ \xymatrix{ q^* : \Spt(\Sm_k) \ar@<2pt>[r] & \ar@<2pt>[l]   \Spt : q_* }.\]

For an object $X$ of $\Spaces$, there is, by reference to \ref{pr:pi0asCoeq}, an isomorphism $q^* \pi_0(X) \iso
\pi_0(q^* X)$. It is also the case that $p^* \Omega^i (\Lnis X) \weq \Omega^i (\mathrm{Ex}^\infty p^* X)$. This gives us the following proposition
\begin{pr} \label{pr:pointsOfHomotopy}
  If $X$ is an object of $\Spaces_\pt$ and $i$ is a positive integer and $q$ a point of $\cat{Sh}_\Nis(\Sm_k)$, then there is an isomorphism of groups
  $\pi_i(|q^*X|) \iso q^* \pi_i(X)$.
\end{pr}

\begin{co}
  If $X$ is an object of $\Spt(\Sm_k)$, and if $i$ is an integer, then there is an isomorphism of abelian groups
  $\pi^s_i(q^* X) \iso q^*\pi_i^s(X)$.
\end{co}

These facts are special cases of results concerning $\infty$-topoi,
\cite[6.5.1.4]{lurie2009-a}. They are well-known, see for instance \cite[2.2 p14]{morel2005}, but seldom stated.

\section{Localization} \label{sec:localization}

Let $P$ denote a nonempty set of prime ideals of $\Z$, and $P' = \bigcap_{(p) \in P} \left( \Z \setminus (p) \right)$ the set of integers
not lying in any of these ideals. We write $\Z_P$ for the localization $(P')^{-1}\Z$, and $\Z_{(p)}$ in the case where $P = \{(p)\}$
consists of a single ideal. Following \cite{casacuberta1993}, where the following
is carried out in the category of CW complexes, we define $S^1_\tau = S^1$, a Kan complex equivalent to $\Delta^1 / \bd \Delta^1$. For any integer $n$, define $\rho_n : S_\tau^1 \to S_\tau^1$ to be a degree-$n$ self-map of $S^1$, and let $T_P$
denote the set of all such $\rho_n$ as $n$ ranges over $P'$.

The local injective and flasque model structures on $\sPre(\Sm_k)$ are cellular in the sense of Hirschhorn, \cite{hirschhorn2003}; a
proof for the injective case appears in \cite[Lemma 1.5]{hornbostel2006} and the flasque case is treated in \cite{isaksen2005}. We may therefore apply the general machinery of \cite{hirschhorn2003}
and left-Bousfield-localize $\Spaces$ at the set $T_P$. We call the resulting model structures
$P$--local, and if $P=\{(p)\}$ we call the resulting model structures $p$--local. Write $\Loc_P$ for the functorial fibrant
replacement functor in each model category. In the case where $P = \{(p)\}$, we may write $\Loc_{(p)}$.

The localization of the usual model structure on $\sSet$ with respect to the set $T_P$ of maps is a form of $P$--local
model structure on $\sSet$, we refer the reader to \cite{casacuberta1993}, especially \cite[Section 8]{casacuberta1993},
for the comparisons between different $P$--localizations in classical topology and for a discussion of non-nilpotent
objects. For nilpotent simplicial sets, the various $P$--localization functors agree up to weak equivalence, see
\cite[Proposition 8.1]{casacuberta1993}. In particular, they agree for simply connected spaces, and by extension to
simply connected simplicial presheaves. For instance, in the case of simply connected simplicial presheaves, the
$P$-localization defined here agrees with $\mathrm{H}_\ast( \cdot, \ZZ_P)$-localization.

\begin{lm} \label{lm:PointLocalization}
  With notation as above, if $s$ is a point of $\cat{Sh}_\Nis(\Sm_k)$, the adjunctions
  \[ \xymatrix{ s^* : \Spaces \ar@<2pt>[r] & \ar@<2pt>[l]   \sSet : s_* }\]
and
 \[ \xymatrix{ s^* : \Spaces_\pt \ar@<2pt>[r] & \ar@<2pt>[l]   \sSet_\pt : s_* }\]
   are monoidal Quillen adjunctions between the $P$--local model categories, where $\Spaces$ and $\Spaces_\pt$ may be
   given either the flasque or the injective model structure.
\end{lm}
\begin{proof}
  It is sufficient to prove the unpointed cases, the pointed follow immediately. The proofs in the flasque and injective
  cases are the same.

  Following \cite[Theorem 3.3.20]{hirschhorn2003}, the adjoint pair
  \[ \xymatrix{ s^* : \Spaces \ar@<2pt>[r] & \ar@<2pt>[l]   \sSet : s_* }\]
  is a Quillen adjunction between the $P$--local model structure on the left and the model structure on $\sSet$ obtained
  by localization at the set of maps 
  \[s^* (\rho_n^k \times \id_U) : s^* (S^k_\tau \times U) \to s^*(S^k_\tau \times U)\]
  where $\rho_n^k \in T_P$. Denote this set of maps by $s^*T_P'$. It will suffice to show that localization of $\sSet$
  at $s^*T_P'$ agrees with localization of $\sSet$ at $T_P$.

Since evaluation at $s^*$ commutes with fiber products, the maps of $s^*T_P'$ maps are of the form $\rho_n^k
  \times \id_{s^* U}$, and setting $U = \pt$, we see that $T_P \subset s^*T_P'$. The maps of $s^*T_P'$ are, moreover,
  weak equivalences in the localization of $\sSet$ at $T_P$. It follows that the localization of $\sSet$ at $s^*T_P'$ is
  simply the ordinary $P$--localization of $\sSet$.  
\end{proof}

We note in addition that the
model categories appearing above are simplicial model categories, and the adjunctions appearing are adjunctions of simplicial model
categories in the sense of \cite[Chapter 4.2]{hovey1999}.

We continue to work principally in the injective local not-localized-at-$P$ model structures, but write $ A \weq_P B$ to
indicate that $A$ is weakly equivalent to $B$ in the $P$--local structure, or equivalently that $L_P A \weq L_P B$. The notation $A \weq_{(p)}
B$ will be used where appropriate. We will use the flasque model structures only when dealing with spectra. 

In this section we will occasionally write groups $\pi_i(X)$ in multiplicative notation even when the groups are
abelian. The $n$--th power map of a group $G$ will be the map $x \mapsto x^n$, which is necessarily a homomorphism if $G$ is
abelian, and is preserved by group homomorphisms in any case. If $P$ is a set of primes, then a group $G$ is said to be
\textit{$P$--local } if the $n$--th power map is a bijection on $G$ whenever $n$ is not divisible by any of the primes
in $P$. We will say that a presheaf of groups is $P$--local if all groups of sections are $P$--local, and a sheaf of
groups is $P$--local if the appropriate $n$--th power maps are isomorphisms of sheaves of sets.

\begin{pr} \label{pr:LocalizationLocalizesHomotopySheaves}
  If $X$ is a connected object of $\Spaces_\pt$, and $P$ is a set of primes, then the sheaves $\pi_i(\Loc_P X)$ are
  $P$--local sheaves of groups.
\end{pr}
\begin{proof}
  It suffices to show that the presheaves
  \[ U \mapsto \pi_i(\Loc_P X(U)) \]
  are $P$--local, the result for the associated sheaves is then an exercise in sheafification.

  Let $n$ be an integer not divisible by any of the primes of $P$, let $i \ge 1$, and let $U$ be an object of
  $\Sm_k$. We wish to show that the $n$--th power map on $\pi_i(\Loc_P X(U))$ is a bijection, but this is the map
  induced by $\rho_n^i \times \id_U$ on $\pi_0\left(\sMap_\pt(S^i_\tau \vee U_+ , \Loc_P X)\right)$. Since $\Loc_P X$ is
  $P$--local and $\rho_n^i \times \id_U$ is in $T_P'$, this map is a bijection.  
\end{proof}

\begin{lm} \label{lm:pointsPreserveLoc}
  Let $X$ be an object of $\Spaces$, let $s$ be a point of $\cat{Sh}_\Nis(\Sm_k)$ and let $P$ be a set of primes. Then $s^* \Loc_P X \weq
  \Loc_P s^* X$.
\end{lm}
\begin{proof}
  We first claim that $s^* \Loc_P X$ is $P$--local. Since it is fibrant, it suffices to show that if $\rho_n^k$ is an
  element of $T_P$, then the induced map
  \[ (\rho_n^k)_* : \sMap( S^k_\tau , s^* \Loc_P X) \to \sMap(S^k_\tau , s^* \Loc_P X) \]
  is a weak equivalence. 
  If $\{U_i\}$ is a system of neighbourhoods for $s^*$ then there is a succession of natural isomorphisms
  \begin{align*}
    \sMap(S^k_\tau , s^* \Loc_P X) & \iso \sMap( S^k_\tau, \colim_U (\Loc_P X)(U) ) \\ 
    & \iso \colim_U \sMap(S^k_\tau, (\Loc_P X)(U)) \qquad \text{since $S^k_\tau$ is compact,} \\
    & \iso \colim_U \sMap(S^k_\tau \times U, \Loc_P X) 
  \end{align*}
  and $\rho_n^k$ induces a weak equivalence on the spaces $\sMap(S^k_\tau \times U, \Loc_P X)$ since $\Loc_P X$ is $P$-local.

  The functor $s^*$ preserves trivial cofibrations, and therefore the map $s^* X \to s^* \Loc_P X$ is a trivial
  cofibration the target of which is fibrant in the $P$-local model structure on $\sSet$. Therefore $s^* \Loc_P X$ is
  weakly equivalent in the ordinary model structure on $\sSet$ to any other $P$-fibrant-replacement for $s^* X$,
  notably to $\Loc_P s^* X$, which is what was claimed.
\end{proof}

\begin{pr} \label{pr:LocalityAtPoints}
  Let $X$ be a fibrant object of $\Spaces$, let $S$ be a conservative set of points of $\cat{Sh}_\Nis(\Sm_k)$ and let $P$ be a set of primes. Then
  $X$ is $P$-local if and only if $s^*X$ is $P$-local for all $s^* \in S$.
\end{pr}
\begin{proof}

  The space $X$ is $P$-local if and only if $X$ is fibrant and $X \to \Loc_P
  X$ is a local weak equivalence. This is the case if and only if $s^* X
  \to s^* \Loc_P X$ is a weak equivalence for all $s^* \in S$, which, by Lemma \ref{lm:pointsPreserveLoc}, is the case
  if and only if $s^* X \to \Loc_P s^* X$ is a weak equivalence for all $s^* \in S$, and since $s^* X$ is fibrant, this
  is the same as saying that $s^* X$ is $P$-local in $\sSet$.
\end{proof}

\begin{df}
  An object $X$ of $\Spaces_\pt$ is said to be \textit{simple} if the action of $\pi_1(X)$ on $\pi_i(X)$ is trivial for
  all $i \ge 1$.
\end{df}

In particular, if $X$ is simple then the sheaf $\pi_1(X)$ is a sheaf of abelian groups which acts trivially on
$\pi_i(X)$ for all $i \ge 2$. A simply-connected object is simple, as is an object with an $H$-space structure.

\begin{pr} \label{pr:localizationOnHomotopySheaves}
  Let $X$ be a connected, simple object of $\Spaces_\pt$, then the natural maps $\Z_P \tensor_\Z \pi_i(X) \to \pi_i(\Loc_PX)$ are isomorphisms.
\end{pr}
\begin{proof}
  Fix a point $s$ of $\cat{Sh}_\Nis(\Sm_k)$. By Lemma \ref{lm:pointsPreserveLoc}, there are isomorphisms
 \[ s^* \pi_i(\Loc_PX) \iso \pi_i(s^* \Loc_P X) \iso \pi_i (\Loc_P s^* X). \] 
 As remarked above for the case of simply-connected spaces, \cite[Proposition 8.1]{casacuberta1993} implies that the $P$-localization of simplicial sets considered here agrees with the $P$-localization of \cite[Section V]{bousfield1972} in the case of simple simplicial sets. By reference to  \cite[V.4.1 \& V.4.2]{bousfield1972}, the group $\pi_i (\Loc_P s^* X)$ is isomorphic to
 \[ \pi_i(s^* X) \tensor_\Z \Z_P \iso s^*( \pi_i(X) \tensor_\Z \Z_P) \]
 which proves the proposition.
\end{proof}

\begin{lm} \label{lm:locCommuteSusLoop}
  Suppose $X$ is a simply connected object of $\Spaces$ and $P$ a set of prime numbers, then $\Loc_P (S^1 \wedge X) \weq S^1
  \wedge \Loc_P X$ and $\Omega \Loc_P X \weq \Loc_P \Omega \Loc_\Nis X$.
\end{lm}
\begin{proof}
  For a pointed simplicial set $X$, there is a map $S^1 \wedge X \to S^1 \wedge \Loc_P X$ which induces $P$--localization
  on homology, and therefore there is a weak equivalence $\Loc_P (S^1 \wedge X) \weq S^1 \wedge \Loc_P X$. This is
  promoted to the setting of simply connected objects in $\Spaces_+$ by arguing at points.

  A similar argument applies to $\Omega X$ using homotopy in place of homology.
\end{proof}

\begin{pr} \label{pr:PlocalProduct}
  If $X$ and $Y$ are objects in $\Spaces_\pt$ and $P$ is a set of primes, then $\Loc_P(X \times Y)
  \weq \Loc_P X \times \Loc_P Y$.
\end{pr}
\begin{proof}
  The object $\Loc_P X \times \Loc_P Y$ is $P$-locally weakly equivalent to $X \times Y$, and therefore to $\Loc_P( X
  \times Y)$, by Corollary~\ref{smash_preserve_ms_we}. Since $\Loc_P X \times \Loc_P Y$ is $P$-locally fibrant, the result follows.
\end{proof}

\begin{pr} \label{pr:PlocalConnectedness}
 Suppose $X$ is an object of $\Spaces$ and $P$ is a set of prime numbers, then $X \to \Loc_P X$ induces an isomorphism
 on $\pi_0$.
\end{pr}
\begin{proof}
  For any simplicial set $K$, the set of maps 
  \[ \sMap( S^1 \times K, S^0 ) \to \sMap( S^1 \times K ,S^0) \]
  induced by multiplication by $n$ in $S^1$ is a bijection, the simplicial set $\sMap( S^1 \times K ,S^0)$ depending
  only on the components of $K$.

  Then the map $X \to \Loc_P X$ induces an equivalence of mapping objects $\Map( \Loc_P X, S^0)  \to \Map( X, S^0)$,
  from which the result follows.
\end{proof}

\subsection{\texorpdfstring{$P$}{P} and \texorpdfstring{$\Aone$}{A1} Localization}

\begin{pr} \label{pr:PAone}
  If $X$ is a connected object of $\Spaces$ such that $X$ is $\Aone$--local and $\pi_1(X)$ is abelian, then $\Loc_P X$
  is again $\Aone$ local.
\end{pr}
\begin{proof}

  By \ref{pr:PlocalConnectedness}, $\Loc_P X$ is again connected.

Under the hypotheses, it suffices to check that the sheaves $\pi_i(X) \tensor_\ZZ \ZZ_P$ are strictly
  $\Aone$--invariant, \cite[Chapter 6]{morel2012}, but this follows immediately since the functor $\cdot \tensor_\ZZ
  \ZZ_P$ is exact.
\end{proof}

In the sequel, we consider only the composite localization $\Loc_P \LAone X$, and not the reverse. The proposition says
that, under connectivity hypotheses, $\Loc_P \LAone X$ is both $\Aone$ and $P$--local.

If $X$ is a connected $H$--space in $\Spaces_\pt$, then it is possible to define self maps
\[ \times n : \xymatrix{ X \ar^{\Delta}[r] & X^{n} \ar^{\mu(\mu( \dots \mu))}[r] & X } \]
by composing the $n$--fold diagonal and an iterated multiplication map. The map $\times n$ represents a class in $[X,
X]$, which we also denote $\times n$ in an abuse of notation.

\begin{pr} \label{pr:HPLocal}
  If $X$ is a connected $H$--space in $\Spaces_\pt$ and $P$ is a set of primes, then $\Loc_P X$ is again a connected
  $H$--space, and the map $X \to \Loc_P X$ is a weak equivalence if and only if $\times n \in [X , X]$ is invertible for
  all $n$ not divisible by the primes of $P$.
\end{pr}
\begin{proof}
  The object $\Loc_P X$ carries an $H$--space structure since $\Loc_P(X \times X)\weq \Loc_P X \times \Loc_P X$, see
  Proposition \ref{pr:PlocalProduct}.

  An Eckmann--Hilton argument implies that $X$ is simple, that is the action of $\pi_1(X)$ on $\pi_i(X)$ is trivial for
  all $i$, and moreover $\times n$ induces multiplication by $n$ on all homotopy sheaves. The result follows.
\end{proof}

\begin{pr}
 Suppose $X$ is a connected object of $\Spaces_\pt$, and further that $X$ is equipped with an $H$--space structure. Then
 $\LAone \Loc_P X \weq \Loc_P \LAone X$, where the localizations are carried out with respect to either the local or the
 flasque model structure on $\Spaces_\pt$.
\end{pr}
\begin{proof}
  We give the proof in the local case, the flasque is the same \textit{mutatis mutandis}.
  
  Starting with the Quillen adjunction from the injective local model structure on $\Spaces_\pt$ to the $\Aone$ local,
  we obtain a commutative diagram of model structures, where the maps indicated are left Quillen adjoints:
  \begin{equation} \label{eq:brf} \xymatrix{ \text{ Local} \ar[r] \ar[d] & \text{$\Aone$} \ar[d] \\ \text{$P$-local} \ar[r] & \text{$P$-$\Aone$} } \end{equation}
  where the $P$-$\Aone$ model structure is the $P$-localization in the evident sense of the $\Aone$ model structure.

  We claim that for a connected $H$--space object of $\Spaces_\pt$, the maps $X \to \LAone \Loc_P X$ and $X \to \Loc_P
  \LAone X$ are both fibrant replacements in the $P$-$\Aone$--model structure, and therefore that $\LAone \Loc_P X \weq
  \Loc_P \LAone X$ in the original model structure.
  
  The lynchpin of the following argument is the observation, by reference to \cite[Proposition 3.4.1]{hirschhorn2003}, an
  object $W$ of $\Spaces$ is $P$-$\Aone$--local if it satisfies the following three conditions:
  \begin{enumerate}
  \item $W$ is fibrant in the injective model structure on $\Spaces$.
  \item For any object $U$ of $\Sm_k$, the maps
    \[ \sMap(U , W) \to \sMap(U \times \Aone, W) \]
    of simplicial mapping objects are weak equivalences.
  \item For any $\rho_n^k$ where $k \ge 1$ and $n$ is not divisible by a prime in $P$, the maps induced by $\rho_n^k$ 
    \[ \sMap( S^k_\tau , W) \to \sMap( S^k_\tau, W). \] 
  \end{enumerate}

  The object $\Loc_P \LAone X$ is both $\Aone$--fibrant, by \ref{pr:PAone}, and $P$-locally fibrant, and it is therefore a $P$--local object in the $\Aone$ model structure. By reference
  to \cite[Proposition 3.4.1]{hirschhorn2003}, it is fibrant in the $P$-$\Aone$--model structure. Since $X \to \LAone X$ is
  an $\Aone$ weak equivalence, it is \textit{a fortiori} a $P$-$\Aone$--weak equivalence, and therefore $X \to \LAone X
  \to \Loc_P \LAone X$ is a $P$-$\Aone$--weak equivalence, and therefore a fibrant replacement.
  
In the other case, we argue similarly. The object $\LAone \Loc_P X$ is $\Aone$-fibrant. Moreover, $\Loc_P X$ is an
$H$-space for which $\times n \in [ \Loc_P X , \Loc_P X]$ is an isomorphism. Since the natural transformation $\id \to
\LAone$ induces a morphism of $H$-spaces, it follows that $\times n \in [ \LAone \Loc_P X, \LAone \Loc_P X]$ is
invertible as well, so by Proposition \ref{pr:HPLocal}, it follows $\LAone \Loc_P X$ is $P$-fibrant. Moreover $X \to \Loc_P X$ is a $P$-local weak equivalence, and therefore a $P$-$\Aone$ weak
  equivalence, and consequently $X \to \Loc_P X \to \LAone \Loc_P X$ is a $P$-$\Aone$ fibrant replacement.
\end{proof}

We recall from \cite[Chapter 2]{morel2012}, that there is a construction on presheaves of groups, $\sh{G}$, given by
\[ \sh{G}_{-1} : U \mapsto \ker( \sh{G} (\Gm \times X) \overset{\text{ev}(1)}\longrightarrow \sh{G}(X) )\] where
$\text{ev}(1)$ is evaluation at $1$ in $\Gm$. Equivalently, $\sh{G}_{-1}$ is the kernel of the map of group sheaves
$\Map(\Gm, \sh{G}) \to \Map(\pt, \sh{G}) \iso \sh{G}$. The assignation $\sh{G} \mapsto \sh{G}_{-1}$ is functorial, and
sends sheaves to sheaves. The $j$-fold iterate of the `$-1$' functor applied to $\sh{G}$ is denoted by $\sh{G}_{-j}$.

The result \cite[Theorem 6.13]{morel2012} says that if $X$ is a connected object of $\Spaces_\pt$, then
$\pi_{i+j\alpha}^\Aone(X) = \pi_i^\Aone(X)_{-j}$. Recall that $\pi_{i+j\alpha}^\Aone(X)$ is notation for
$\pi_i(\Map(\Gm^{\wedge j}, \Laone X))$.

\begin{pr} \label{pr:minusOne}
  If $\sh{G}$ is an abelian sheaf of groups, then $\sh{G}_{-1}$ is also abelian and there is a natural isomorphism $(R \tensor \sh{G})_{-1} \iso R \tensor \sh{G}_{-1}$.
\end{pr}
\begin{proof}
  The abelian property of $\sh{G}_{-1}$ follows immediately from the definition.

  For any object $U$ of $\Sm_k$, we have a natural commutative diagram of left-exact sequences
  \[ \xymatrix{ 1 \ar[r] & (R \tensor \sh{G} )_{-1} (U) \ar[r] \ar^f[d] &  (R \tensor \sh{G})(\Gm \times U) \ar[r]
    \ar@{=}[d] & (R \tensor \sh{G})(U) \ar@{=}[d]  \\
    1 \ar[r] & R \tensor \sh{G}_{-1}(U) \ar[r] & R \tensor (\sh{G} (\Gm \times U)) \ar[r] & R \tensor (\sh{G}(U)), } \]
  from which the natural isomorphism $(R \tensor \sh{G})_{-1} \iso R \tensor \sh{G}_{-1}$ follows.
\end{proof}

\begin{pr}\label{pr:loopGmLocal}
  If $j$ is a nonnegative integer, $X$ is a simply connected object of $\Spaces_\pt$,  and $P$ is a set of primes, then there
  is a natural isomorphism $\Map_{\pt}(\Gm^{\wedge j}, \Loc_P \Laone X) \to \Loc_P \Map_\pt(\Gm^{\wedge j}, \Laone X)$
  in $\ho_\Nis \Spaces_\pt$.
\end{pr}
\begin{proof}
  Each of the two spaces in question is equipped with a natural map to $\Loc_P \Map_\pt(\Gm^{\wedge j}, \Loc_P \Laone
  X)$. It suffices to show that each of these maps is a simplicial weak equivalence.

  By Proposition \ref{pr:PAone}, the space $\Loc_P \LAone X$ is $\Aone$-local. By the unstable $\Aone$ connectivity
  theorem, \cite[Theorem 6.38]{morel2012}, it is also connected. As is shown in the proof of \cite[Theorem
  6.13]{morel2012}, the functor $\Map_\pt(\Gm^{\wedge j}, \cdot)$, preserves the subcategory of connected, $\Aone$-local
  objects in $\Spaces_\pt$. 

  Let $Y$ denote either $\LAone X$ or $\Loc_P \LAone X$, both of which are $\Aone$-local and connected. Then, for any $
  i \ge 1$, the homotopy sheaf $\pi_i(\Map_\pt(\Gm^{\wedge j}, \Loc_P Y))$ is naturally isomorphic to
  \[ \pi_i(\Loc_P  Y)_{-j} \iso \pi_i^{P, \Aone}(Y)_{-j} \iso \pi_i^{\Aone}(Y)_{-j} \tensor_\Z \Z_P \iso
  \pi_i(\Map_\pt(\Gm^{\wedge j}, Y)) \tensor_\Z \Z_P \iso \pi_i( \Loc_P \Map_\pt(\Gm^{\wedge j}, Y)). \]
  In particular the spaces $\Map_\pt( \Gm^{\wedge j}, \Loc_P \LAone X)$ and $\Loc_P \Map_\pt(\Gm^{\wedge j}, \LAone X)$
  are both weakly equivalent to the space $\Loc_P \Map_\pt(\Gm^{\wedge j}, \Loc_P \LAone X)$, as required.
\end{proof}

\begin{df} \label{df:PA}
  For an object $X$ of $\Spaces_\pt$, and nonnegative integers $i$ and $j$, the notation $\pi_{i + j\alpha}^{P, \Aone}
  (X)$ is used to denote $\pi_i (\Map_\pt (\Gm^{\wedge j}, \Loc_P \LAone X))$.
\end{df}

\begin{pr} \label{pr:pAoneGmHsheaves}
  If $i$, $j$ are nonnegative integers, $\Laone X$ is a simply connected, $\Aone$ local object of $\Spaces_\pt$,  and
  $P$ a set of primes, then there are natural isomorphisms
  \[ \pi_{i+j\alpha}^{P, \Aone}(X) \iso \pi_{i}^{P, \Aone}(X)_{-j} \iso \pi_{i}^\Aone(X)_{-j} \tensor_\Z \Z_P \iso \pi_{i+ j \alpha}^{\Aone}(X) \tensor_\ZZ \ZZ_P. \]
\end{pr}
\begin{proof}
  The sheaf $\pi_{i+j\alpha}^{P, \Aone}(X)$ is isomorphic to $\pi_{i+j \alpha}^\Aone(X) \tensor_\Z \Z_P$ by Proposition
  \ref{pr:loopGmLocal}. This is isomorphic to $\pi_{i}^\Aone(X)_{-j} \tensor_\Z \Z_P \iso \pi_i^{P, \Aone}(X)_{-j}$ as required.
\end{proof}

\begin{pr} \label{pr:lesPAone}
  If $P$ is a set of primes and if $X$, $Y$ and $Z$ are simply connected objects of $\Spaces_\pt$
  such that $X \to Y \to Z$ is a $P$-$\Aone$--fiber sequence up to homotopy, and if $j$ is a nonnegative integer, then
  there is a natural long exact sequence
  \[ \dots \to \pi_{i + j\alpha}^{P, \Aone} (X) \to \pi_{i+j\alpha}^{P, \Aone} (Y) \to \pi_{i + j \alpha}^{P, \Aone}(Z)
    \to \pi_{i-1 + j\alpha}^{P, \Aone}(X) \to \dots \]
\end{pr}

\begin{proof}
  This is an immediate consequence of Definition \ref{df:PA}.
\end{proof}

\subsection{\texorpdfstring{$P$}{P} Localization of Spectra}

Throughout this section, the underlying model structure on $\Spaces_\pt$ is taken to be the flasque,
rather than the injective.

One can construct a $P$--local model category of presheaves of spectra, following \cite{hovey2001}, as the $S^1$--stable model
category on the $P$--local flasque model structure on $\Spaces_\pt$.

\begin{lm} \label{lm:StableLocalization}
  The adjunction
  \[ \xymatrix{ \Sigma^{\infty} : \Spaces_\pt \ar@<2pt>[r] & \ar@<2pt>[l]   \Spt(\Sm_k) : \operatorname{Ev}_0 }\]
   is a Quillen adjunction between the $P$--local model categories.
\end{lm}
\begin{proof}
  This is implicit in \cite{hovey2001}, being the combination of Proposition 1.16 and the definition
  of the stable model structure as a localization of the level model structure on spectra.
\end{proof}

Explicitly, the fibrant-replacement functor, $\Loc_P$, in $\Spt(\Sm_k)$ with the $P$--local model structure is given by
\[ (\Loc_P E)_i  = \colim_{k \to \infty} \Map_\pt(S^k, \Loc_P E_{i+k}). \]
With the $P$--local model structures and the smash product, the category $\Spt(\Sm_k)$ is a
$\Spaces_\pt$--model category, in the sense of \cite[Chapter 4.2]{hovey1999}

\begin{lm} \label{lm:PointSpectrumLocalization} If $s$ is a point of $\cat{Sh}_\Nis(\Sm_k)$, the adjunction
  \[ \xymatrix{ s^* : \cat{Spt}(\Sm_k) \ar@<2pt>[r] & \ar@<2pt>[l] \cat{Spt} : s_* }\] is a Quillen adjunction between
  the $P$--local model categories.
\end{lm}

\begin{co}\label{co:sspLoc_P=Loc_Pssp}
  For any object $X$ of $\Spaces_\pt$, and any set $P$ of primes, there is a stable weak equivalence:
  \[ \ssp \Loc_P X \to \Loc_P \ssp X. \]
\end{co}

A spectrum $E$ is said to be \textit{$P$--local} if it is fibrant and the map $E \to \Loc_P E$ is a stable weak
equivalence. Since it is possible to check stable weak equivalence of spectra at points, we deduce the following by
arguing at points.

\begin{pr} \label{pr:nDivisibilitySpectra}
  A spectrum $E$ is $P$--local if and only if it is fibrant and the maps
\[ E \overset{n}{\to} E \]
are weak equivalences for all $n$ not divisible by the primes in $P$.
\end{pr}

\begin{pr}
  A spectrum $E$ is $P$--local if and only if it is fibrant and the localization maps $\pi^s_i(E) \to \pi_i^s(E)
  \tensor_\ZZ \ZZ_P$ are isomorphisms for all $i$.
\end{pr}

\subsection{\texorpdfstring{$P$}{P}-- and \texorpdfstring{$\Aone$}{A1}--Localization of Spectra} 

We begin with a commutative diagram of model structures on the category $\Spt(\Sm_k)$, which is the application of
\cite{hovey2001} to the flasque version of diagram \eqref{eq:brf}:
\[ \xymatrix{ \text{Stable} \ar[r] \ar[d] & \Aone \ar[d] \\ \text{$P$--Stable} \ar[r] & \text{$P$-$\Aone$--Stable}. } \]
Fibrant replacements in the $\Aone$ or $P$-local model structures are effected by replacing $E$ by the
spectrum that has level $i$ given by
\[ \colim_k \Omega^k \LAone E_{i+k} \] 
or
\[ \colim_k \Omega^k \Loc_P E_{i+k} \]
respectively, $\Loc_P$ and $\LAone$ being taken in the flasque model structures. The stable fibrant replacements are
denoted $\Loc_P E$ and $\LAone E$. 

\begin{lm}
  The classes of $P$--locally flasque fibrant and $P$-$\Aone$--locally flasque fibrant objects in $\Spaces$ and
  $\Spaces_\pt$ are closed under filtered colimits.
\end{lm}
\begin{proof}
  Since an object is $P$-$\Aone$--locally flasque fibrant if and only if it is both $P$- and $\Aone$--locally flasque
  fibrant, it suffices to prove the case of $P$--locally flasque fibrant objects. 

  Suppose $X_k$ is a filtered diagram of $P$--locally flasque fibrant objects, then $\colim_k X_k$ is flasque fibrant,
  by \cite{isaksen2005}. We wish to show that for any $\rho_n^n\times \id_U: S^k_\tau \times U \to S^k_\tau \times U$ in
  $T_P$, the induced map
  \[ \sMap( S^k_\tau \times U , \colim_k X_k) \to \sMap( S^k_\tau \times U , \colim_k X_k) \]
  is a weak equivalence. Since $S^k_\tau \times U$ is equivalent to a compact object, $S^k \times U$, of $\Spaces_\pt$,
  and since the $X_k$ and $\colim_k X_k$ are all fibrant, and $S^k_\tau \times U$ and $S^k \times U$ are all cofibrant,
  the given map, we wish to show that the induced map
  \[\colim_k \sMap( S^k_\tau \times U , X_k) \to \colim_k \sMap( S^k_\tau \times U , X_k) \]
 is a weak equivalence of simplicial sets, but since the $X_k$ are themselves $P$--locally fibrant, this is immediate.
 \end{proof}

\begin{pr}\label{pr:LPLaone=LaoneLP}
  For any object $E$ of $\Spt(\Sm_k)$ there is a stable weak equivalence $\Loc_P \LAone E \weq \LAone \Loc_P E$.
\end{pr}
\begin{proof}
  The objects in question are levelwise fibrant for the flasque model structure. It suffices therefore to show that they
  are levelwise weakly equivalent for the flasque model structure. Since we are working the flasque model structure,
  filtered colimits of fibrant objects are again fibrant, and so we deduce the existence of weak equivalences
  \[ \LAone \colim_k  X_k \weq \LAone \colim_k \LAone X_k \weq \colim_k \LAone X_k \]
  and similarly for $\Loc_P$.
  
  We may assume that the spaces $E_i$ appearing are all simply-connected $H$--spaces, and therefore $\Loc_P \LAone E_i
  \weq \LAone \Loc_P E_i$. We then have
  \[ \colim_k \Omega^k \LAone \left( \colim_{k'} \Omega^{k'} \Loc_P E_{k + k' + i} \right) \weq \colim_k \Omega^k
  \left( \colim_{k'} \Omega^{k'} \LAone \Loc_P E_{k+k'+i} \right)\] 
  which is symmetric in $\LAone$, $\Loc_P$, up to weak equivalence, whence the result.  
\end{proof}

We therefore conflate $\Loc_P \LAone E$ and $\LAone \Loc_P E$, calling either the
$P$-$\Aone$--localization of $E$. We say that a map of spectra $f:E \to E'$ is a $P$-$\Aone$--weak equivalence if
$\Loc_P \LAone f$ is a stable weak equivalence of spectra, or equivalently if $\Loc_P f$ is an $\Aone$ weak equivalence
of spectra, or equivalently again if $\LAone f$ is a $P$-local equivalence of spectra. We write $\pi_i^{P, \Aone, s}(E)$
for the homotopy sheaves $\pi_i^s(\Loc_P \LAone E)$.

\begin{pr} \label{pr:PAoneLocalHomotopySheaves}
  If $E$ is an object in $\Spt(\Sm_k)$, and $P$ is a set of prime numbers then there is a natural isomorphism
 \[ \pi_i^{s, P, \Aone}(E) \iso \pi_i^{s, \Aone}(E) \tensor_\ZZ \ZZ_P.\]
\end{pr}
\begin{proof}
  Immediate from the above.
\end{proof}

\begin{pr}\label{f_s-P-Aone_weak_equiv_iff_pisAoneotimesP_iso}
  If $f: E \to E'$ is a map in $\Spt(\Sm_k)$ and $P$ is a set of prime numbers, then $f$ is a $P$-$\Aone$ weak
  equivalence if and only if $\pi_i^{s, \mathbb{A}^1}(f) \tensor_\ZZ \ZZ_P$ is an isomorphism of abelian groups for all $i$.
\end{pr}
\begin{proof}
  Immediate from the above.
\end{proof}

\begin{pr} \label{pr:PlocalHoCommutesWithFilteredColimits}
  Suppose $\{E_n\} $ is a filtered system of objects in $\Spt(\Sm_k)$ and $P$ is a set of prime numbers, then the
  natural maps 
  \[ \colim_n \pi_i^{s, P}(E_n) \to \pi^{s, P}_i(\colim_n E_n) \]
  and
  \[ \colim_n \pi_{i+j\alpha}^{s, P, \Aone}(E_n) \to \pi_{i+j\alpha}^{s, P, \Aone}(\colim E_n) \]
  are isomorphisms
\end{pr}
\begin{proof}
  Immediate from the above and Proposition \ref{pr:stableHoCommuteFilteredColimits}.
\end{proof}

\begin{pr}
  Suppose $f: X \to Y$ is a map of simply-connected objects $\Spaces_\pt$ such that $\ssp f$ is a $P$-$\Aone$--stable weak
  equivalence. Then $f$ is a $P$-$\Aone$ weak equivalence.
\end{pr}
\begin{proof}

  The map  $\Loc_P \LAone f: \Loc_P \LAone \ssp X \to \Loc_P \LAone \ssp Y$ is a weak equivalence,
  and this map agrees in the stable homotopy category with $\Loc_P \LAone \ssp \LAone X \to \Loc_P
  \LAone \ssp \LAone Y$. We may commute $\Loc_P$ past $\LAone$ and past $\ssp$, so that
  we conclude that $\LAone \ssp \Loc_P \LAone X \to \LAone \ssp \Loc_P \LAone Y$ is a weak
  equivalence. By Corollary \ref{co:destabilization}, since
  $\Loc_P \LAone X$ and $\Loc_P \LAone Y$ are simply connected, we deduce that $\LAone \Loc_P \LAone
  X \to \LAone \Loc_P \LAone Y$ is a weak equivalence, and since $\Loc_P \LAone X$, $\Loc_P \LAone
  Y$ are already $\Aone$--local by Proposition \ref{pr:PAone}, the result follows.
\end{proof}

\section{The Grothendieck--Witt Group}\label{section:GW}

\subsection{The homotopy of spheres}\label{subsection_homotopy_spheres}

Consider a motivic sphere $X= S^{n+q\alpha} = S^n \wedge \Gm^{\wedge q}$. 

We make frequent use of the following result, which is a paraphrase of some results of \cite[Section 6.3]{morel2012}:
\begin{lemma}[Morel] \label{le:0stem}
  If $(n,q)$ and $(n',q')$ are pairs of nonnegative integers, and if $n\ge 2$, then
\[ \pi_{n+q\alpha}^{\Aone}(S^{n'+q'\alpha}) = \begin{cases} 0 & \quad \text{ if $n < n'$;} \\
  \KMW_{q'-q} & \quad \text{if $n=n'$ and $q' > 0$}; \\
  0 & \quad \text{ if $n = n'$, $q' = 0$ and $q>0$}; \\
  \Z & \quad \text{ if $n=n'$ and $q=q'=0$.} \end{cases} \]
\end{lemma}

The stable version of this result was known earlier, but may be deduced from the unstable. 
\begin{co}\label{co:mapsSphereSpectra}
  If $(n,q)$ and $(n',q')$ are pairs of integers with $q, q'$ nonnegative, then the sets of maps between $S^{n+q\alpha}$
  and $S^{n'+q' \alpha}$ in the $\Aone$ homotopy category of $S^1$--spectra take the form
  \[ \pi_{n+q\alpha}^{s, \Aone}(S^{n' + q' \alpha}) = \begin{cases} 0 & \quad \text{ if $n < n'$;} \\
  \KMW_{q'-q} & \quad \text{if $n=n'$ and $q' > 0$} \\
  0 & \quad \text{ if $n = n'$, $q' = 0$ and $q>0$}; \\
  \Z & \quad \text{ if $n=n'$ and $q=q'=0$.} \end{cases} \]
\end{co}

We remark that $\KMW_0$ is the sheaf of Grothendieck--Witt groups, also denoted $\GW$. We observe that if $q > 0$, then,
by Corollary \ref{co:homapsAsGlobalSections}:
\[[\ssp S^{n+q\alpha}, \ssp S^{n+q\alpha}]_\Aone = \GW(\pt).\]

\begin{pr}
  Suppose $n , n', q, q'$ are integers such that $n, q, q'$ are nonnegative and $q' \ge 1$. We have an identification
  \[ \pi_{n+q\alpha}^{s, P, \Aone} (S^{n'+q'\alpha}) = \begin{cases} 0 & \text{ if $n < n'$,}\\
   \KMW_{q'-q}   \tensor_\ZZ \ZZ_P   & \text{ if $n=n'$}. \end{cases}\]
\end{pr}
\begin{proof}
  This follows immediately from Corollary \ref{co:mapsSphereSpectra} and
  Proposition \ref{pr:PAoneLocalHomotopySheaves}.
\end{proof}

\begin{remark}
  Since $\ZZ$ is a subring of $\KMW_0(k) = \GW(k)$, it follows that $\ZZ_P$ is a subring of $\GW(k) \tensor_\ZZ \ZZ_P$.
\end{remark}

\subsection{Twist classes}

For $a \in k^*$, following \cite[Chapter 3]{morel2012}, we define $\langle a \rangle: S^{(1,0)} \wedge \G_m \to
S^{(1,0)} \wedge \G_m$ to be the map induced by multiplication $a: \G_m \to \G_m$ by forming $a_+: (\G_m)_+ \to \G_m$,
suspending \[ S^{(1,0)} \wedge (a_+): S^{(1,0)} \wedge (\G_m)_+ \cong S^{1,0} \vee S^{(1,0)} \wedge \G_m \to S^{(1,0)}
\wedge \G_m \] and letting $\langle a \rangle$ denote the restriction of this map to the $ S^{(1,0)} \wedge \G_m$
summand. 

\begin{remark} \label{rm:defe}
  The interchange of any two 
terms in $\left(S^{n+ q \alpha}\right)^{\wedge r} = S^{n+q\alpha} \wedge S^{n+q\alpha} \wedge \dots \wedge S^{n+q\alpha} \iso S^{rn + rq\alpha}$ represents the element 
\[ e_{n,q} = (-1)^{n+q}\langle -1\rangle^q  \in \pi^\Aone_{rn + rq\alpha}(S^{rn+rq\alpha}),\]
by \cite[Lemma 3.43]{morel2012}. Observe that $e_{n,q}^2 = 1$.
\end{remark}

Much of the following work depends on showing a class $A + B \langle -1 \rangle$, where $A$ and
$B$ are integers, is a unit in the ring
$\GW(k)$ or $\GW(k) \tensor_\ZZ \ZZ_{(2)}$.

We remind the reader that a field $k$ is said to be \textit{formally real} if $-1$ cannot be written as a sum of squares
in $k$, \cite[Chapter VIII]{lam1973}.

\begin{pr} \label{pr:unitUnlocalized}
  Suppose $A, B$ are integers, and let $R$ be a localization of $\ZZ$. Then $A + B \langle -1 \rangle$ is a unit in
  $\GW(k)\tensor_\ZZ R$ if and only if one of the following
  conditions is met:
  \begin{enumerate}
  \item \label{li:i} $A+B$ and $A-B$ are units in $R$ and $k$ is formally real;
  \item $A + B$ is a unit in $R$ and the field $k$ is not formally real.
  \end{enumerate}
\end{pr}
\begin{proof}
  We remark that the dimension homomorphism makes $R$ into a split subring of $\GW(k) \tensor_\ZZ R$.

  We first handle the case where $k$ is formally real.

  Since $(A + B \langle -1 \rangle)( A + B \langle -1 \rangle) = A^2 - B^2 = (A+ B) (A-B)$, the condition in \eqref{li:i}
  is sufficient.

  We may embed $k$ in a real closure $\phi: k \to k^r$. This embedding induces a ring homomorphism $\phi: \GW(k)
  \tensor_\ZZ R \to \GW(k^r) \tensor_\ZZ R = R \oplus R\langle -1 \rangle$, \cite[Proposition
  II.3.2]{lam1973}. Abstractly, the ring $\GW(k^r)\tensor_\ZZ R$ is endowed with an automorphism $\langle -1 \rangle
  \mapsto - \langle -1 \rangle$, and if $\phi(A+ B \langle -1 \rangle) = A + B \langle -1 \rangle$ is a unit, then so
  too is $A - B \langle -1 \rangle$, from which we deduce that their product, $(A + B \langle -1 \rangle)(A - B \langle
  -1 \rangle) =A^2 - B^2$ is a unit as well. But $A^2 - B^2$ is a unit if and only if $A + B$ and $A-B$ are units, showing
  that this condition is necessary and sufficient if $k$ is formally real.

  \bigskip

  Suppose now that $k$ is not formally real.

  We may employ the dimension map $\GW(k)\tensor_\ZZ R \to R$ to show that if $A + B\langle -1 \rangle$ is a
  unit, then necessarily $A + B$ is a unit, $u$.

  We wish to show that this condition is also sufficient to imply $A + B \langle -1 \rangle$ is a unit. If the
  characteristic of $k$ is $2$, then $ \langle -1 \rangle=1$ and the result is trivial. We may therefore assume the
  characteristic of $k$ is not $2$. Write $A = u - B$.

  Suppose first the characteristic of $k$ is not $2$. In the non-formally-real case the Witt group
  $\mathrm{W}(k) = \GW(k)/(1 + \langle -1 \rangle )$ is $2$-primary torsion, \cite[Theorem VIII.3.6]{lam1973}. Moreover,
  the class $1 - \langle -1 \rangle$ lies in the fundamental ideal of $\GW(k)$, which is isomorphic to the its image in
  $\mathrm{W}(k)$, so it follows $1 - \langle -1 \rangle$ is $2$-primary torsion. Since
  $(1 - \langle -1 \rangle)^N = 2^{N-1} (1 - \langle -1 \rangle)$ by an easy induction, the element
  $1 - \langle -1 \rangle$ is nilpotent $\GW(k)$ and therefore also in $\GW(k) \tensor_\ZZ R$. We deduce
  $A + B \langle -1 \rangle = u - B( 1 - \langle -1 \rangle)$ is a unit in $\GW(k) \tensor_\ZZ R$ as required.
\end{proof}

We owe the argument in the non-formally-real case to the anonymous referee.

We employ Proposition \ref{pr:unitUnlocalized} via the following two corollaries.
\begin{co} \label{co:2localUnits}
  Suppose $m$ is a nonnegative integer and $e_{n,q} = (-1)^{n+q}\langle -1 \rangle^q$ is the twist class of the sphere $S^{n+q
    \alpha}$. Then the class $1+m + m e_{n,q}$ is a unit in $\GW(k) \tensor_\ZZ \ZZ_{(2)}$. 
\end{co}
\begin{proof}
  There are, in general, four cases, $e_{n,q} = \pm 1$ and $e_{n,q} = \pm \langle -1 \rangle$; although it is possible
  that $\langle - 1 \rangle = 1$ in $\GW(k)$. The two cases $e_{n,q} = \pm 1$ are immediate.

  For the other two cases, by the proposition, it suffices to check that one or both of $m +1 + m=2m + 1$ and $m+ 1 -m =
  1$ are units in $\ZZ_{(2)}$, which they both are.
\end{proof}

\begin{co} \label{co:GWUnits}
  Suppose $m$ is a positive integer and $e_{n,q} = (-1)^{n+q}\langle -1 \rangle^q$ is the twist class of the sphere $S^{n+q
    \alpha}$. Then the class $1+m + m e_{n,q}$ is a unit in $\GW(k)$ if and only if one of the following conditions
  holds
  \begin{itemize}
  \item $n$ is odd and $q$ is even;
  \item $n+q$ is odd and $k$ is not formally real.
  \end{itemize}
\end{co}
\begin{proof}
  The cases where $q$ is even, so $e_{n,q} = \pm 1$, are easily dealt with and do not depend on the field. Assume
  therefore that $q$ is odd.

  Suppose $k$ is formally real, then the Proposition says that $1+m \pm m \langle -1 \rangle$ is a unit if and only
  if $1$ and $1+2m$ are units in $\ZZ$. Therefore there are no cases where the class is a unit, $q$ is odd and $k$ is
  formally real.

  Suppose $k$ is not formally real, and $q$ is odd. Then by the proposition $1 + m + me_{n,q}$ is a unit if and only if
  $1+m - (-1)^n m$ is a unit, whereupon it is necessary and sufficient for $n$ to be even.
\end{proof}

\section{The Hilton-Milnor splitting}\label{Hilton-Milnor_Snaith_section}

The James construction on a pointed simplicial set was introduced by I.~James in \cite{james1955}. The idea of applying it in $\Aone$ homotopy
theory, and thereby obtaining a weak equivalence $J(X) \weq \Omega \Sigma X$ as in Proposition
\ref{J(X)=ho-nis=Omega_SigmaX}, is not original to us. We learned of it from A.~Asok and J.~Fasel, who attribute it to F.~Morel.

Our presentation is based on that of \cite[Chapter VII.2]{whitehead2012}, and we also refer frequently to
\cite{wu2010}. Suppose $X$ is a pointed simplicial set. An injection $\alpha: (1,2,\ldots,n) \to (1,2,\ldots, m)$ induces a map
  $\alpha_* : X^n \to X^m$. Let $\sim$ denote the equivalence relation on $\coprod_{n=0}^{\infty} X^n$ generated by $x
  \sim \alpha_*(x)$ for all injections $\alpha$. The James construction on $X$ is $J(X)= \coprod_{n=0}^{\infty} X^n /
  \sim $. The construction $J(X)$ is the free monoid on the pointed simplicial set $X$. The $k$-simplices $J(X)_k$ of
  $J(X)$ are the free monoids on the pointed sets $X_k$, that is $J(X)_k =\coprod_n X_k^n / \sim$. The James
  construction is filtered by pointed simplicial sets $J_n(X)$, defined $J_n(X)= \coprod_{m=0}^{n} X^m / \sim$.

Define spaces $D_n(X)$ as the cofibers  of sequences \[J_{n-1}(X) \to J_n(X) \to D_n(X).\] 
There are canonical weak equivalences $D_n(X) \to X^{\wedge n}$. Define $D(X) = \bigvee_{n=0}^{\infty} D_n(X)$.

\begin{df}\label{J_etc_def}
  For a pointed simplicial presheaf $X$, define $J (X), J_n(X), D_n(X), D(X) \in \sPre$ by $$J (X) (U) = J(X(U)), \quad
  J_n (X) (U) = J_n(X(U)), \quad D_n(X) (U) = D_n (X(U)), \quad D(X) (U) = D (X(U)).$$ Let $\ell: X=J_1(X) \to J(X)$
  denote the map induced by the canonical maps $X(U) \to J(X(U))$ for $U \in \Sm$.

The {\em James construction} is then defined to be $\Laone J(X)$.
\end{df} 

We learned the following result from A.~Asok and J.~Fasel. It is a functorial version of a result due to James,
\cite[Theorem 5.6]{james1955}, and is presented in more recent terminology in \cite[Chapter VII, 2.6]{whitehead2012}
\begin{pr}\label{J(X)=ho-nis=Omega_SigmaX}
  Suppose $X$ is a connected object of $\Spaces_\pt$. There is a natural isomorphism $J(X) \to \Omega \Sigma X$ in
  $\ho_\Nis \Spaces_\pt$.
\end{pr}
\begin{proof}
  For any object $U$ in $\Sm_k$, there is a functorial map $j: J(X(U)) \to F(X(U))$ where $F(X(U))$ is Milnor's
  construction---the free abelian group on the pointed space $X(U)$--, as laid out in \cite[Sections 3.3.2 and
  3.3.3]{wu2010}. This map is a weak equivalence when $X(U)$ is connected, \cite[Theorem 3.3.5]{wu2010}. There is a
  natural isomorphism $F(X(U)) \iso G \Sigma X(U)$ \cite[Chapter V, Theorem 6.15]{goerss1999}. Here, $G \Sigma X(U)$ is
  a fibrant model for $\Omega \Sigma (X(U))$. In particular, there is a zigzag of maps of simplicial presheaves
 \[ J(X) \leftarrow F(X) \overset{\iso}{\to} G \Sigma (X) \]
 The formation of $J(X(U))$ and $F[X(U)]$ commute with colimits, since the free abelian monoid and the free abelian
 monoid functors do. Therefore, both $J$ and $F$ commute with taking stalks at a point $p^*$. Since $p^*X$ is a
 connected simplicial set, it follows that $p^* J(X) \iso J(p^*X) \to F(p^*X) \iso p^* F(X)$ is a weak equivalence,
 whence the result.

\end{proof}

Note that this natural isomorphism induces a natural isomorphism $J(X) \to \Omega \Sigma X$ in $\ho_\Aone
\Spaces_\pt$. It also follows immediately from this result that if $X \to Y$ is a local weak equivalence, then the
functorial map $J(X) \to J(Y)$ is a local weak equivalence. 

\begin{co}\label{co:pi_iAoneJ=pi_i+1AoneSigma}
Suppose $X$ is a connected object of $\Spaces_\pt$. Then there is a natural isomorphism $$\pi_i^{\Aone} J(X) \cong \pi_{i+1}^{\Aone} \Sigma X.$$
\end{co}

\begin{proof}
By Proposition \ref{J(X)=ho-nis=Omega_SigmaX}, there is a natural isomorphism $\pi_i^{\Aone} J(X) \cong \pi_i^{\Aone} \Omega \Sigma X$. By definition, $\pi_i^{\Aone} \Omega \Sigma X = \pi_i \Laone \Omega \Sigma X$. Since $\Sigma X$ is simplicially simply connected, $\Omega \Sigma X$ is simplicially connected. By Morel's connectivity theorem \cite[Theorem 6.38]{morel2012}, $\Laone \Omega \Sigma X$ is also simplicially connected. Thus $\pi_0 \Laone \Omega \Sigma X \cong \ast$ is strongly $\Aone$-invariant. By \cite[Theorem 6.46]{morel2012}, it follows that the canonical morphism $\Laone \Omega \Sigma X \to \Omega \Laone \Sigma X$ is a simplicial weak equivalence. Thus there is a natural isomorphism $\pi_i \Laone \Omega \Sigma X \cong  \pi_{i+1} \Laone \Sigma X$. Combining with the previous gives the claimed natural isomorphism $\pi_i^{\Aone} J(X) \cong \pi_{i+1}^{\Aone} \Sigma X$.
\end{proof}

\begin{co}
  Suppose $X \to Y$ is an $\Aone$ weak equivalence of connected objects of $\Spaces_\pt$, then the functorial map
  $J(X) \to J(Y)$ is an $\Aone$ weak equivalence.
\end{co}
\begin{proof}
  We will show that if $X \to Y$ is an $\Aone$ weak equivalence of connected objects, then $\Omega  \Sigma X \to \Omega
  \Sigma Y$ is an $\Aone$ weak equivalence. Since there is a natural isomorphism  $\Omega \Sigma X \iso J(X)$ in $\ho
  \Spaces_\pt$, and therefore in $\ho_\Aone \Spaces_\pt$, it will follow that $J(X) \to J(Y)$ is an isomorphism in
  $\ho_\Aone \Spaces_\pt$, as claimed.

  The following is a sequence of local weak equivalences.
  \begin{align*}
    \Laone  \Sigma  X & \weqto \Laone \Sigma  Y  & \text{(since $\Aone$ localization is simplicial)}\\
    \Omega \Laone \Sigma X & \weqto \Omega \Laone \Sigma Y \\
    \Laone \Omega \Sigma X & \weqto \Laone \Omega \Sigma Y  & \text{ by \cite[Theorem 6.46]{morel2012}} 
  \end{align*}
  but this is precisely what was to be shown.
\end{proof}

Given $W,X \in \sSet_*$ and a map $f: (J_n W, J_{n-1} W) \to X$, we define the {\em combinatorial extension} of $f$
\[h(f): J(W) \to J(X)\]
by following the procedure of \cite[1.4, \S 2 ]{james1955} (cf.~\cite[Chapter VII.2]{whitehead2012}). 

 We first define the restriction of $h(f)$ to $J_m(W)$. For
 $m<n$, the restriction of $h(f)$ is the constant map. Suppose $m \ge n$. To an injection $(1,2,\ldots, n)
 \to (1,2,\ldots,m)$, we may associate a map $W^m \to W^n$ and therefore a map $W^m \to W^n \to J_n(W)$. 

Consider the set of all $\binom{m}{n}$ increasing, $n$--term subsequences of $(1,2,\dots, m)$. Order these by
lexicographic ordering, reading from the right. Each sequence is an injective map $\{1,\dots,n\} \to \{1, \dots
m\}$. Taking the ordered product over all injections, we obtain a total map \[W^m \to J_n(W)^{\binom{m}{n}}.\] 

The ${m \choose n}$-fold product of the map $f$ is map \[ J_n(W)^{{m \choose n}} \to X^{m \choose n}.\] We set the restriction
of $h(f)$ to $J_m(W)$ to be \[W^m \to J_n(W)^{m \choose n} \to X^{m \choose n} \to J_{m \choose
  n}( X) \to J(X).\] One checks that this is well-defined.

 This definition is functorial, and extends immediately to presheaves: 
 
 \begin{df}\label{presheaf_combinatorial_extension_and_j_n} Given $W,X \in \sPre_{\ast}$ and a map $f: (J_n (W), J_{n-1} (W)) \to X$,
 we may define the {\em combinatorial extension of $f$}: \[h(f): J (W) \to J (X), \quad\quad h(f) (U) = h(f(U)).\]

 For $X \in \sSet_{\ast}$, the cofiber sequences $J_{n-1}(X) \to J_n(X) \to D_n(X)$ induce natural maps $(J_n(X),
 J_{n-1}(X)) \to D_n(X)$.  For $X \in \sPre$, we thereby obtain maps $(J_n(X), J_{n-1}(X) )\to D_n(X) $, and
 consequently maps \[j_n: J (X) \to J(D_n(X))\] by combinatorial extension. 
 \end{df}
 
 Let $i_n: J(D_n(X)) \to J(D(X))$ be the map
 induced by the canonical inclusions $D_n(X(U)) \to D (X(U))$. The monoid structure on $J(D(X(U)))$ induces
 multiplication maps $\mu_n: J(D(X))^n \to J(D(X))$.  Consider the maps \[\mu_{n+1} \prod_{m=0}^n i_m j_m : J (X) \to
 J(D(X))\] for $n = 0,1,2,\ldots$. It is important here that the product be ordered, and we declare it to be ordered by increasing values of $m$.

 The composition of $i_n j_n$ with $J_{n-1} (X) \to J (X)$ is the constant based map. We'll say that the restriction of
 $i_n j_n$ to $J_{n-1} (X)$ is the constant map. It follows that $ \mu_{n+1} \prod_{m=0}^n i_m j_m $ restricted to
   $J_{n-1} (X)$ is equal to the restriction of $ \mu_{N+1} \prod_{m=0}^{N} i_m j_m $ to $J_{n-1} (X)$ for all $N \geq
   n$. Note that $J (X) = \colim J_n (X)$. Thus we may define $f: J (X) \to J(D(X))$ by $$f = \colim_n \mu_{n+1}
   \prod_{m=0}^n i_m j_m .$$ For convenience, extend $f$ to $f_+ : J (X)_+ \to J(D(X))$ by mapping the disjoint point
  via $\pt = X(U)^{\wedge 0} \to D X \to J(D(X))$.

  Taking the simplicial suspension of $f_+$, we obtain \[\Sigma f_+: \Sigma (J (X)_+) \to \Sigma J(D(X)).\] 

Let $| \cdot |: \sSet \to \cat{K}$ denote the geometric realization functor from simplicial sets to Kelly spaces,
and let $\ssimp: \cat{K} \to \sSet$ be the right adjoint functor, which is the functor of singular simplices.

We claim that for any simplicial presheaf $Y$, for example $Y = D(X)$, there is an evaluation
map \begin{equation}\label{SigmaJ_to_ssimp_vert_Sigma}\Sigma J(Y) \to \ssimp | \Sigma Y|.\end{equation} 
To see this, let $\Omega^M: \cat{Top} \to \cat{Top}$ denote the Moore loops functor, \cite{carlsson1995}. There is a
strictly associative multiplication $\Omega^M \times \Omega^M \to \Omega^M$. Since taking $\ssimp$ commutes with finite products, there
is a strictly associative multiplication on $\ssimp \Omega^M$, and therefore an induced commutative diagram 
\[ \xymatrix
{ Y \ar[rr] \ar[rd] && \ssimp \Omega^M \vert \Sigma Y \vert \\ & J(Y). \ar[ur] & }\]
Applying $\Sigma$, we obtain a map $\Sigma J(Y) \to \Sigma \ssimp \Omega^M \vert \Sigma Y \vert$. There is a natural
transformation of functors $\Sigma \ssimp \to \ssimp \Sigma$ and so we have a composite
\begin{equation} \label{eq:0} \Sigma J(Y) \to \Sigma \ssimp \Omega^M \vert \Sigma Y \vert \to \ssimp \Sigma \Omega^M \vert \Sigma Y \vert .\end{equation}
The counit of the adjuction between loops and suspension produces a natural transformation $\Sigma \Omega^M \to
\id$. Composing with \eqref{eq:0} produces a map \[\Sigma J(Y) \to \Sigma \ssimp \Omega^M \vert \Sigma Y \vert \to
\ssimp \Sigma \Omega^M \vert \Sigma Y \vert \to \ssimp \vert \Sigma Y \vert, \] which is what we claimed in \eqref{SigmaJ_to_ssimp_vert_Sigma}.

Composing $\Sigma f_+$ with \eqref{SigmaJ_to_ssimp_vert_Sigma} for $Y = D(X)$ produces a
map \begin{equation}\label{Jpre+toSimpvertSigmaD} \Sigma J(X)_+ \to \ssimp \vert \Sigma D(X) \vert.\end{equation} For
each $U$, this map \eqref{Jpre+toSimpvertSigmaD} evaluated at $U$ is a weak equivalence, see \cite[VII Theorem
2.10]{whitehead2012}. Thus \eqref{Jpre+toSimpvertSigmaD} is a weak equivalence in the simplicial model structures on $\Spaces_\pt$. Combining with the
injective weak equivalence $\Sigma D(X) \to \ssimp \vert \Sigma D(X) \vert$, we have the zig-zag of injective weak
equivalences  \begin{equation} \label{zigzag_for_SigmaJ}\xymatrix{ \Sigma J(X)_+ \ar[r] & \ssimp \vert \Sigma D(X) \vert
    & \ar[l] \Sigma D(X)}. \end{equation} 

We have shown:

\begin{pr}
Suppose $X$ is a connected object of $\Spaces_\pt$. There is a canonical isomorphism $\Sigma  J(X)_+ \to \Sigma D(
 X)$ in $\ho\Spaces_\pt$. 
\end{pr}

\begin{co}\label{co:J(X)weakequivD(X)}
There is a canonical isomorphism $\sigma:  J(X)_+ \to  D(X)$ in $\ho_\Nis \soneSpt$.
\end{co}

\begin{remark}
  Here and subsequently we write $J(X)_+$, $D(X)$ in place of the stable $\ssp J(X)_+$, $\ssp D(X)$ whenever the context
  demands $S^1$ stable objects.
\end{remark}

\subsection{The low-dimensional simplices of the $J$ construction}\label{section:low-dimensional_simplices_J}

\begin{df}
  We say a simplicial presheaf $X$ is \textit{$n$-reduced} if the unique map $X \to \pt$ induces an isomorphism $X_i \to
  \pt_i$ for $i \le n$. The term ``$0$-reduced'' may be abbreviated to ``\textit{reduced}''.
\end{df}

Equivalently $|X_i(U)| = 1$ for all smooth schemes $U$ and all $i \le n$.

\begin{example}
  The constant simplicial presheaf representing the simplicial $n$-sphere, $\Delta^n/ \bd \Delta^n$ is $(n-1)$-reduced.
\end{example}

\begin{example}
 Suppose $X, Y$ are $n$-reduced. Then $X \times Y$ is $n$-reduced.
\end{example}

\begin{example}
  Suppose $X$ is a $n$-reduced simplicial presheaf, given the unique pointed structure, and $Y$ is a simplicial presheaf
  pointed by a map $s_0: \pt \to S$. Then $X \wedge Y$ is $n$-reduced. To see this, fix a smooth scheme $U$ and consider
  the construction of $(X \wedge Y)(U)_i$ for $i \le n$. It is given by the pushout
  \begin{equation}
    \label{eq:6}
     \xymatrix{ (X \vee Y)(U)_i \ar[r] \ar[d] & (X \times Y)(U)_i \ar[d] \\ \ast \ar[r] & (X \wedge Y)(U)_i, }
  \end{equation}
  but since $X$ is $n$-reduced, $(X \vee Y)(U)_i = X(U)_i \vee Y(U)_i = Y(U)_i$ and $(X(U) \times Y(U))_i = X(U)_i
  \times Y(U)_i = Y(U)_i$. In particular, the top horizontal arrow of diagram \eqref{eq:6} is a bijection, so so too is
  the bottom arrow.
\end{example}

\begin{example} \label{example:reductionOfSpheres} As a special case of the above, our models for the motivic spheres,
  $S^n \wedge \Gm^m = S^{n+ m \alpha}$ are $(n-1)$-reduced. Note that the space $S^{n+ m \alpha} \wedge S^{n+m\alpha}$
  is $(n-1)$-reduced but not $n$-reduced, whereas the weakly equivalent space $S^{2n+2m\alpha}$ is $2n-1$-reduced. This
  holds even when $m=0$, that is, in the case of classical homotopy theory.
\end{example}

\begin{pr}\label{pr:Jreduced}
  Let $X$ be an $n$-reduced simplicial presheaf. Then $J(X)$ is $n$-reduced.
\end{pr}
\begin{proof}
  Let $i \le n$. We can calculate $J(X)(U)_i$ directly. Since the category of simplicial presheaves on $\Sm_k$ is really
  the category of presheaves of sets on $\Sm_k \times \Delta$, where $\Delta$ is the standard simplex category, both
  evaluation at a smooth scheme $U$ and taking $i$-th simplicies (which is ``evaluation at $\Delta^i$'') commute with
  all limits and colimits.

Therefore, we may calculate
  \[ J(X)(U)_i =  \Big(\coprod_{m=0}^\infty X(U)^m / \sim \Big)_i = \coprod_{m=0}^\infty (X(U)_i)^m / \sim \]
  but this last is simply $\coprod_{m=0}^\infty \pt/\sim$. The $\sim$ relation identifies two $i$-simplices in
  $(X(U)_i)^m$ and $(X(U)_i)^{m'}$ if one is obtained from the other by means of an order-preserving injection of the
  indexing set. But all such injections induce the identity map $\ast= (X(U)_i)^m \to (X(U)_i)^{m'}= \ast$, so it
  follows that $J(X)(U)_i$ collapses to a singleton set, as required.
\end{proof}

\section{The Stable Isomorphism}\label{Section:stable_isomorphism}

\subsection{The diagonal}

Let $\sigma: J(X)_{+} \to  D(X)$ denote the stable isomorphism in $\ho\soneSpt$ of Corollary
\ref{co:J(X)weakequivD(X)}. The category $\ho_\Nis \soneSpt$ is equipped with localization functors to $\ho_\Aone
\soneSpt$, $\ho_p \soneSpt$ and $\ho_{p,\Aone} \soneSpt$, this being the upshot of Section \ref{sec:localization}. We will denote the
images of objects and morphisms under the various localization functors by the same notation as we use in the category
$\ho_\Nis \soneSpt$, and in order to avoid confusion we will specify the category in which we are working.

Let $\Delta^q: J(X) \to J(X)^q$ denote the order-$q$ diagonal of $J(X)$.

The study of $\Delta^q$ has been extensively carried out by N.~Kuhn in \cite{Kuhn_87}, \cite{kuhn2001} and elsewhere. We
describe some of that study here. Fix a natural number $i$ and a finite sequence $A=(a_1, a_2, \dots, a_q)$ of natural
numbers. Following \cite{kuhn2001}, we define a \textit{partial $A$ cover} to be a collection of subsets
$T= ( T_1, \dots, T_q)$\footnote{This is denoted $S$ in \cite{kuhn2001}, but we have reserved $S$ for the symmetric
  group because we use $\Sigma$ for the reduced suspension functor} of $\{1, \dots, i\}$ with the property that
$|T_t| = a_t$. We will say this partial cover is of \textit{type} $(a_1, \dots, a_q)$. The partial cover $T$ is called a \textit{cover} if $\bigcup_{t=1}^q T_t = \{1, \dots, i\}$. In \cite[Section 2]{kuhn2001},
a natural stable map of spaces is constructed, associated to a cover $T$:
\[ \Psi_T: D_i X \to D_{a_1}(X) \wedge D_{a_2}(X) \wedge \ldots \wedge D_{a_q}(X).\]
We should mention that the definition of $D_i X$ in \cite{kuhn2001} is not the same as ours, merely homotopy
equivalent. Since he is considering a vastly more general context than we are, Kuhn defines $D_iX $ in terms of a
coefficient system, whereas we consider, in the language of coefficient systems, only the little-intervals operad. The
little intervals operad has $\mathcal{C}(n)$ homotopy equivalent to the space of $n$ distinct marked points on a unit
interval, which is homotopy equivalent to the symmetric group $S_n$. This equivalence allows us our much more elementary
definition of $D_iX$, homotopy equivalent to that of Kuhn. Specifically, where we have $X^{\wedge i}$, Kuhn takes the
space of $i$ points on an interval, each labelled with an element of $X$, and then contracts the subspace where any label is the
basepoint of $X$ to a point.

Using our construction of $D_i X$, it is possible to give an elementary description of
$\Psi_T : D_i X \to D_{a_1}(X) \wedge D_{a_2}(X) \wedge \ldots \wedge D_{a_q}(X)$. The set $\{1, \dots, i\}$ is
covered by subsets $A_1, \dots, A_q$ of cardinalities $a_1, \dots, a_q$. To each we associate the order-preserving
inclusion $\alpha_{T, t}: \{1, \dots, a_t\} \to \{1 ,\dots, i\}$, and then define a map
\[ \psi_T : X^{\times i} \to X^{\times a_1} \times \dots X^{\times a_t}, \quad  (x_1, \dots, x_i) \mapsto \prod_{t=1}^q
\Big( \prod_{k=1}^{a_t} x_{\alpha_{T,t}(k)} \Big) \]
where the products are taken in the usual order.

For any $\sigma \in S_i$, we may act on $X^{\times i}$ by permuting the factors, and thereby obtain $\psi_T \circ
\sigma : X^{\times i} \to  X^{\times a_1} \times \dots X^{\times a_t}$. We act on the codomain by the unique element
of  $S_{a_1} \times
\dots \times S_{a_t}$ so that for each $j \in \{1,\dots,t\}$, the composite map $X^{\times i} \to X^{\times a_1} \times \dots X^{\times a_t} \to
X^{\times a_j}$ is order preserving. Write $[\sigma]\psi_T$ for the resulting map
\[ [\sigma]\psi_T : X^{\times i} \to X^{\times a_1} \times \dots X^{\times a_t}. \]
It depends only on the coset of $\sigma$ modulo $S_T$, the stabilizer of the cover $T$.

If $T$ is a cover, and if the basepoint of $X$ appears in the $i$-tuple on the left, then it will appear on the right,
and so the maps $[\sigma]\psi_T$ descend to a maps $X^{\wedge i} = D_i X \to D_{a_1} X \wedge \dots \wedge D_{a_q} X$,
which we also write $[\sigma]\psi_T$ in an abuse of notation. Finally the stable map
\[ \Psi_T : D_i X \to D_{a_1}(X) \wedge D_{a_2}(X) \wedge \ldots \wedge D_{a_q}(X) \]
is produced as a sum $\sum_{ \sigma \in S_i/S_T } [\sigma] \psi_T$.

For example, given the cover $T=(\{1,2\}, \{3,4\})$, the function $\Psi_T$ can be represented
as the sum of the functions sending a reduced word $x_1x_2x_3x_4$ to each of the following:
\begin{align*}
 &((x_1x_2),(x_3x_4))  \quad &((x_1x_3),(x_2x_4)) \\ &((x_1x_4),(x_2x_3)) \quad & ((x_2x_3),(x_1x_4)) \\ &(
                                                                                                           (x_2x_4),(x_1x_3))
                                                                                                           \quad &( (x_3x_4),(x_1x_2) ) .
\end{align*}

The construction of the map $\Psi_T$ is natural, and therefore it induces a natural stable map of simplicial presheaves
which we also denote by $\Psi_T$.

Two covers are said to be \textit{equivalent} if they lie in the same orbit of the symmetric group action on
$\{1, \dots, i\}$. The proof in \cite[Lemma 2.6]{Kuhn_87} that equivalent covers induce homotopic maps is entirely
formal and carries over to the setting of simplicial presheaves.

\begin{df}
Let $\Delta^q_{i, (a_1,a_2,\ldots, a_q)}(X)$ denote the composition in $\ho_\Nis \soneSpt$:
\[ \xymatrix{D_i X \ar[r] & D(X) \ar[r]^{\sigma^{-1}} & J(X)_{+} \ar[r]^{\Delta^q} & (J(X)^q)_+ \ar[r]^{\bigwedge^q \sigma} & \bigwedge^q D(X) \ar[r] & D_{a_1}(X) \wedge D_{a_2}(X) \wedge \ldots \wedge D_{a_q}(X)} .\]
\end{df}

The following result is essentially \cite[Theorem 2.4]{kuhn2001}, but we assert it in the category of simplicial
presheaves.

\begin{pr} \label{th:KuhnMain} There is an equality in the stable homotopy category
  of presheaves
  \[\Delta^q_{i, (a_1, a_2, \dots, a_q)}  = \sum_T \Psi_T : D_i X \to D_{a_1}(X) \wedge D_{a_2}(X) \wedge \ldots \wedge
  D_{a_q}(X) \]
  where the sum runs over equivalence classes of covers $T$ of type $(a_1, \dots, a_q)$.
\end{pr}
\begin{proof}
  We outline the argument of Kuhn, noting that at all points all constructions are natural, and therefore the proof
  carries over essentially without modification to the context of simplicial presheaves.

  The problem is reduced from that of $DX$ to $D(X_+)$, that is, to the case of a space where the basepoint is
  disjoint. The device that allows this reduction is that the map in the stable homotopy category $D(X_+) \to DX$ is a
  split epimorphism. This follows from the fact that in general the map $X \times Y \to X \wedge Y$ is split after
  application of a single suspension functor, and applies equally well in the case of presheaves as in the case of
  spaces.

  Kuhn observes immediately after Lemma 3.1 of \cite{kuhn2001} that $C(X_+)_+$ is naturally homeomorphic to $D(X_+)$,
  denoting the natural homeomorphism by $s'$. We work with a different model for $D(X_+)$ and the $C(X_+)$ of
  \cite{kuhn2001} corresponds to our $J(X_+)$, and the homeomorphism $s'$ in our context amounts to the observation that
  the reduced free monoid on a simplicial set $K_+$ is naturally isomorphic to $S^0 \vee K_+ \vee (K \times K)_+ \vee \dots$.

  The homeomorphism $(s'): J(X_+)_+ \to D(X_+)$ does not stabilize to give the Hilton--Milnor splitting
  $\sigma: J(X_+)_+ \to D(X_+)$ of Corollary \ref{co:J(X)weakequivD(X)}. For this reason, \cite[Definition
  3.2]{kuhn2001} defines $\Theta: D(X_+) \to D(X_+)$ to be the stable map $\sigma \circ (s')^{-1}$, and defines
  $\Theta_{n,m}: D_n(X_+) \to D_m(X_+)$ to be the stable maps given by the $(n,m)$-th component of $\Theta$, i.e.,
  $\Theta_{n,m}$ is a stable map $D_n(X_+) \to D_m(X_+)$. The proofs of Propositions 3.3 (the proof of which is the
  proof of
  \cite[Proposition 4.5]{Kuhn_87}) , 3.4 and Lemmas 3.5 and 3.6 of \cite{kuhn2001} are entirely formal, relying on the
  behaviour of the stable transfer maps associated to subgroups of the symmetric group. They apply without modification
  to simplicial presheaves, and therefore suffice to establish our proposition.
\end{proof}

\begin{pr}[Kuhn \cite{kuhn2001}] \label{pr:KuhnDelta}
  If $X$ is an object of $\Spaces_\pt$ equipped with a co-$H$-structure, then $\Delta^q_{i, (a_1, \dots, a_q)}(X) \iso
  \ast$ in $\ho_\Nis \soneSpt$, unless $i = \sum_{j=1}^q a_j$.
\end{pr}
\begin{proof}
  If $i < \sum_{j=1}^q a_j$, then this is true even without the co-$H$-structure. In Proposition \ref{th:KuhnMain}, there are no covers in this
  case, and therefore the sum is empty.

  If $X$ is equipped with a co-$H$-structure, then the diagonal $X \to X \times X$ factors up to homotopy through a
  co-$H$-map $X \to X \vee X$. In this case, the diagonal map $X \to X \wedge X$ is nullhomotopic. The argument in the
  first part of \cite[Appendix A]{kuhn2001} applies directly, amounting to the claim that if $i >\sum_{j=1}^q a_j$, then
  each term in the sum of Proposition \ref{th:KuhnMain} is also null.
\end{proof}

\subsection{Combinatorics}

Our first aim is to prove that a specific stable map, namely $\Delta^2_{2m+r, (2m, r)} : X^{\wedge 2m +r} \to X^{\wedge
  2m} \wedge X^{\wedge r}$ is as an equivalence, at least after localizing at $(2)$, when $X\weq S^{2n +q \alpha}$ is a
motivic sphere and when $m$ is a natural number and $r$ is either $0$ or $1$. This is a corollary of
Propositions \ref{th:KuhnMain} and \ref{pr:KuhnDelta}, but the proof requires us to consider some elementary
combinatorics. The same combinatorics will prove useful later, when we turn to studying the James--Hopf map.
\smallskip

For a positive integer $a$ and an $m$-tuple of positive integers $(a_1, a_2, \ldots a_m)$, let ${a \choose a_1, a_2, \ldots a_m}$ denote the set of functions
$\sigma: \{1,2,\ldots, a \} \to \{1,2,\ldots,m\}$ such that $\sigma^{-1}(i)$ has cardinality $a_i$.  Note that ${a
  \choose a_1, a_2, \ldots a_m}$ is non-empty if and only if $a = \sum a_i$.

Given an element $\sigma \in {a \choose a_1, a_2, \ldots a_m}$, and a natural number $i \le m$, write $\sigma^{-1}(i)$
as $\{ \sigma^{-1}(i)_1, \sigma^{-1}(i)_2, \dots, \sigma^{-1}(i)_{a_i} \}$ in such a way that $\sigma^{-1}(i)_j <
\sigma^{-1}(i)_{j+1}$ for all $j \le a_i -1$. Define $\tilde \sigma $ to be the permutation on $a$ letters sending $(1 ,
2, \dots, a)$ to $(\sigma^{-1}(1)_1, \sigma^{-1}(1)_2, \dots, \sigma^{-1}(1)_{a_1}, \sigma^{-1}(2)_1 , \dots,
\sigma^{-1}(2)_{a_2}, \dots, \sigma^{-1}(m)_{a_m} )$. 

For instance, if $\sigma$ is the element of ${ 5 \choose 1,2,2 }$ given by sending $2 \mapsto 1$ and $1,5 \mapsto 2$ and
$3,4 \mapsto 3$, then $\tilde \sigma$ is the permutation taking $(1,2,3,4,5)$ to $(2,1,5,3,4)$.

Suppose $X$ is an objects of $\Spaces_\pt$ For $\sigma \in {a \choose a_1, a_2, \ldots, a_m}$, define $e(\sigma): X^{\wedge a} \to X^{\wedge a}$ to be the map
induced by $\tilde \sigma$, and define $\sgn (\sigma)$ to be the number of pairs $r<k$ in $ \{1,2,\ldots, a \}$ such
that $\tilde \sigma(r) > \tilde \sigma(k)$. In the example given, $\sgn(\sigma)$ is the cardinality of $\{ (1,2), (3,4),
(3,5) \}$, i.e.~$3$.

\begin{remark}
  Suppose $X$ is an object of $\Spaces_\pt$. Let $i, a_1, a_2, \ldots, a_q$ be positive integers such that $i = \sum a_k$. Then
  \[ \Delta^q_{i, (a_1,a_2,\ldots, a_q)}(X) = \sum_{\sigma \in {i \choose a_1, a_2, \ldots a_q}} e( \sigma)\]
  in $\ho_\Nis \soneSpt$---this is merely a restatement of a special case of Proposition \ref{th:KuhnMain} in different notation.
\end{remark}

In the case where $X = S^{n + q\alpha}$, Remark \ref{rm:defe} says that $e(\sigma)$ is $e_{n,q}^{\sgn \sigma}$, where
$e_{n,q} = (-1)^{n+q}\langle -1 \rangle^q$ in $\GW(k)$. Using the remark above then gives:
\begin{co}\label{co_delta=sum_esignsigma}
 Suppose $X=S^{n+q\alpha}$. Let $i, a_1, a_2, \ldots, a_w$ be non-negative integers such that $i = \sum a_k$. Then 
\[ \Delta^w_{i, (a_1,a_2,\ldots, a_w)}(X) = \sum_{\sigma \in {i \choose a_1, a_2, \ldots a_w}} e_{n,q}^{\sgn \sigma} \]
in $\ho_\Nis \soneSpt$.
\end{co}

We will have occasion to use an involution 
\[ \gamma: { a_1 + a_2 + \dots + a_m \choose a_1, a_2, \dots, a_m } \to  { a_1 + a_2 + \dots + a_m \choose a_1, a_2,
  \dots, a_m } \]
which is defined as follows:

Take $\sigma \in { a_1 + a_2 + \dots + a_m \choose a_1, a_2, \dots, a_m }$. There are two possibilities
\begin{enumerate}
\item  $\sigma(2i-1) = \sigma(2i)$ for all applicable $i$. In this case, we say $\gamma(\sigma) = \sigma$, so $\sigma$
  is fixed under the involution. We write $F_\gamma(a_1 + a_2 + \dots + a_m; a_1, a_2, \dots, a_m )$ for the set of
  fixed points, or $F_\gamma$ when the coefficients are clear from the context.
\item Otherwise, there exists a least integer $i$ such that $\sigma(2i-1) \neq \sigma(2i)$. We then let $\gamma(\sigma)$
  be the function that agrees with $\sigma$ except that $\gamma(\sigma)(2i-1) = \sigma(2i)$ and $\gamma(\sigma)(2i) =
  \sigma(2i-1)$.
\end{enumerate}

If $\sigma$ is not a fixed point of $\gamma$, then $\sgn(\sigma) + \sgn(\gamma(\sigma)) \equiv 1 \pmod 2$, so that the
number of elements in ${ a_1 + a_2 + \dots + a_m \choose a_1, a_2, \dots, a_m }$ of even sign is given by the formula
\begin{equation}
  \label{eq:2}
 \begin{split} \text{Number of elements of even sign} = \frac{1}{2}  \left(\left| { a_1 + a_2 + \dots + a_m \choose a_1, a_2, \dots, a_m
    }\right| - |F_\gamma|\right) + |F_\gamma| = \\ = \frac{1}{2}\left(\left| { a_1 + a_2 + \dots + a_m \choose a_1, a_2, \dots, a_m
    }\right| + |F_\gamma|\right). \end{split}
\end{equation}
We can often find ways to calculate $\left| { a_1 + a_2 + \dots + a_m \choose a_1, a_2, \dots, a_m }\right|$ and
$|F_\gamma|$. 

The cardinality of ${x+y \choose x, y}$ is the binomial coefficient $\frac{(x+y)!}{x!y!}$. Let $[\frac{x}{2}]$ denote the greatest integer less than or equal to $\frac{x}{2}$.

\begin{pr} \label{pr:evenSignForDelta}
  Among the elements of ${ x+y \choose x, y}$, the number having even sign is
  \[\frac{1}{2} \left( \left|  {x+y \choose x, y } \right| +\left|  { [\frac{x+y}{2}] \choose [\frac{x}{2}],
        [\frac{y}{2}] } \right| \right).\]
\end{pr}
  By virtue of our definitions, the second summand is $0$ in the case where $x$ and $y$ are both odd.
\begin{proof}
  We rely on the involution $\gamma$ and \eqref{eq:2}. Since $| { x+y \choose x,y } | $
  is known, it remains to calculate $|F_\gamma|$.

There are several cases to consider:
\begin{enumerate}
\item If $x$ and $y$ are both odd, then every number in $\{1,\dots, x+y \}$ forms part of a pair $(2i-1, 2i)$, and there
  must be at least one pair for which $\sigma(2i-1) \neq \sigma(2i)$, since $\sigma^{-1}(1)$ is odd. There are therefore
  no fixed points of the involution.
\item If $x$ and $y$ are both even, then every number in $\{1,\dots, x+y \}$ forms part of a pair $(2i-1, 2i)$. In order
  for $\sigma$ to be fixed by $\gamma$, it must be the case that $\sigma(2i-1) = \sigma(2i)$ for all $i$. Defining $\tau
  \in { (x+y)/2 \choose x/2, y/2 }$ by the formula $\tau(i) = \sigma(2i)$, we see that there is a bijection between
 $F_\gamma$ and ${ (x+y)/2 \choose x/2 , y/2 }$.
\item If one of $x$ and $y$ is even and the other odd---say for specificity that $x$ is even and $y$ odd---then $\sigma$
  is fixed under $\gamma$ if $\sigma(x+y) = 2$ and $\sigma(2i-1) = \sigma(2i)$ for all $i \le
  \left[\frac{x+y}{2}\right]$. Similarly to the previous case, there is a bijection in this case between $F_\gamma$ and ${ (x+y-1)/2 \choose x/2, (y-1)/2}$. 
\end{enumerate}
Since $|F_\gamma|$ agrees in all cases with 
\[ \left| { [\frac{x+y}{2}] \choose [\frac{x}{2}], [\frac{y}{2}] } \right| \]
the proposition is proved.
\end{proof}

\begin{pr} \label{pr:DiagIso2}
  Let $X \weq S^{n+ q\alpha}$ be a motivic sphere, suppose $m$ is a nonnegative integer and $r\in \{0,1\}$. Let $P$ be a
  set of primes, and let $m$ be a natural number such that $m + 1 + me_{n,q}$ is a unit in $\GW(k) \tensor_\ZZ \ZZ_{P}$,
  then the map $\Delta^2_{2m+r, (2m,r)} : X^{\wedge 2m+r} \to X^{\wedge 2m} \wedge X^{\wedge r}$ is an isomorphism in
  $\ho_{P, \Aone} \Spt(\Sm_k)$.

  In particular, $\Delta^2_{2m+r, (2m,r)}$ is an isomorphism in $\ho_{2, \Aone} \Spt(\Sm_k)$.
\end{pr}
\begin{proof}
  We have calculated $\Delta^2_{2m+r, (2m, r)}$ in Corollary \ref{co_delta=sum_esignsigma} and Proposition
  \ref{pr:evenSignForDelta}. If $r=0$, we find $\Delta^2_{2m, (2m, 0)} = 1$ in $\GW(k) \tensor_\ZZ \ZZ_{(2)}$. 

  When $r=1$, we find $\Delta^2_{2m+1, (2m, 1)} = m+1 + m e_{n,q}$, from which the first claim follows immediately.

The element $m+1 + me_{n,q}$ is a unit in $\GW(k) \tensor_\ZZ \ZZ_{(2)}$ by Corollary \ref{co:2localUnits}.
\end{proof}

The same calculations, referring to Corollary \ref{co:GWUnits}, show the following:
\begin{pr} \label{pr:DiagIsoE}
  Let $X \weq S^{n+ q\alpha}$ be a motivic sphere. Assume one of the following two conditions holds:
  \begin{enumerate}
  \item $n$ is odd and $q$ is even,
  \item $n+q$ is odd and the ground field $k$ is not formally real.
  \end{enumerate}
  Suppose $m$ is a nonnegative intger and $r \in \{0,1\}$ . The
  diagonal map $\Delta^2_{2m+r, (2m,r)} : X^{\wedge 2m+r} \to X^{\wedge 2m} \wedge X^{\wedge r}$ is an isomorphism in
  $\ho_{\Aone} \Spt(\Sm_k)$.
\end{pr}

\begin{df} \label{def:totalOrderMultinom}
We impose a total order on the elements of ${ a_1 + \dots + a_m \choose a_1, a_2, \dots, a_m }$ by declaring $\sigma < \sigma'$ if $\sigma(j) =
\sigma'(j)$ for all $j \le k$ and $\sigma(k) < \sigma'(k)$.  
\end{df}

Following \cite{Kuhn_87}, define a regular $(r,s)$--set of size $m$ to be a set, $\{S_1, \dots, S_s\}$, of subsets of
$\{1, \dots, m\}$ satisfying
\begin{enumerate}
\item $|S_i| = r$ for all $i$
\item $\bigcup_{i=1}^s S_i = \{ 1, \dots, m\}$.
\end{enumerate}
Let $L(r,s,m)$ denote the set of all regular $(r,s)$ sets of size $m$. This goes by the name $B(B(m,s), r)$ in
\cite{Kuhn_87}, and the discussion that follows here is a much reduced version of the discussion to be found there. In
particular, we concentrate on the case where $r=2$ and $m=2s$.

There is a cover $\phi: {2s \choose 2,\dots, 2} \to L(2,s,2s)$, the source being the set of ordered partitions of $\{1,
\dots, 2s\}$ into $s$ disjoint subsets of cardinality $2$, and the latter be the set of unordered partitions. There is an
$S_s$-action on functions $\sigma: \{1 , \dots, 2s \} \to \{1 ,\dots, s\}$ induced from the action on the target; and
the orbits of this action are in bijective correspondence with $L(2,s,2s)$. For any $\lambda \in L(2,s,2s)$, define $\sgn(\lambda)$
to be $\sgn(\sigma)$ where $\sigma$ is the least element, in the order of Definition \ref{def:totalOrderMultinom}, of ${
  2s \choose 2,\dots, 2}$ that maps to $\lambda$ under $\phi$.

Write $E(2,s)$ and $O(2,s)$ for the number of elements in $L(2,s,2s)$ having even and odd sign,
respectively. Trivially, $E(2,1) = 1$ and $O(2,1) = 0$.

\begin{pr} \label{pr:EOcalc}
The quantities $E(2,s)$ and $O(2,s)$ satisfy $E(2,s) = O(2,s) + 1$.
\end{pr}
\begin{proof}
  An element $\lambda \in E(2,s)$ is a partition of $\{1, \dots, 2s\}$ into $s$ disjoint subsets $\{S_1, \dots,
  S_s\}$. One orders these subsets in ascending order of their least members. The quantity $\sgn(\lambda)$ is the number
  of pairs of numbers $j_1 < j_2$ such that $j_1 \in S_{\ell_1}$ and $j_2 \in S_{\ell_2}$ with $S_{\ell_1} >
  S_{\ell_2}$. We can set up an involution $\bar \gamma$ on $L(2,s,2s)$ by observing that $\gamma$ descends from ${ 2s
    \choose 2,\dots, 2}$.

  Explicitly, if $\lambda = \{S_1, \dots, S_s\} \in L(2,s,2s)$ is not the partition $\lambda_0 = \{\{1,2\}, \{3,4\},
  \dots, \{2s-1, 2s\}\}$ then there is a least pair of integers $(2i-1, 2i)$ such that $2i-1$ and $2i$ lie in different
  sets $S_i$, $S_j$ with $i \neq j$. Let $\bar \gamma(\lambda)$ be the partition obtained from $\lambda$ by
  interchanging $2i-1$ and $2i$. The exceptional partition, $\lambda_0$, is the unique fixed point of $\bar
  \gamma$. 

Observe that if $\lambda \neq \lambda_0$, then $\sgn(\lambda) + \sgn(\bar\gamma(\lambda)) \equiv 1 \pmod 2$.  Since $\sgn(\lambda_0) = 0$ is even, it follows that $E(2,s) - 1
 = O(2,s)$, as asserted.\end{proof}

\begin{remark}
  The cardinality of ${2s \choose 2,\dots, 2}$ is $\dfrac{(2s)!}{2^s}$, and that of $L(2,s,2s)$ is $\dfrac{(2s)!}{s! 2^s}$. Explicitly therefore
  \[  E(2,s) =  \frac{1}{2}\left(\frac{(2s)!}{ s! 2^s} - 1 \right) + 1 = \frac{(2s)!}{s! 2^{s+1}} + \frac{1}{2} \]
  and
  \[  O(2,s) =  \frac{1}{2}\left(\frac{(2s)!}{s! 2^s} - 1 \right)  = \frac{(2s)!}{s! 2^{s+1}} - \frac{1}{2}. \]
  The quantity $E(2,s) + O(2,s) = \dfrac{(2s)!}{s! 2^s}$ is the product of the first $s$ odd integers: $(2s-1)(2s-3)\dots(5)(3)(1)$. The fact that this is a unit in
  $\ZZ_{(2)}$ appears in the classical study of $j_2$.
\end{remark}

\subsection{Decomposing the second James--Hopf map}

\begin{df}\label{air_def}
Let $a^2_{i,s} : D_i (X) \to D_s (X^{\wedge 2}) $ denote the composition in $\ho_\Nis \soneSpt$
\[ \xymatrix{D_i X \ar[r] & D(X) \ar[r]^{\sigma^{-1}} & J(X)_{+} \ar[r]^{j_2} & J(X^{\wedge 2})_{+} \ar[r]^{\sigma} & D(X^{\wedge 2}) \ar[r] & D_s(X^{\wedge 2})} .\]
\end{df}

For example, we have \begin{equation}\label{a22=id}a^2_{2,1} = \id_{D_2(X)}\end{equation} by the commutative
diagram 
\[ \xymatrix{ J(X) \ar[r]^{j_2} & J(X^{\wedge 2}) \\ J_2(X) \ar[u] \ar[r] & X^{\wedge 2} \ar[u] } \]
where the lower horizontal map is the composite
\[ J_2(X) \to J_2 (X) /J_1(X) = D_2 (X) \cong X^{\wedge 2} \stackrel{\id} \to X^{\wedge 2}
\cong D_1(X^{\wedge 2}). \]

\begin{pr} \label{pr:classOfa}
  Let $X \weq S^{n+q\alpha}$ be a motivic sphere with $n \geq 1$. Let $i \ge 2$ be an integer. The maps in $\ho_\Nis \Spt(\Sm_k)$
  \[ \xymatrix{ D_i(X) \ar@/^2ex/^{a^2_{i,s}}[rrr] \ar[r] & J(X) \ar_{j_2}[r] & J(X^{\wedge 2}) \ar[r] & D_s (X^{\wedge 2}) } \]
  agree in $\ho_\Aone\Spt(\Sm_k)$ with
  \[ a^2_{i,s} = \begin{cases} E(2,s) + O(2,s)e_{n,q} \text{ \quad if $i=2s$,} \\
    \pt \text{\quad otherwise.} \end{cases}\]
\end{pr}
\begin{proof}
  The case where $i \neq 2s$ follows from Corollary 6.3(1) and (3) of  \cite{Kuhn_87}. As with the results of
  \cite{kuhn2001}, the arguments here all yield natural homotopies of maps, and therefore the results carry over from
  the case of spaces to the case of simplicial presheaves.

  When $i = 2s$, then by Theorem 6.2 of \cite{Kuhn_87}, the class $a^2_{i,s}$ is equal to the sum of the classes of permutations of
  $X^{\wedge 2s}$ associated to regular $(2,s)$ sets of size $2s$. Of these, $E(2,s)$ are even permutations, and
  therefore equivalent to the identity, and $O(2,s)$ are odd, and therefore equivalent to the single interchange $e_{n,q}$.
\end{proof}

\begin{co}  \label{co:2a}
  The map $a^{2}_{2s,s}$ is an isomorphism in $\ho_{2, \Aone}(\Spt(\Sm_k))$.
\end{co}
\begin{proof}
  The map in question is $E(2,s) + O(2,s)e_{n,r}$. Since $E(2,s) = O(2,s)+1$ by Proposition \ref{pr:EOcalc}, it follows
  from Corollary \ref{co:2localUnits} that it is a unit in $\GW(k) \tensor_\ZZ \ZZ_{(2)}$.
\end{proof}

Similarly, we have the following corollary:

\begin{co} \label{co:pa}
  If $X= S^{n+q\alpha}$ in $\GW(k)$ is a motivic sphere and one of the following conditions is satisfied:
  \begin{enumerate}
  \item $n$ is even and $q$ is odd,
  \item $n+q$ is odd and the field $k$ is not formally real
  \end{enumerate}
  then $a^2_{2s,s}$ represents an isomorphism in $\ho_{\Aone}(\Spt(\Sm_k))$.
\end{co}
\begin{proof}
  This follows from Proposition \ref{pr:EOcalc} and Corollary \ref{co:GWUnits}.
\end{proof}

\subsection{The Stable Weak Equivalence}\label{subsection:stable_weak_equivalence}

Let $X$ be of the form $S^{n+ q\alpha}$ for $n \ge 1$, $q \ge 2$. Let $e_{n,q}$ be the class $(-1)^{n+q}\langle -1
\rangle^q$ in $\GW(k)$.

Fix a set of primes, $P$.  All objects and maps in this section belong to the category $\ho_{P, \Aone} \Spt(\Sm_k)$,
unless otherwise stated. If $P$ is the set of all primes, then $\ho_{P, \Aone} (\Spt(\Sm_k)) = \ho_{\Aone}
(\Spt(\Sm_k))$. This case and the case $P=\{(2)\}$ are the two cases that are applied in subsequent sections of this paper.

Write $b_+:  J(X)_+ \iso \bigvee_{i=0}^\infty  X^{\wedge i} \to   X_+$ for the projection map.

We will need the following construction again in the sequel, so we present it here for later reference.
\begin{construction} \label{cons:added_in_process} Suppose given an unstable map $j: J \to Y$ in $\Spaces_\pt$ and a
  stable map $b: J \to X$ in $\ho_{P,\Aone}(\Spt(\Sm_k))_\pt$, that is to say a homotopy class of maps $b:
  \Sigma^\infty J \to \Sigma^\infty X$ . We produce a stable map $(j \wedge b_+) \circ \Delta_+: J_+ \to (X \times Y)_+$
as follows.

  We may extend $b: J \to X$ and $j:J \to Y$ to maps $b_+: J_+ \to X_+$ in $\ho_{P,\Aone}(\Spt(\Sm_k))_\pt$ and $j_+:
  J_+ \to Y_+$ in $\Spaces_\pt$. Then we take the smash product of these two maps. For convenience, we note that $j
  \vee \id$ is a map in the unstable homotopy category, so this may be carried out in an elementary way without recourse
  to a smash product of spectra. This gives a map $b_+ \wedge j_+ : J_+ \wedge J_+ \to X_+ \wedge Y_+$. But the source and
  target of this map may be identified with $(J \times J)_+$ and $(X \times Y)_+$ respectively. Then precomposing with
  the diagonal map $J \to (J \times J)_+$ gives the result.
\end{construction}

By means of the above, we construct stable homotopy class of maps $c=(j_+ \wedge b_+) \circ \Delta_+$
\[ \xymatrix@C=30pt{  J(X)_+ \ar_{\Delta}[r]  \ar@/^20px/^{c}[rrr] &  \left(J(X)
    \times J(X)\right)_+ \ar[r] &   \left( J(X) \times X \right)_+
  \ar[r] &  \left(J(X^{\wedge 2}) \times X \right)_+.} \]
Here $\Delta$ is the image in $\ho_{2, \Aone} \Spt(\Sm_k)$ of the diagonal map $J(X)_+ \to (J(X)  \times J(X))_+$ in
$\Spaces_\pt$, and $j$ is the James--Hopf map $j: J(X) \to J(X^{\wedge 2})$ in $\Spaces$. Since $j_+$ is a map in
$\Spaces_\pt$, we can form the product map $(J(X) \times X)_+ \to (J(X^{\wedge 2} \times X)_+$ in $\ho_\Aone \soneSpt(\Sm_k)$ by means
of the action of $\Spaces_\pt$ on $\soneSpt$.

Both $ J(X)_+$ and $ \left( J(X^{\wedge 2}) \times X \right)_+$ are isomorphic in the homotopy category $\ho_{2, \Aone} \Spt(\Sm_k)$ to the
spectrum $\bigvee_{i=0}^\infty X^{\wedge i}$. To see the latter, decompose
  \begin{equation}
    \label{eq:4}
    \begin{split}
     \left( J(X^{\wedge 2}) \times X \right)_+ \iso  S^0 \vee  J(X^{\wedge 2}) \vee   X \vee
    \left( J(X^{\wedge 2}) \wedge X  \right)  \iso \\ \iso  S^0  \vee   \Big( \bigvee_{i=1}^\infty 
      X^{\wedge 2i} \Big) \vee  X \vee \Big( \bigvee_{i=1}^\infty X^{\wedge 2i + 1} \Big).
    \end{split}
  \end{equation}
Use the above to fix a standard isomorphism $\bigvee_{i=0}^\infty X^{\wedge i} \iso  \left( J(X^{\wedge 2}) \times X
\right)_+$ in $\ho_\Aone \soneSpt$; an isomorphism $J(X)_+ \iso \bigvee_{i=0}^\infty X^{\wedge i}$ already having been
fixed in the form of the stable map $s$.

\begin{pr} \label{pr:stableIsomorphism}
  Fix a sphere $X = S^{n+q\alpha}$. If elements $m+1 + me_{n,q}$, where $m$ is an integer,
  are units in $\GW(k) \tensor_\ZZ \ZZ_P$, then the map $c$ is a weak equivalence.
\end{pr}
\begin{proof}
   Consider the ring
  \[ R= \End_{\ho_{2, \Aone} \Spt(\Sm_k)}\Big( \bigvee_{i=0}^\infty  X^{\wedge i} \Big). \]
  We wish to show that $(j_+ \times b_+)\circ \Delta$ is a unit of this ring.

  We may write
  \[ R = \prod_{i=0}^\infty \Hom_{\ho_{2, \Aone}\Spt(\Sm_k)} \Big( X^{\wedge i} , \bigvee_{l=0}^\infty X^{\wedge l} \Big) \] and
  \[\bigvee_{l=0}^\infty  X^{\wedge l} \weq \bigvee_{l=0}^i  X^{\wedge l} \vee \bigvee_{l=i+1}^\infty  X^{\wedge l}.\]
  It follows from the Hurewicz theorem that $[ X^{\wedge i} , \bigvee_{l=i+1}^\infty  X^{\wedge l}] = 0$, and so
  \[ \Hom_{\ho_{2, \Aone} \Spt(\Sm_k)} \Big(  X^{\wedge i} , \bigvee_{l=0}^\infty  X^{\wedge l} \Big) =
\bigoplus_{l=0}^i \Hom_{\ho_{2, \Aone} \Spt(\Sm_k)} \left(  X^{\wedge i } ,  X^{\wedge l} \right),\] 

so that $R = \prod_{i=0}^\infty \bigoplus_{l=0}^i \pi_{in + iq\alpha}(S^{ln + lq\alpha})$. We may represent elements of $R$ as
infinite, upper-triangular matrices $(d_{i,l})$ such that $d_{i,l} \in \pi_{in + iq \alpha}(S^{ln + lq \alpha})$ by
decreeing $d_{i,l} = 0$ whenever $i < l$. It follows from the usual algebra of matrix multiplication that an element of
$R$ is a unit if and only if the terms $d_{i,i} \in \pi_{in + iq\alpha}(S^{in + iq\alpha})$ are units for all $i$.

  The invertibility of $c$ in $R$ may be deduced from the classes $d_{i,i}$ appearing in this
  diagram
  \begin{equation} \label{eq:diag0}\xymatrix@C=60pt{   J(X)_+ \ar_{\Delta}[r]  \ar@/^20px/^{c}[rrr] &  \left(J(X) \times J(X)\right)_+ \ar[r] &   \left( J(X) \times X\right)_+
  \ar[r] &  \left(J(X^{\wedge 2}) \times X \right)_+  \ar[d] \\
       X^{\wedge i} \ar^{d_{i,i}}[rrr] \ar[u] & & &    X^{\wedge i}, } \end{equation} where the unmarked arrows are inclusion and
    projection maps.

    We can factor $d_{i,i}$ in diagram \eqref{eq:diag0} as
  \begin{equation} \label{eq:diag1}\xymatrix@C=60pt{   J(X)_+ \ar_{\Delta}[r]  \ar@/^20px/^{c}[rrr] &  \left(J(X)
    \times J(X) \right)_+ \ar[r] &   \left( J(X) \times X \right)_+ \ar[r] &  \left(J(X^{\wedge 2}) \times X \right)_+  \ar[d] \\
       X^{\wedge i} \ar^{f} [r] \ar[u] \ar@/_3ex/_{d_{i,i}}[rrr]& \bigvee_{n=0}^i X^{\wedge i-n} \wedge X^{\wedge n } \ar[u] \ar[rr] & &    X^{\wedge i}, } \end{equation}
   where $f$ is the wedge sum of maps $\Delta^2_{i, (i-n, n)}: X^{\wedge i} \to  X^{\wedge i-n} \wedge X^{\wedge n}$ as
   $n$ varies. This factorization follows from Proposition~\ref{pr:KuhnDelta}. 

   We can further factorize $d_{i,i}$ because the map $\ssp (J(X) \times J(X)))_+ \to \ssp (J(X) \times X)_+$
   is identity on the first and projection on the second factor:
    \begin{equation} \label{eq:diag2}\xymatrix@C=60pt{   J(X)_+ \ar_{\Delta}[r]  \ar@/^20px/^{c}[rrr] &  \left(J(X)
    \times J(X) \right)_+ \ar[r] &   \left( J(X) \times X \right)_+ \ar[r] &  \left(J(X^{\wedge 2}) \times X \right)_+  \ar[d] \\
       X^{\wedge i} \ar^{f} [r] \ar[u] \ar@/_3ex/_{d_{i,i}}[rrr]& \bigvee_{n=0}^i X^{\wedge i-n} \wedge X^{\wedge n }
       \ar[u] \ar[r] &  (X^{\wedge i} \wedge X^{\wedge 0}) \vee ( X^{\wedge i-1} \wedge X) \ar[u] \ar[r]& X^{\wedge
         i}, }.\end{equation}

   Write $i = 2m + s$ where $s \in \{0,1\}$. By use of Proposition \ref{pr:classOfa}, we deduce that the bottom row can be further factored as
   \begin{equation} \label{eq:diag3}\xymatrix@C=60pt{  
       X^{\wedge i} \ar_{f} [r] \ar@/^4ex/^{d_{i,i}}[rrr]& \bigvee_{n=0}^i X^{\wedge i-n} \wedge X^{\wedge n }
       \ar[r] & (X^{\wedge i} \wedge X^{\wedge 0}) \vee ( X^{\wedge i-1} \wedge X) \ar[d] \ar[r]&    X^{\wedge i} \\
      &  & X^{\wedge 2m} \wedge X^{\wedge s} \ar_{a^{2}_{2m,m} \wedge \id}[ur] }.\end{equation}
  It follows that $d_{i,i}$ factors as $(a^2_{2m,m} \wedge \id)\circ \Delta^2_{i, (2m,s)}$, and since both these maps
  are isomorphisms by virtue of Propositions \ref{pr:DiagIso2} and \ref{pr:classOfa}, so too is $d_{i,i}$, and therefore so too is $c= (j_+ \times b_+) \circ \Delta$.  
\end{proof}

\begin{remark} \label{rmk:OddEvenSphere}
  The hypothesis of the Proposition that elements of the form $(m+1) + me_{n,q} \in \GW(k) \tensor_\ZZ \ZZ_P$ be units
  holds in particular in the following cases:
  \begin{enumerate}
  \item The ring $\ZZ_P$ is $\ZZ_{(2)}$ or $\Q$. In this case, the hypothesis holds by Corollary \ref{co:2localUnits}.
  \item The integer $n$ is odd and the integer $q$ is even. In this case, $e_{n,q} = -1$,
    and the hypothesis holds by Corollary \ref{co:GWUnits}.
  \item The integer $n+q$ is odd, and the field $k$ is not formally real. Again, the hypothesis holds in this case by
    Corollary \ref{co:GWUnits}.
  \end{enumerate}
\end{remark}

\begin{remark}
If $X$ is an object in $\Spaces_\pt$, there is an action of the symmetric group $S_{n}$ on $X^{\wedge n}$. In the case
where $X$ is a motivic sphere, this action factors through the sign representation of $S_n$. The fact that $c$ is a
$2$-local weak equivalence depends on this fact, as we can see in the following example.
\end{remark}

\begin{example} \label{ex:SnSnexample}
Let $X$ be the simplicial set $X= S^2 \vee S^2$. The map $S_n \to [X^{\wedge n}, X^{\wedge n}]$ is injective because the action of $S^n$ on $\rH^{2n}(X^{\wedge n}, \Q) \iso \rH^2(X,
\Q)^{\tensor n} \iso \Q^{2^n}$ contains
a direct sum of two copies of the permutation representation of $S_n$ over $\Q$ as summands. These two copies can be described as follows. The wedge product $X^{\wedge n}$ is the direct sum of copies of $S^{2n}$ indexed by $n$-tuples of elements of $\{1,2\}$. The $n$-tuples which have a single $1$ and the rest $2$'s form one of the summands, and the other is obtained by switching the roles of $1$ and $2$.

The objects $J(X)_+$ and $J(X^{\wedge 2})_+$ split stably as $\bigvee_{i=0}^\infty X^{\wedge i}$
and$\bigvee_{i=0}^\infty X^{\wedge 2i}$ respectively. The second James--Hopf map 
\[ j_2 : \bigvee_{i=0}^\infty X^{\wedge i} \to \bigvee_{i=0}^\infty X^{\wedge 2i} \]
restricts to a map $a^2_{4, 2}: X^{\wedge 4} \to X^{\wedge 4}$. The paper \cite{Kuhn_87} calculates this map explicitly
as the sum of permutations $a^2_{4,2} = \sum_{\sigma \in {4 \choose 2,2}} e(\sigma)$. Note that ${4 \choose 2,2}$ is in bijection with $\{ ((1,2)(3,4)), ((1,3)(2,4)), ((1,4)(2,3))\}$ under the bijection sending $((a,b),(c,d))$ to the map sending $a$ and $b$ to $1$ and sending $c$ and $d$ to $2$. Using cycle notation for permutations, and representing the identity by $e$, this sum is
\[ a^2_{4,2} = e + (23) + (243). \] The induced map on the singular cohomology $\rH^8(X^{\wedge 4}, \Q) \iso \Q^{16}$ is
not of full rank, since $e + (23) + (243)$ is not an isomorphism on the permutation representation, namely on either of the submodules mentioned above $a^2_{4,2}$ acts by the matrix $$ \left( \begin{array}{cccc}
3 & 0 &0 & 0 \\
0 & 1 & 1 & 1\\
0 & 2 & 1 & 0 \\
0&0&1&2
\end{array} \right)$$ which has determinant $0$.

The map induced by $j_2$ on rational cohomology
$\rH^8(J(X^{\wedge 2}), \Q) \to \rH^8(J(X), \Q)$ is not an isomorphism in this case, and is in particular not injective,
and so the analogue of Proposition \ref{pr:stableIsomorphism} fails in this case, even $\Q$-locally.

\begin{figure}[h] 
 
  \centering
  \begin{equation*}
    \xymatrix@R=2pt@C=12pt{
\rH^3(\hofib(j_2), \Q)\iso 0 \ar^{d_4=0}[dddrrrr] & & & & 0\\
\Q^2 & & & & \Q^8& & & \\
0 \\
\Q & 0 & 0 & 0 & \Q^4 & 0 & 0 & 0 & \Q^{16} \iso \rH^8(J(X^{\wedge 2}), \Q) }
  \end{equation*}
\caption{ The first four rows, nine columns of the $\mathrm{E}_2$-page of the Serre spectral sequence for $\rH^*(\cdot, \Q)$ associated to $\hofib(j_2) \to J(X) \overset{j_2}{\to} J(X^{\wedge 2})$.}
\label{fig:1}
 
\end{figure}

Moreover, associated to the fiber sequence
\[ \hofib(j_2) \to J(X) \overset{j_2}{\to} J(X^{\wedge 2}) \]
there is a Serre spectral sequence for rational cohomology, part of which is shown in Figure \ref{fig:1}. We have shown that the edge map $\rH^8(J(X^{\wedge 2}), \Q)
\to \rH^8(J(X), \Q)$ is not injective. Since the edge map is not injective, it is not the case that the spectral
sequence collapses at the $\mathrm{E}_2$ page, and since $\rH^*(J(X^{\wedge 2}), \Q)$ is concentrated in even degrees,
it follows that $\rH^*(\hofib(j_2), \Q)$ is not also concentrated in even degrees. In particular, $\hofib(j_2)$ does not
have the same rational cohomology as $X = S^2 \vee S^2$, showing that even the $\Q$-local version of the EHP sequence
does not hold for a general space $X$.
\end{example}

\section{Fiber of the James-Hopf map}\label{Section:Fiber_JH_map}

In Section \ref{subsection_cancelation_property}, we will have two fiber squences $F \to E \to Y$ and $X \to X \times Y \to Y $ with the same base $Y$ and a stable weak equivalence between the total spaces $E$ and $X \times Y$, which is compatible with the map to the base. We will show that in fact the fibers are stably weakly equivalent as well (Proposition \ref{bsmashf_we_implies_ba_we}). For this, it is natural to ask for a Serre spectral sequence, as the Serre spectral sequence gives a good way to measure the size of the total space of a fibration in terms of the size of the base and the fiber. Since the base spaces of the fibrations $f$ and $p$ are the same, and their total spaces are the same size, a Serre spectral sequence would give us a tool with which to attempt to `cancel off the base space' and conclude that the fibers have the same size. The purpose of the first part of this section is to show that enough of these ideas remain available in $\Aone$-homotopy theory. In Section \ref{subsection_spectral_sequence}, we construct a spectral sequence to substitute for the Serre spectral sequence. We develop needed properties in Section \ref{subsection_functoriality}, and in Section \ref{subsection_cancelation_property}, we show that the desired cancelation is possible.  

\subsection{A spectral sequence} \label{subsection_spectral_sequence}

Let $\ms$ be a left Bousfield localization of the global model structure on $\Spaces$. There is an associated stable
model structure on the category of $S^1$--specta, $\Spt(\Sm_k)$. See Section \ref{subsection:Spectra}. Let $\sHf_i: \Spt(\Sm_k) \to \cat{Sh}_\Nis$ be an
$\ms$--corepresentable functor, given by a spectrum $E$, so that $\sHf_i(F)$ is the Nisnevich sheaf
associated to the presheaf
\[ U \mapsto [ \ssp S^i \wedge E \wedge \ssp U_+ , F]_{\ms, s}. \]
We write $\sHf_i(X)$ for $\sHf_i(\ssp X)$ when $X$ is an object of $\Spaces_{\ast}$.

Since left Bousfield localization does not change which maps are cofibrations, the notions of global cofibration, Nisnevich local cofibration, $\Aone$-cofibration , and $\ms$-cofibration for $\sPreK$ are the same. For $X_1 \to X_2$ a cofibration with respect to these model structures, the cofiber $C$ is the push-out $C = \colim  \xymatrix{X_1 \ar[r] \ar[d] & X_2 .\\ \pt &}$ A sequence is said to be a cofiber (respectively fiber) sequence up to homotopy, if the sequence is isomorphic in the homotopy category to a cofiber (respectively fiber) sequence.

\begin{pr} \label{sHf_properties}
  The homology theory $\sHf_{\ast}$ has the following properties:
\begin{enumerate} 
\item $\sHf_i$ takes $\ms$ weak equivalences to isomorphisms.
\item \label{sHf_LES_cofib_stable_df} Given a cofibration $X_1 \to X_2$ with cofiber $C$ in $\ms$, there is a natural long exact sequence
  of sheaves of abelian groups 
  \[ \ldots \to \sHf_i X_1 \to \sHf_i X_2 \to \sHf_i C \to  \sHf_{i-1} X_1 \to\ldots \]
\item \label{Sigma_is_shift_sHf} As a special case of \eqref{sHf_LES_cofib_stable_df}, we see that $\sHf_i(\Sigma X) \weq \sHf_{i-1}(X)$ for $X \in \Spaces_{\ast}$. 
\end{enumerate}
\end{pr}

\begin{asm} \label{sHf_assumption} We assume that $\sHf_\ast$ satisfies two further axioms.
\begin{enumerate}
\item \textbf{Boundedness:} \label{sHf_negative=0} $\sHf_i(X) = 0$ for $i < 0$ for all objects $X$ in $\Spaces_{\ast}$.
\item \textbf{Compactness:} \label{sHf_filtered_colimits} $\colim \sHf_i(X_j) = \sHf_i(\colim X_j)$ for all filtered diagrams $\{X_j\}$ in $\Spaces_{\ast}$.
\end{enumerate}
\end{asm}

These axioms are satisfied by $\sHf_i=\pi_{i+j\alpha}^{s,\Aone}$ and $\sHf_i=\pi_{i+j\alpha}^{s,P, \Aone}$: we take
$\ms$ to be the $\Aone$ injective structure or the $P$-$\Aone$ injective structure respectively. In each case, $E = \ssp
\Gm^{\wedge j}$. The boundedness axiom follows from Lemma \ref{lem:sspPreserveA1conn}, the compactness from
Proposition \ref{pr:stableHoCommuteFilteredColimits} or Proposition \ref{pr:PlocalHoCommutesWithFilteredColimits}.

For $f: X \to Y$ a global fibration, we construct a spectral sequence $E^r_{ij} \Rightarrow \sHf_{i+j} X$. 

This spectral sequence for $\ms=\Aone$ or its $P$-localizations will have the property that it relates $\ms$-homotopy invariant information about the total space with $\ms$-homotopy invariant information about the fiber and more delicate information about the base.

Let $\Delta_{\leq n}$ be the full subcategory of $\Delta$ on the objects $\{0,1,\ldots, n\}$. The pointed {\em $n$-skeleton} $\sk_n: \sSet_{\ast} \to \sSet_{\ast}$ can be defined as the composite of the $n$-truncation functor  $\sSet_{\ast} \to \Fun(\Delta_{\leq n}^{\op}, \Set_{\ast})$ with its left adjoint. Given a simplicial presheaf $X$, define $\sk_n X \in \sPreK$ by $U \mapsto \sk_n X(U)$. For $n <0$, the definition gives $\sk_n X = \ast$. 

\begin{df}\label{global_sk_base_SS}
Let $f: X \to Y$ be a fibration in the global model structure between pointed simplicial presheaves. We define the following spectral sequence $$(E^r_{ij}, d^r: E^r_{i+r,j} \to E^r_{i,j+r-1}).$$ The cofibrations $$\sk_0 Y \to \sk_1 Y \to \ldots \to \sk_n Y \to \ldots Y $$ pull-back to cofibrations $$\sk_0 Y \times_Y X \to \sk_1 Y \times_Y X \to \ldots \to \sk_n Y \times_Y X \to \ldots Y\times_Y X = X.$$

The cofiber sequences $$ \sk_{n-1} Y \times_Y X \to \sk_n Y \times_Y X \to C_n$$ for $n \geq 0$ give rise to the long exact sequences of \eqref{sHf_LES_cofib_stable_df} Proposition \ref{sHf_properties}, which form an exact couple $$\xymatrix{ \oplus_{i,n=0}^{\infty} \sHf_i (\sk_n Y \times_Y X) \ar[rr] &&  \oplus_{i,n=0}^{\infty} \sHf_i (\sk_n Y \times_Y X) \ar[dl] \\ & \ar@{.>}[ul]  \oplus_{i,n=0}^{\infty} \sHf_i C_n } $$

This exact couple gives rise to the spectral sequence $E^r_{i,j}$ with $E^1_{i,j} = \sHf_{i+j} C_i$.
\end{df}

We now relate $C_i$ to the fiber of $f$. Assume for simplicity that $Y$ is reduced in the sense that $Y_0 = \ast$, and let $F$ denote the fiber of $f$ over $Y_0$. 

For $U \in \Sm_k$, let $L_n Y(U) \in \sSet_{\ast}$ denote the $n$th latching object, defined $L_n Y(U) = (\sk_{n-1} Y(U))_n$, and let $N_n Y(U)$ be the set of non-degenerate $n$-simplices of $Y(U)$, defined $N_n Y(U) = Y_n(U) - L_n Y(U)$.

Despite the fact that $N_n Y (-)$ does not necessarily define a presheaf, $$\vee_{N_nY} (F_+ \wedge  (\Delta^n/\partial \Delta^n)) = U \mapsto \vee_{y \in N_n Y(U)}(
F(U)_+ \wedge (\Delta^n/\partial \Delta^n)) $$ is a presheaf because it could equally well be
written 
\[\frac{\bigvee_{y \in Y_n(U)}(F(U)_+ \wedge (\Delta^n/\partial \Delta^n))}{ \bigvee_{y \in L_n Y(U)}
  (F(U)_+ \wedge (\Delta^n/\partial \Delta^n))},\] and both $Y_n$ and $L_n Y$ are presheaves.
 This presheaf is weakly equivalent in the global model
structure to the cofiber $C_n$ as shown by the following lemma.

\begin{lm}\label{C_i_as_wedge}\label{lemma_formula_C_i(U)}
There is a global weak equivalence $C_n \simeq \vee_{N_n Y}( F_+ \wedge (\Delta^n/\partial \Delta^n))$ in $\sPreK_{\ast}$.
\end{lm}

\begin{proof}
Let $\partial \Delta^n$ denote the boundary of $\Delta^n$. By \cite[VII Proposition 1.7 p. 355]{Goerss_Jardine_Simplicial_Homotopy}, there is a push-out \begin{equation}\label{sk_n_from_n-1_push-out_square_new}\xymatrix{(Y_n \times \partial \Delta^n )\cup_{L_n Y \times \partial \Delta^n}( L_n Y \times \Delta^n) \ar[d] \ar[r] & \sk_{n-1} Y \ar[d] \\ Y_n \times \Delta^n \ar[r] & \sk_n Y}.\end{equation} The pull-back of a push-out square of simplicial sets is a push-out square because small colimits are pull-back stable in the topos of simplicial sets. Since limits and colimits in $\sPreK_{\ast}$ commute with taking the sections above $U \in \Sm$, it follows that the pull-back of a push-out square in $\sPreK_{\ast}$ is also a push-out square. Thus applying the functor $(-)\times_Y X$ to \eqref{sk_n_from_n-1_push-out_square_new} produces a push-out square, which then produces a global weak equivalence between $C_n$ and the cofiber of \begin{equation}\label{Cnfromsimplices}((Y_n \times \partial \Delta^n )\times_Y X) \cup_{(L_n Y \times \partial \Delta^n) \times_Y X }(( L_n Y \times \Delta^n)\times_Y X )\to (Y_n \times \Delta^n)\times_Y X.\end{equation}

The composition $Y_n \times F \to F \to X$ of the projection with the inclusion determines a map $Y_n \times F \to X$. The product of the identify map on $Y_n \times F$ with inclusion of the $0$th vertex into $\Delta^n$ gives a map $Y_n \times F \to Y_n \times F \times \Delta_n$.   The canonical map $Y_n \times \Delta^n \to Y$ factors $Y_n \times \Delta^n \to \sk_n Y \to Y$, and precomposing with the projection $Y_n \times F \times \Delta^n \to Y_n \times \Delta^n$, we obtain maps $Y_n \times F \times \Delta^n \to Y_n \times \Delta^n \to \sk_n Y \to Y$ .  These maps fit into the commutative diagram $$\xymatrix{ Y_n \times F \ar[r] \ar[d] & (Y_n \times \Delta^n) \times_Y X \ar[d] \ar[r]& \sk_n Y \times_Y X \ar[r] \ar[d] &X \ar[d]\\ Y_n \times F \times \Delta^n \ar[r] \ar@{.>}[ur]& Y_n \times \Delta^n \ar[r]& \sk_n Y \ar[r] &Y}$$ formed by the solid arrows. 

Since $Y_n \times F \to  Y_n \times F \times \Delta^n$ is a global trivial cofibration, this commutative diagram extends to include a map denoted by the dotted arrow, by the lifting property of global trivial cofibrations and global fibrations. The dotted arrow is a global weak equivalence because it is a map of global fibrations over $Y_n \times \Delta$ such that the induced map on the fibers is a global weak equivalence \begin{equation}\label{gweFDeltanYn} \xymatrix{ & (Y_n \times \Delta^n) \times_Y X \ar[d] \\Y_n \times F \times \Delta^n \ar[r] \ar@{.>}[ur]& Y_n \times \Delta^n .}\end{equation} 

Pulling back the diagram~\eqref{gweFDeltanYn} by a map $A \to Y_n \times \Delta^n$ produces a map of global fibrations over $A$ such that the induced map on fibers is a global weak equivalence, and it follows that the pullback of the dotted arrow remains a global weak equivalence as well. We apply this to the canonical maps from $A =Y_n \times \partial \Delta^n$, $A = L_n Y \times \partial \Delta^n$, and $A = L_n Y \times \Delta^n$. Since the union in the domain of the map \eqref{Cnfromsimplices} is a homotopy pushout as well as a pushout, we obtain a diagram $$\xymatrix{ ((Y_n \times \partial \Delta^n )\times_Y X) \cup_{(L_n Y \times \partial \Delta^n) \times_Y X }(( L_n Y \times \Delta^n)\times_Y X )\ar[rrrr] \ar[d] &&&& (Y_n \times \Delta^n)\times_Y X \ar[d] \\ (Y_n \times \partial \Delta^n\times F) \cup_{L_n Y \times \partial \Delta^n \times F }(L_n Y \times \Delta^n \times F ) \ar[rrrr] &&&& Y_n \times \Delta^n \times F }$$ where the vertical arrows are weak equivalences. Since the horizontal arrows are monomorphisms whence global cofibrations, we obtain an induced weak equivalence between the cofibers, proving the lemma.
\end{proof}

It is convenient to introduce notation for the presheaf $\vee_{N_iY} F_+$, so we do that now.

\begin{df}
For a global fibration $f: X \to Y$ in $\sPreK_{\pt}$ such that $Y_0 =\ast$, define $K_i$ in $\sPreK_{\pt}$ to be $$K_i = \vee_{N_iY} F_+ $$ where $F$ is the fiber of $f$.\hidden{Notation for $K_i$ was originally $D_i$, but since this clashed with section 6, it was changed.}
\end{df}

Lemma \ref{C_i_as_wedge} shows that there is a global weak equivalence $C_i \simeq S^i \wedge K_i$. 

\begin{pr}\label{properties_of_SS_7.1} Let $\{ E^r_{ij}, d^r_{ij}\}$ denote the spectral sequence of Definition \ref{global_sk_base_SS} associated to a global fibration $f: X \to Y$ in $\sPreK_{\pt}$ such that $Y_0 =\ast$.
\begin{enumerate}
\item \label{E1_from_F} There is a canonical isomorphism $E^1_{i,j} \cong \sHf_j K_i$
\item \label{SS_converges} $E^r_{ij} = 0$ for $i$ or $j$ less than $0$. 
\item \label{SS_converges_to_piA1s} This spectral sequence converges to the values of the functors $\sHf_{\ast}$ on $X$
$$(E^r_{ij}, d^r: E^r_{i+r,j} \to E^r_{i,j+r-1}) \Rightarrow \sHf_{i+j} X.$$ 
\end{enumerate}
\end{pr}

\begin{proof} We prove the claims in order.
\begin{enumerate}
\item Since there is a global weak equivalence $C_i \simeq S^i \wedge K_i$ (Lemma \ref{C_i_as_wedge}), it follows that $$E^1_{i,j} = \sHf_{i+j} C_i \cong \sHf_{i+j} S^i \wedge K_i \cong \sHf_{j} K_i$$ by \eqref{Sigma_is_shift_sHf} of Proposition \ref{sHf_properties}.
\item The claim is immediate for $i<0$. We show that $E^1_{ij} = 0$ for $j <0$, which is sufficient because $E^r_{ij}$ is a subquotient of $E^1_{i,j}$. By \eqref{E1_from_F}, $E^1_{i,j} \cong \sHf_{j}  K_i $. For $j <0$, we have $\sHf_{j}  K_i =0$ by \eqref{sHf_negative=0} of Assumption \ref{sHf_assumption}. 
\item The convergence follows from \eqref{SS_converges}. Since $\colim_n \sk_n Y = Y$ and finite limits commute with filtered colimits, it follows that $\colim_n \sk_{n} Y \times_Y X  \cong X$. Since $\sHf_{j+i}$ is preserves filtered colimits (Assumption \ref{sHf_assumption} \eqref{sHf_filtered_colimits}), $\colim_n \sHf_{j+i}(\sk_{n} Y \times_Y X) \cong \sHf_{j+i}(X)$, proving the claim.
\end{enumerate}
\end{proof}

We can summarize this subsection as follows: For a map $f: X \to Y$ in $\sPreK_{\pt}$ such that $Y_0 =\ast$, we have a spectral sequence $$(E^r_{ij}, d^r: E^r_{i+r,j} \to E^r_{i,j+r-1}) \Rightarrow \sHf_{i+j} X$$ satisfying the properties of Proposition \ref{properties_of_SS_7.1}, where $F = \hofib_{\textrm{global}}f$ is the homotopy fiber in the global model structure of $f$, i.e., factor $f$ as the composition $f = \iota \circ f'$, with $\iota: X \to Z$ a global trivial cofibration and $f': Z \to Y$ a global fibration. Let $F$ be the fiber of $f'$ over the basepoint. The spectral sequence is that of Definition \ref{global_sk_base_SS} applied to $f'$. 

\begin{remark}
Note that the construction of this spectral sequence \hidden{(i.e. the spectral sequence of Definition \ref{global_sk_base_SS})} as well as those of its properties given in this section only require the lifting properties of global fibrations. The subtler lifting properties for $\ms$-fibrations have not been exploited.
\end{remark}

\subsection{A functoriality property}\label{subsection_functoriality}

We will use a functoriality property of the spectral sequence constructed in Section \ref{subsection_spectral_sequence}
with respect to a particular sort of stable map in the homotopy category. 

Recall that $\ms$ is a left Bousfield localization of the global injective model structure. 

\begin{pr}\label{needed_model_cat_properties_pr}
The model structure $\ms$ on $\Spaces$ satisfies the following properties:
\begin{enumerate}
\item \label{sPre_ms_monomorphisms_cofibrations}\label{sPre_ms_objects_cofibrant} All monomorphisms in $\sPreK$ are $\ms$
  cofibrations, and in particular, objects of $\sPreK$ are $\ms$ cofibrant.
\item \label{sPre_ms_simplical_model_category} $\ms$ is compatible with the tensor,
  cotensor and simplicial enrichment in the sense that this structure makes $\sPreK$ into a simplicial model category
  \cite[Definition 11.4.4]{Riehl}.  
\item \label{Sigma_preserve_ms_we}  $\Sigma: \sPreK \to \sPreK$ takes $\ms$ weak equivalences to $\ms$ weak
  equivalences.
\item \label{ms_left_quillen_stable} There is a left Quillen functor $\ssp : \Spaces_+ \to \Spt(\Sm_k)$, where
  $\Spt(\Sm_k)$ is endowed with a stable model structure which we also call $\ms$, in an abuse of notation.
\end{enumerate}
\end{pr}

\begin{proof}
  Property \eqref{sPre_ms_monomorphisms_cofibrations} is immediate: $\ms$ is a left Bousfield localization of the global
  injective model structure, and the same property holds there. Property \eqref{sPre_ms_simplical_model_category}
  follows from \cite[Theorem 4.1.1(4)]{hirschhorn2003} because the global injective model structure is left proper,
  simplicial, and cellular (see \cite{hornbostel2006}).  Property \eqref{Sigma_preserve_ms_we} is a special case of
  \eqref{sPre_ms_simplical_model_category}. Property \eqref{ms_left_quillen_stable} is Proposition \ref{pr:unstableStableAdjunction}.
\end{proof}

We furthermore make the following assumption.

\begin{asm}\label{products_preserve_ms_we}
If $X$ is an object of $\Spaces$, then $X \times \cdot$ preserves weak equivalences.
\end{asm}

To construct the EHP fiber sequence, we will use $\ms$ to be the $P$-localized $\Aone$-model structure for $P$ a set of primes. By Proposition \eqref{pro:symm_monoidal} and Corollary \ref{smash_preserve_ms_we} this choice is valid. 

We will employ the following construction, which is a version of Construction \ref{cons:added_in_process} that applies
to maps $b$ that exist after one suspension, rather than simply stably: we use the notation $\Sigma (-)_+$ to mean $\Sigma ((-)_+)$.
 \begin{construction} \label{cons:added_in_process2}
    Suppose given a map $j: J\to Y$ in $\Spaces_\pt$ and a map $b: \Sigma J \to \Sigma X$ in
    $\ho_{\ms}\Spaces_\pt$. We produce a map $\constr{b}{j}:\Sigma  J_+ \to \Sigma( X \times Y)_+$ as follows:

    We may extend $b: \Sigma J \to \Sigma X$ to a map in $\ho_{\ms} \Spaces_\pt$ 
    \[b: \Sigma J_+ \to \Sigma X_+\]
    since $\Sigma J_+ \weq \Sigma J \vee S^1$ and similarly for $X$.  We take the smash product of $b:  \Sigma J_+ \to \Sigma X_+$ and $j_+ : J_+ \to Y_+$, obtaining a map $\Sigma J_+ \wedge  J_+\to \Sigma X_+ \wedge Y_+$. The left
    hand side is identified with $\Sigma (J \times J)_+$ and the right with $\Sigma (X \times Y)_+$. Then we
    precompose with the diagonal map $\Sigma \Delta_+: \Sigma J_+ \to \Sigma(J \times J)_+$ to obtain a map in
    $\ho_{\ms} \Spaces_\pt$:
    \begin{equation}
      \label{eq:5}
      \constr{b}{j}: \Sigma J_+ \to \Sigma (X \times Y)_+
    \end{equation}

    This construction is functorial in the map $j$ as follows. Suppose given a commutative square
    \[ \xymatrix{ J' \ar^p[r] \ar^{j'}[d] & J \ar^{j}[d] \\ Y' \ar^q[r] & Y }. \]
    Then there is a map $b \circ \Sigma p : \Sigma J' \to \Sigma X$, and the evident square
\[ \xymatrix@C=7em{ \Sigma J' \ar^{\constr{(b \circ \Sigma p)}{j}}[rr] \ar[d] && \Sigma  (X \times Y')_+ \ar[d]  \\ \Sigma J
  \ar^{\constr{b}{j}}[rr] && \Sigma  (X \times Y)_+} \]
is commutative.
\end{construction}

We remark that the map $\constr{b}{j} : \Sigma J \to \Sigma(X \times Y)$ constructed here agrees in the stable category
with Construction \ref{cons:added_in_process}.

We assume the following setup: a map $j : J \to Y$ in $\sPreK_\pt$ and a map $b: \Sigma J \to \Sigma X$ in $\ho_{\ms}
\sPreK_\pt$ with $X$ fibrant. We replace $j: J \to Y$ by a fibration by means of a canonical factorization $J \overset{\iota}{\to} E
\overset{f}{\to} Y$, where $\iota$ is a trivial cofibration and $f$ is a fibration.
 There is a well-defined map in the homotopy category
\begin{equation}
  \label{eq:3}
  b_E = b \circ \Sigma \iota^{-1} :\Sigma E \to \Sigma X.
\end{equation}
Let $p$
denote the projection $X \times Y \to Y$. Associated to each of the fibrations
\begin{align*}
  &f : E \to Y \\
  &p: X \times Y \to Y
\end{align*}
there are spectral sequences as in Definition \ref{global_sk_base_SS}. We denote these by $E^r_{ij}$ and $(E')^r_{ij}$
respectively. Let $$\mathfrak{b}:  \Sigma E \to
\Sigma(X \times Y)$$ denote the map formed from $\constr{b_E}{f} : \Sigma E_+ \to
\Sigma(X \times Y)_+$ (construction \ref{cons:added_in_process2}) by precomposing with a map $\Sigma E \to \Sigma E \vee S^1 \simeq \Sigma E_+$ and post composing with $\Sigma(X \times Y)_+ \to \Sigma (X \times Y)$. Since $\sHf_\ast$ is a stable theory, we obtain a map $\sHf_i \Sigma^{-1} \mathfrak{b} : \sHf_\ast(
J) \to \sHf_\ast(X \times Y)$.

\begin{lm}\label{Spectral_sequence_functoriality_wrt_(bEwedgeetc}
  With notation as above, $\mathfrak{b}$ induces a map of spectral sequences $E^r_{ij} \to (E')^r_{ij}$. The induced map
  $E^{\infty}_{ij} \to (E')^{\infty}_{ij}$, which by Proposition \ref{properties_of_SS_7.1} \eqref{SS_converges_to_piA1s} is a map from the
  associated graded of a filtration of $\sHf_{i+j}(E)$ to the associated graded of a filtration of $\sHf_{i+j}(X \times
  Y)$, is compatible with \[ \sHf_{i+j} \Sigma^{-1}\mathfrak{b}: \sHf_{i+j}(E) \to \sHf_{i+j}(X \times Y).\]
\end{lm}

\begin{proof}
  In order to construct the map of spectral sequences, we start with the skeletal filtration $\sk_\bullet Y$ of $Y$. For
  all $n$, there are maps $f_n: \sk_nY \times_Y E \to \sk_n Y$. We remark in passing that the fiber product $\sk_n Y
  \times_Y E$ is a homotopy fiber product by virtue of $f:E \to Y$ being a fibration. 
Moreover, there are composite maps $b_n: \Sigma (\sk_n Y \times_Y E ) \to \Sigma E \to \Sigma X$, where the first map is
induced by the inclusion of the skeleton and the second is $b$. 

By means of Construction \ref{cons:added_in_process2}, we obtain maps $\Sigma(\sk_n Y \times_Y E)_+ \to \Sigma(\sk_n Y \times
X)_+$ in the homotopy category, and by functoriality, these maps are compatible in that the diagram
\begin{equation}\label{sk_n-1_n_full_stable_map_compatibility_cd}\xymatrix{
    \Sigma (\sk_{n-1} Y \times_Y E)_+ \ar[d] \ar[r]& \Sigma (\sk_n Y
    \times_Y E)_+ \ar[d]\ar[r] & \Sigma E_+  \ar[d]\\ \Sigma (\sk_{n-1} Y \times X)_+\ar[r] & \Sigma (\sk_{n} Y \times X)_+\ar[r] &
    \Sigma (Y \times X)_+}\end{equation} commutes.

The commutative diagram \eqref{sk_n-1_n_full_stable_map_compatibility_cd} induces a commutative diagram
\begin{equation}\label{map_cofiber_sequences_filtration_in_Spectral_sequence}\xymatrix{ \Sigma (\sk_{n-1} Y
    \times_Y E) \ar[d]\ar[r]& \Sigma (\sk_n Y \times_Y E)
    \ar[d] \ar[r] & \Sigma C_n \ar[d]\\ \Sigma (\sk_{n-1} Y \times
    X)\ar[r] & \Sigma (\sk_{n} Y \times X) \ar[r] & \Sigma C_n'}
\end{equation}
in $\ho_{\ms} \sPreK_{\pt}$, where the horizontal rows are cofiber sequences and suspensions of cofiber sequences. Applying
$\sHf_{\ast} \Sigma^{-1}$ to the entire diagram then defines a morphism of long exact sequences, and thus a morphism of
exact couples, and therefore a morphism of spectral sequences. The compatibility with the induced map on $E^{\infty}$
pages follows from applying $\sHf_{\ast} \Sigma^{-1}$ to \eqref{sk_n-1_n_full_stable_map_compatibility_cd}.
\end{proof}

We can compute the map $E^1_{ij} \to (E')^1_{ij}$ of $E^1$-pages of the map of spectral sequences of Lemma
\ref{Spectral_sequence_functoriality_wrt_(bEwedgeetc}.  Suppose again that $Y\in \sPreK_\pt$ is such that
$Y_0 = \pt$. Let $a: F \to E$ denote the canonical map of simplicial presheaves given by the definition $F =
\hofib_{\textrm{global}}f = \pt \times_Y E \to E$.  Composing $\Sigma a_+$ with $b_E$ yields a map $b_E \circ \Sigma
a_+: \Sigma F_+ \to \Sigma X_+$ in $\ho_{\ms} \sPreK_\pt$.

Recall that $K_i \in \sPreK_\pt$ is defined by $K_i(U) = \vee_{N_iY (U)} (F(U)_+).$ The analogous definition for the global fibration $p$ is then $K_i' = \vee_{N_iY (U)} (X(U)_+)$. We claim that the map $b_E \circ \Sigma a_+: \Sigma F_+ \to \Sigma X_+$ in $\ho_{\ms} \sPreK_{\pt}$ defines a map $\Sigma K_i \to \Sigma K_i'$ in $\ho_{\ms} \sPreK_{\pt}$. This claim is established by the following lemma. 

\begin{lm}\label{wedge_over_PreSheaf_sets_preserves_Aone_w-e}\label{wedge_over_PreSheaf_sets_preserves_stable_Aone_w-e}\label{presheaf_I_to_N_nX(U)_we_version}
Suppose $I$ is a presheaf of sets on $\Sm_k$. \begin{enumerate}
\item \label{wedge_over_PreSheaf_sets_preserves_Aone_w-e:unstable} If $g: A \to B$ is $\ms$-weak equivalence in pointed spaces $\sPreK_{\pt}(\Sm_k)$, then $\vee_I g: \vee_I A \to \vee_I B$ is an $\ms$-weak equivalence. 
\item \label{wedge_over_PreSheaf_sets_preserves_Aone_w-e:stable} If $g: A \to B$ be an $\ms$-weak equivalence in spectra $\Spt(\Sm_k)$, then $\vee_I g: \vee_I A \to \vee_I B$ is an $\ms$-weak equivalence. 
\end{enumerate}
Furthermore, suppose that $Y \in \sPreK_{\pt}$ is pointed. Then we may replace the presheaf $I$ of sets on $\Sm_k$ by $U \mapsto N_n Y(U)$.
\end{lm}

\begin{proof}
\eqref{wedge_over_PreSheaf_sets_preserves_Aone_w-e:unstable}: $\vee_I g: \vee_I A \to \vee_I B$ is canonically identified with $I_+ \wedge g: I_+ \wedge A \to  I_+ \wedge B$, so $\vee_I g$ is a weak equivalence by Corollary \ref{smash_preserve_ms_we}.

\eqref{wedge_over_PreSheaf_sets_preserves_Aone_w-e:stable}: We again have a canonical identification of $\vee_I g: \vee_I A \to \vee_I B$ with $I_+ \wedge g: I_+ \wedge A \to I_+ \wedge B$. For any pointed simplicial sheaf $\mathcal{X}$, the functor $$\Spt(\Sm_k) \to \Spt(\Sm_k) $$ $$E \to E \wedge \mathcal{X} $$ preserves stable $\Aone$-weak equivalences, as in \cite[ \S 4 pg 27]{morel2005}.

Furthermore, note that we have a canonical bijection $Y_n(U)/ L_n Y (U) \cong N_n Y (U) \coprod \pt$ and that $U \mapsto Y_n(U)/ L_n Y (U)$ is a presheaf of (pointed) sets, which we will denote by $I$. There is a monomorphism $A \to \vee_I A$, and similarly for $B$. The map $U \mapsto \vee_{N_n Y (U)} g(U)$ is the map induced by taking cofibers in the diagram $$\xymatrix{ A \ar[r] \ar[d]_{g} & \vee_I A \ar[d]_{\vee_{I} g}\\ B \ar[r] & \vee_I B}.$$ Since the top and bottom horizontal arrows are cofibrations, the cofibers of these maps are also homotopy cofibers, so it follows that $\vee_{N_n Y (U)} g$ is an $\ms$-weak equivalence. 
\end{proof}

It follows that $b_E \circ \Sigma a_+$ in $\ho_{\ms} \sPreK_{\pt}$ induces a map $$\vee_{ N_n Y} (b_E \circ \Sigma a_+): \Sigma K_i \to \Sigma K_i'$$ in $\ho_{\ms} \sPreK_{\pt}$. As before, we may apply $\sHf_{j} \Sigma^{-1}$ to any map in $\ho_{\ms} \sPreK_{\pt}$.  Applying $\sHf_{j} \Sigma^{-1}$ to $\vee_{ N_n Y} (b \circ \Sigma a_+)$ gives our identification of the map of $E^1$-pages in the map of spectral sequences $E^r_{ij} \to (E')^r_{ij}$ in Lemma \ref{Spectral_sequence_functoriality_wrt_(bEwedgeetc}.

\begin{lm}
\label{E1_functoriality}
Suppose that $Y_0 = \pt$. The identification of $E^1_{i,j}$ with $ \sHf_j(K_i)$ in Proposition \ref{properties_of_SS_7.1} \eqref{E1_from_F} is functorial with respect to the map $$E^r_{ij} \to (E')^r_{ij}$$ in the sense that the induced map for $r=1$ is $\sHf_{j} \Sigma^{-1}(\vee_{ N_n Y} (b_E \circ \Sigma a_+)): \sHf_j(K_i) \to \sHf_j(K_i') $.
\end{lm}

\begin{proof}
By construction, the lemma is equivalent to the claim that the map $\sHf_{n+j} C_n \to \sHf_{n+j} C'_n$ given by applying $\sHf_{n+j} \Sigma^{-1}$ to the commutative digram \eqref{map_cofiber_sequences_filtration_in_Spectral_sequence} is identified with $\sHf_{n+j} \Sigma^{-1}$ applied to $$\vee_{\id_{N_nY}}(b_E \circ \Sigma a_+) \wedge S^n: \Sigma \vee_{N_nY} F_+ \wedge S^n \to \Sigma \vee_{N_nY} X_+ \wedge S^n$$ via the equivalences given in Lemma \ref{lemma_formula_C_i(U)}.

The equivalences of Lemma \ref{lemma_formula_C_i(U)} are constructed by choosing a trivialization $$\xymatrix{Y_n \times F \ar[r] \ar[d] & E\ar[d]^{f} \\ Y_n \times F \times \Delta^n  \ar[r] \ar@{.>}^{\phi}[ur]& Y},$$ and any two are homotopic, as either is equivalent in the homotopy category to the composition $$ Y_n \times F \times \Delta^n \stackrel{\cong}{\leftarrow} Y_n \times F \to E.$$ The product $X \times Y$ admits a canonical trivialization $$\xymatrix{Y_n \times X \ar[r] \ar[d] & X \times Y\ar[d]^{f} \\ Y_n \times X \times \Delta^n \ar[r] \ar[ur]& Y}.$$ 

Applying $\Sigma(-)_+$  to $Y_n \times F \to E$ and composing with the map $\constr{b_E}{f} : \Sigma E_+ \to
\Sigma(X \times Y)_+$ from Construction \ref{cons:added_in_process2} produces a map $$ \Sigma(Y_n \times F)_+ \to \Sigma E_+ \to  \Sigma (Y \times X)_+$$ which factors through the inclusion $$ \Sigma (Y_n \times X)_+ \to \Sigma (Y \times X)_+.$$ The resulting map  $$ \Sigma(Y_n \times F)_+ \to \Sigma E_+ \to  \Sigma (Y_n \times X)_+$$ is the smash product of the identity $1_{Y_n}: Y_n \to Y_n$ on $Y_n$ and $b_E \circ \Sigma a_+$. It follows that the map $\Sigma C_n \to \Sigma C_n'$ of \eqref{map_cofiber_sequences_filtration_in_Spectral_sequence} is identified with $\vee_{\id_{N_nY}}(b_E \circ \Sigma a_+) \wedge 1_{S^n}$, proving the lemma.

\hidden{Older proof with an explicit homotopy replacing the observation that the homotopy class of the trivialization is independent of choice:

Let $\Sigma p_X: \Sigma (X \times Y)\to \Sigma X$ denote $\Sigma$ applied to the projection $X \times Y \to X$.  For each element of $Y_n$, we have a map in $\ho_{\ms} \sPreK$ $$\Sigma p_X \circ (b_E \wedge  f_+) \circ \Sigma \Delta_+ \circ \Sigma \phi:\Sigma  (F \times \Delta^n) \to \Sigma X.$$ This gives rise to a map in $\ho_{\ms} \sPreK$ $$ \coprod_{Y_n} \Sigma  (F \times \Delta^n) \to \coprod_{Y_n} \Sigma X.$$

Choose a homotopy $H: I \times \Delta^n \to \Delta^n$ from the identity to the constant map whose value is the vertex $0$. Pulling back along the resulting map $$  \coprod_{Y_n} \Sigma (F  \times I\times \Delta^n) \to \coprod_{Y_n} \Sigma (F \times \Delta^n)$$ produces a map $$ \coprod_{Y_n} \Sigma(F \times \Delta^n \times I) \to   \coprod_{Y_n} \Sigma X$$ in $\ho_{\ms} \sPreK$.

By Corollary \ref{smash_preserve_ms_we}, we may smash the maps $\Sigma (F \times \Delta^n \times I)_+ \to \Sigma X_+$ with $\id_{\Delta^n_+}$ in $\ho_{\ms} \sPreK$. We obtain a map $$ \coprod_{Y_n} \Sigma (F  \times I\times \Delta^n \times \Delta^n)_+ \to \coprod_{Y_n} \Sigma(X\times \Delta^n)_+$$ in $\ho_{\ms} \sPreK$. Precomposing with the diagonal $\Delta^n \to \Delta^n \times \Delta^n$ yields a map $$\overline{H} : \coprod_{Y_n} \Sigma (F \times I \times \Delta^n)_+ \to \coprod_{Y_n} \Sigma(X\times \Delta^n)_+$$ in $\ho_{\ms} \sPreK$, which then produces $$\overline{H} : \coprod_{Y_n} \Sigma (F \times I \times \Delta^n) \to \coprod_{Y_n} \Sigma (X\times \Delta^n).$$

For $i=0,1$, let $\overline{H}|_i :\coprod_{Y_n} \Sigma (F \times \Delta^n) \to \coprod_{Y_n} \Sigma (X\times \Delta^n)$ in $\ho_{\ms} \sPreK$ denote the pullback of $\overline{H}$ by the inclusion $i \to I$.

Note that $\Sigma p_+ \circ (b_E \wedge f_+) \circ \Sigma \Delta_+ = \Sigma f_+$. Unwinding definitions, the map $\Sigma C_n \to \Sigma C_n'$ in the commutative diagram  \eqref{map_cofiber_sequences_filtration_in_Spectral_sequence} is indentified  via equivalences given in Lemma \ref{lemma_formula_C_i(U)} with the map $$ \vee_{N_nY} \Sigma (F_+ \wedge \Delta^n/\partial \Delta^n) \to  \vee_{N_nY} \Sigma(X_+ \wedge \Delta^n/\partial \Delta^n)$$ induced by $\overline{H}|_0$ via the commutative diagram $$ \xymatrix{ \coprod_{L_n Y}\Sigma (F \times \Delta^n) \cup_{\coprod_{L_n Y} \Sigma(F \times \partial \Delta^n)} \coprod_{Y_n} \Sigma(F \times \partial \Delta^n) \ar[d] \ar[rr]^{\overline{H}|_0} &&  \coprod_{L_n Y}\Sigma (X \times \Delta^n) \cup_{\coprod_{L_n Y} \Sigma(X \times \partial \Delta^n)} \coprod_{ Y_n} \Sigma(X \times \partial \Delta^n)\ar[d] \\  \coprod_{Y_n} \Sigma (F \times \Delta^n) \ar[d] \ar[rr]^{\overline{H}|_0} && \coprod_{Y_n} \Sigma (X \times \Delta^n) \ar[d] \\  \vee_{N_nY} \Sigma(F_+ \wedge \Delta^n/\partial \Delta^n) \ar[rr] && \vee_{N_nY} \Sigma(X_+ \wedge \Delta^n/\partial \Delta^n)}.$$  It follows that $H$ induces a homotopy between $\Sigma C_n \to  \Sigma C_n'$ and $\overline{H}|_1$, which is $\vee_{Y_n/L_n Y} ((b_E \circ \Sigma a_+) \wedge \id_{ \Delta^n/\partial \Delta^n})$ showing the result.}

\end{proof}

The description of the map of $E^1$-pages given in Lemma \ref{E1_functoriality} means that understanding the construction taking a map $g: A \to B$ in $\ho_{\ms} \sPreK_{\ast}$ to $\vee_{N_n Y} g: \vee_{N_n Y} A \to \vee_{N_n Y} B$ in $\ho_{\ms} \sPreK_{\ast}$ implies understanding the map of $E^1$-pages. The next lemma and corollary give some understanding of this construction $g \mapsto \vee_{N_n Y} g$ in the case where $(\ms, \sHf_{\ast})$ is $(\Aone, \pi_{\ast}^{s,\Aone})$ or $(P-\Aone, \pi_{\ast}^{s,P,\Aone})$. 

To show Lemma \ref{connectivity_wedge_over_PreSheaf_sets}, it is useful to remark the following:

\begin{remark}\label{p-local-stable-Aone-connectivity_theorem}
If $X$ is a simplicially $n$-connected spectra in the sense that $X \in \Spt(\Sm_k)$ satisfies $\pi_i^s X = 0$ for $i \leq n$, then Morel's stable connectivity theorem \cite{morel2005} implies that $\pi_i^{s,\Aone} X = \pi_i^s L_{\Aone} X = 0$ for $i \leq n$. Because $\pi_i^{s,\PAone}$ is naturally isomorphic to $\ZZ_P \otimes \pi_i^{s,\Aone}$ by Proposition \ref{pr:PAoneLocalHomotopySheaves}, it also follows that $\pi_i^{s,\PAone}(X) = 0$ for $i \leq n$, whence $L_{P} L_{\Aone} X$ is simplicially $n$-connected.
\end{remark} 

\begin{lm}\label{connectivity_wedge_over_PreSheaf_sets}\label{connectivity_wedge_over_PreSheaf_sets_plus_section}
Let $I$ be a presheaf of sets on $\Sm_k$, and $g: A \to B$ be a map in $\sPreK_{\ast}$. Let $(\ms, \sHf_i)$ be either $(\Aone, \pi_i^{s,\Aone})$ or $(P-\Aone, \pi_i^{s,\PAone})$.  
\begin{enumerate}
\item If $g$ induces an isomorphism on $\sHf_i$ for $i <n-1$ and  surjection for $i=n-1$, then $\vee_I g: \vee_I A \to \vee_I B$ induces  an isomorphism on $\sHf_i$ for $i <n-1$ and a surjection for $i=n-1$. 
\item If $g$ induces an isomorphism on $\sHf$ for $i <n$, and there is a map $h: B \to A$ in $\ho_{\ms}\Spt(\Sm_k)$ such that $ g \circ h = \id_{B}$ in $\ho_{\ms}\Spt(\Sm_k)$, then $\vee_I g: \vee_I A \to \vee_I B$ induces an isomorphism on $\sHf_i$ for $i <n$ and a surjection for $i=n$.
\end{enumerate}
\end{lm}

\begin{proof}
The second statement follows from the first, and the first is proven as follows.

Let $C$ denote the $\Aone$-homotopy cofiber of $g$ (i.e. factor $g$ as $g_1 \circ g_2$ with $g_2: A \to B'$ a cofibration and $g_1:B' \to B$ an $\Aone$ fibration and an $\Aone$ weak equivalence and let $C$ be the cofiber of the cofibration $g_2$). Since $g_2$ is also a $P$-$\Aone$ cofibration, and $B' \to B$ is also a $P$-$\Aone$ weak equivalence, we have a long exact sequence \begin{equation}\label{LESsHfi-is-P-and-AonesABC}\to \sHf_{i+1} C \to \sHf_i A \stackrel{\sHf_i(g)}{\to} \sHf_i B \to \sHf_i C \to \ldots.\end{equation} 

Since $\sHf_i(g)$ is an isomorphism for $i<n-1$, we have that $\Image ( \sHf_{i+1} C \to \sHf_i A) = 0$. Thus $ \sHf_{i+1} C = \Ker (  \sHf_{i+1} C \to \sHf_i A)$. Since \eqref{LESsHfi-is-P-and-AonesABC} is exact, $Ker (  \sHf_{i+1} C \to \sHf_i A) = \Image (  \sHf_{i+1} B \to  \sHf_{i+1} C)$. 

Since $\sHf_i(g)$ is a surjection for $i<n$, we have that $\sHf_i B \to \sHf_i C $ is the zero map. Thus for $i < n-1$, we have that $ \Image (  \sHf_{i+1} B \to  \sHf_{i+1} C)=0$, from which it follows that  $\sHf_{i+1} C = 0$. In other words, $\sHf_{i} C = 0$ for $i<n$. 

For any point, $q^*$, the spectrum $q^* L_{\ms} \Sigma^{\infty} C$ satisfies that condition that $\pi_i^s q^*  L_{\ms} \Sigma^{\infty} C = 0$ for $i<n$ because $\pi_i^s L_{\ms} \Sigma^{\infty} C = \sHf_i C = 0$, and $q^* \pi_i^s L_{\ms}\Sigma^{\infty} C =  \pi_i^s q^* L_{\ms} \Sigma^{\infty} C$ \cite[2.2 p. 12]{morel2005}.

Thus $\vee_{q^* I}q^* L_{\ms} \Sigma^{\infty} C $  satisfies that condition that $\pi_i^s \vee_{q^* I}q^* L_{\ms}
\Sigma^{\infty} C = 0$ for $i<n$. Note that $\vee_{q^* I}q^* L_{\ms} \Sigma^{\infty} C = q^* (\vee_I L_{\ms}
\Sigma^{\infty} C)$.

 Thus $q^* \pi_i^s \vee_I L_{\ms} \Sigma^{\infty} C = \pi_i^s q^* (\vee_I L_{\ms} \Sigma^{\infty} C) = 0$. Since $q$ was arbitrary, we conclude that $\pi_i^s \vee_I L_{\ms} \Sigma^{\infty} C = \pi_i^s q^* (\vee_I L_{\ms} \Sigma^{\infty} C) = 0$.

By Remark \ref{p-local-stable-Aone-connectivity_theorem}, we conclude that \begin{equation*}\sHf_i \vee_I L_{\ms} \Sigma^{\infty} C  = 0~\text{ for }~i < n.\end{equation*} By Lemma \ref{wedge_over_PreSheaf_sets_preserves_stable_Aone_w-e}, the map $\vee_I C \to \vee_I L_{\ms} \Sigma^{\infty} C$ is an $ms$-weak equivalence, whence  \begin{equation}\label{sHf_iveeC=0_i<q}\sHf_i \vee_I C  = 0~\text{ for }~i < n.\end{equation}

By definition, we have that $A \stackrel{g_2}{\to} B' \to C$ is such that $g_2$ is a cofibration, and $C$ is the cofiber of $g_2$. By the definition of the global, injective local, $\Aone$, or $\PAone$-model structures, $\vee_I A \stackrel{g_2}{\to} \vee_I B'$ is a cofibration with cofiber $\vee_I C$. The sequence $$ \vee_I A \stackrel{g_2}{\to} \vee_I B' \to \vee_I C$$ therefore gives rise to a long exact sequence in $\sHf_i$. By Lemma \ref{wedge_over_PreSheaf_sets_preserves_Aone_w-e}, this long exact sequence can be written \begin{equation}\label{LESsHfi_piP-AsonesveeIABC}\to \sHf_{i+1} \vee_I C \to \sHf_i \vee_I A \stackrel{\sHf_i(g)}{\to} \sHf_i \vee_I B \to \sHf_i \vee_I C \to \ldots. \end{equation}

The lemma is proven by combining \eqref{LESsHfi_piP-AsonesveeIABC} and \eqref{sHf_iveeC=0_i<q}.

\end{proof}

\begin{co}\label{presheaf_I_to_N_nX(U)}
Suppose that $Y \in \sPreK_{\ast}$ is pointed. Then we may replace the presheaf $I$ of sets on $\Sm_k$ by $U \mapsto N_n Y(U)$ in Lemma \ref{connectivity_wedge_over_PreSheaf_sets}.
\end{co}

\begin{proof}
There is a functorial cofiber sequence $$\vee_{ L_n Y (U)} A \to \vee_{ Y_n(U)} A \to \vee_{N_n Y (U)} A,$$ whence a functorial long exact sequence $$\ldots \to \sHf_i(\vee_{ L_n Y (U)} A) \to \sHf_i(\vee_{ Y_n(U)} A )\to \sHf_i( \vee_{N_n Y (U)} A )\to \sHf_{i-1}(\vee_{ L_n Y (U)} A) \to \ldots.$$ Applying Lemma \ref{connectivity_wedge_over_PreSheaf_sets} with $I = L_n Y$ and $I = Y_n$, the claim follows by the $5$-Lemma. 
\end{proof}

\subsection{A cancelation property}\label{subsection_cancelation_property}

Say that a spectral sequence $E^r_{i,j}$ is a {\em first quadrant spectral sequence} if the differential on the $r$th page is of bidegree $(-r, r-1)$ i.e. $d^r_{i,j}: E^r_{i,j} \to E^r_{i-r,j+(r-1)}$, and if $E^r_{i.j}$ satisfies the condition that $E^r_{i,j}=0$ when $i$ or $j$ is less than $0$. The following lemma is a straight-forward consequence of degree considerations, but we include the proof for completeness.

\begin{lm}\label{map_first_quad_spectral_sequences_iso_below_row_q}
Suppose $\theta^r_{i,j}: E^r_{i,j} \to (E')^r_{i,j}$ is a map of first quadrant spectral sequences such that $\theta^1_{i,j}$ is an isomorphism for $j < q$. Then \begin{enumerate}
\item \label{theta_injective_j<q} $\theta^r_{i,j}$ is injective when $j<q$.
\item \label{theta_iso_j + (r-1)-1 < q} $\theta^r_{i,j}$ is an isomorphism when $j + (r-1)-1 < q$ and $r \geq 2$.
\item \label{theta_iso_j<q_i+j_leq_q} $\theta^r_{i,j}$ is an isomorphism when $j < q$ and $j + i \leq q$.
\end{enumerate}
\end{lm}

\begin{proof}
We prove the claim by induction on $q$. For $q=0$, there is nothing to show. Suppose the claim holds for $q-1$. Now  induct on $r$, which we assume $\geq 2$. Suppose that the claim holds for $r-1$. 
 
\eqref{theta_injective_j<q}: Choose $i,j$ with $j<q$. By the inductive hypothesis on $r$, we have that $\theta^{r-1}_{i,j}$ is injective. Since $(d')^{r-1}_{i,j} \theta^{r-1}_{i,j} = \theta^{r-1}_{i-(r-1), j+(r-1) -1} d_{i,j}^{r-1}$, it follows that $\Ker d_{i,j}^{r-1} \to \Ker (d')^{r-1}_{i,j}$ is injective. It thus suffices to show that $\Image d^{r-1}_{i+(r-1), j- ((r-1) -1)} \to \Image (d')^{r-1}_{i+(r-1), j- ((r-1) -1)}$ is surjective (which is equivalent to being an isomorphism because $\theta^{r-1}_{i,j}$ is injective). Let  $j' = j- ((r-1) -1)$. Note that $j' + ((r-1)-1) -1 = j -1 < q$.  Thus by the inductive hypothesis \eqref{theta_iso_j + (r-1)-1 < q} on $r$, $\theta^{r-1}_{i+(r-1),j- ((r-1) -1)}$ is an isomorphism from which the desired surjectivity follows.
 
\eqref{theta_iso_j + (r-1)-1 < q}: Note that \eqref{theta_iso_j + (r-1)-1 < q} holds for $r=2$. Choose $i,j$ such that $j + (r-1)-1 < q$. Since $j + ((r-1)-1)-1 < q$, we have by induction that $\theta^{r-1}_{i,j}$ is an isomorphism. We show that $\theta^r_{i,j} $ is an isomorphism. For this, it suffices to show that the inclusions $\Ker d^{r-1}_{i,j} \subseteq \Ker  (d')^{r-1}_{i,j}$ and $\Image d^{r-1}_{i+(r-1),j-((r-1)-1)} \subseteq \Image (d')^{r-1}_{i+(r-1),j-((r-1)-1)}$ are isomorphisms. Since $j - ((r-1)-1) + ((r-1)-1) -1 = j-1 <q$, by the inductive hypothesis, we have that $\theta^{r-1}_{i+(r-1),j-((r-1)-1)}$ is an isomorphism. Thus $\Image d^{r-1}_{i+(r-1),j-((r-1)-1)} \subseteq \Image (d')^{r-1}_{i+(r-1),j-((r-1)-1)}$ is an isomorphism. Note that $d^{r-1}_{i,j}$ is a map $E^{r-1}_{i,j} \to E^{r-1}_{i-(r-1),j+((r-1)-1)}$ and similarly for $(d')^{r-1}$. Thus to show that $\Ker d^{r-1}_{i,j} \subset \Ker  (d')^{r-1}_{i,j}$ is an isomorphism, it suffices to show that $\theta^{r-1}_{i-(r-1),j+((r-1)-1)}$ is injective.  Since $j + (r-1)-1 < q$, $\theta^{r-1}_{i-(r-1),j+((r-1)-1)}$ is injective by \eqref{theta_injective_j<q}.

\eqref{theta_iso_j<q_i+j_leq_q}: Choose $i,j$ such that $j < q$ and $j + i \leq q$. By \eqref{theta_injective_j<q}, we have that  $\theta^r_{i,j}$ is injective. We show surjectivity. By the inductive hypothesis, $\theta^{r-1}_{i,j}$ is an isomorphism. Thus is suffices to see that $\Ker d^{r-1}_{i,j} \subseteq \Ker  (d')^{r-1}_{i,j}$ is an isomorphism. Note that if $r=2$, then $\theta^1_{i,j}$ and $\theta^1_{i-1, j}$ are isomorphisms and $\theta^1_{i-1,j} d^1_{i,j} = (d')^1_{i,j}\theta^1_{i.j}$, whence $\Ker d^{1}_{i,j} \subseteq \Ker  (d')^{1}_{i,j}$ is an isomorphism. So, we may assume that $r>2$. Since $(d')^{r-1}_{i,j} \theta^{r-1}_{i,j} = \theta^{r-1}_{i-(r-1), j+(r-1) -1} d_{i,j}^{r-1}$, it suffices to see that $ \theta^{r-1}_{i-(r-1), j+(r-1) -1}$ is injective. For $i - (r-1) < 0$, we have that $E^{r-1}_{i-(r-1), j+(r-1) -1} = 0$ so  $ \theta^{r-1}_{i-(r-1), j+(r-1) -1}$ is injective. Thus we may assume $i -(r-1) \geq 0$ or equivalently $r \leq i + 1$. Thus $j + (r-1) -1\leq  j + (i+1-1) -1 = j + i -1 \leq q -1<q$. Thus $ \theta^{r-1}_{i-(r-1), j+(r-1) -1}$ is injective by \eqref{theta_injective_j<q}.  

\end{proof}

Let $(\ms, \sHf_i)$ be either $(\Aone, \pi_i^{s,\Aone})$ or $(P - \Aone, \pi_i^{s, P, \Aone})$. 

As in Section \ref{subsection_functoriality}, let $j: J \to Y$ be a map of pointed simplicial presheaves, and let $b: \Sigma J \to \Sigma X$ be a map in $\ho_{\ms} \sPreK_{\ast}$ between the suspensions of $J$ and $X$, for $X$ a pointed simplicial presheaf.  Factor $j$ as $j = \iota \circ f$ with $\iota: J \to E$ an $\ms$ weak equivalence and cofibration, and $f: E \to Y$ an $\ms$ fibration. Assume as above that $X$ is fibrant. Then $p: X \times Y \to Y$ is also an $\ms$ fibration.

Let $\mathfrak{b}: \Sigma E \to \Sigma (X \times Y)$ in $\ho_{\ms}\sPreK_{\ast}$ be as in Construction \ref{cons:added_in_process2}. Suppose that $Y$ is $1$-reduced, that is to say, $Y_0 = Y_1 = \pt$.  Let $a: F \to E$ denote the canonical map of simplicial presheaves $F = \pt \times_Y E \to E$. We suppose that the resulting map $\Sigma F \to \Sigma X$ induces a surjection on $\sHf_i$ for all $i$. (Indeed, we will later use this construction when $\Sigma F \to \Sigma X$ has a section up to homotopy.) 

Assume that $\sHf_0(X)$,$\sHf_0(F)$,$\sHf_0(Y)$, and $\sHf_0(E)$ are all $0$. For example, this is satisfied if $X$,$F$,$Y$, and $E$ are $\Aone$-connected, because $\pi_i^{s, P, \Aone}(-) \iso \pi_i^{s, \Aone}(-) \tensor_\ZZ \ZZ_P$ -- see Proposition \ref{pr:PAoneLocalHomotopySheaves}. 

\begin{pr}\label{bsmashf_we_implies_ba_we}
If  $\sHf_i(\Sigma^{-1} \mathfrak{b})$ is an isomorphism for all $i$, then $$\sHf_i (\Sigma^{-1} b \circ  a): \sHf_i F \to \sHf_i X$$ is an isomorphism for all $i$, and $\Sigma^{-1} b \circ  a: \Sinf F \to \Sinf X$ is an $\ms$ weak equivalence.
\end{pr}

\begin{proof}
Since $(\ms, \sHf_i)$ is either $(\Aone, \pi_i^{s,\Aone})$ or $(P - \Aone, \pi_i^{s, P, \Aone})$, the functors $\sHf_i$ detect stable $\ms$ weak equivalences by Propositions \ref{pi_s_detect_we_Aone_and_simplicial}, or \ref{f_s-P-Aone_weak_equiv_iff_pisAoneotimesP_iso} and \ref{pr:PAoneLocalHomotopySheaves}. Thus it suffices to prove that $\sHf_i (\Sigma^{-1} b \circ  a): \sHf_i F \to \sHf_i X$ is an isomorphism for all $i$.

By Lemma \ref{Spectral_sequence_functoriality_wrt_(bEwedgeetc}, we have an induced morphism of spectral sequences $\theta^r_{i,j}: (E^r_{i,j}, d^r) \to ((E')^r_{i,j},(d')^r) $ from the spectral sequence of Definition \ref{global_sk_base_SS} induced by $f$ to the one induced by $p$.

Suppose for the sake of contradiction that $\sHf_i (\Sigma^{-1} b \circ  a): \sHf_i F \to \sHf_i X$ is not an isomorphism for all $i$. Then there exists a minimal $q$ such that $\sHf_{q} F \to \sHf_{q}  X$ is not an isomorphism, and $q>0$ by the assumption that both $\sHf_0 F$ and $\sHf_0 X$ are both $0$. By Lemma \ref{connectivity_wedge_over_PreSheaf_sets_plus_section} and Corollary \ref{presheaf_I_to_N_nX(U)}, it follows that \begin{equation}\label{E1iso_belowq} \theta^1_{i,j} :E^1_{i,j}\stackrel{\cong}{ \to} (E')^1_{i,j}\end{equation}  is an isomorphism for $j < q$. 

By Proposition \ref{properties_of_SS_7.1} \eqref{SS_converges} and Lemma \ref{map_first_quad_spectral_sequences_iso_below_row_q}, $\theta^r_{i,j}$ is an isomorphism for $i+j = q$ and $i>0$ and all $r$. For $r$ sufficiently large, $\theta^r_{i, j} = \theta^{\infty}_{i,j}$ (Proposition \ref{properties_of_SS_7.1} \eqref{SS_converges}), whence $\theta^{\infty}_{i,j}$ is an isomorphism for $i+j = q$ and $i>0$.

Introduce the notation $$0 \subseteq R^0_{n} \subseteq R^1_{n} \subseteq R^2_{n}  \subseteq \ldots  \subseteq R^{n}_{n}= \sHf_{n} (E)$$ for the filtration associated to the spectral sequence $E^r_{ij}$. Let the corresponding filtration associated to the spectral sequence $(E')^r_{i,j}$ be denoted $T^i_{n} \subseteq \sHf_{n} (X \times Y)$. Note that we have the commutative diagram \begin{equation}\label{Tii+1E'infty_from_Rii+1Einfty} \xymatrix{0 \ar[r] & T^i_{q} \ar[r] & T^{i+1}_{q} \ar[r] & (E')^{\infty}_{i+1,q-(i+1)} \ar[r] &0 \\ 
0 \ar[r] &R^i_{q} \ar[u] \ar[r] & R^{i+1}_{q} \ar[r] \ar[u] & E^{\infty}_{i+1,q-(i+1)} \ar[r]  \ar[u]^{\theta^{\infty}_{i+1, q-(i+1)}} &0} \end{equation} for $i=-1,\ldots,q-1$, where by convention $R^{-1}_n=0$ and $T^{-1}_n = 0$ for all $n$. 

Since $\Sigma^{-1} \mathfrak{b}$ induces an isomorphism on $\sHf_q$, we have that $R^q_q \to T^q_q$ is an isomorphism. Applying the $5$-lemma and \eqref{Tii+1E'infty_from_Rii+1Einfty} for $i=q-1,q-2,\ldots, 0$, we conclude that $R^0_q \to T^0_q $ is an isomorphism. By definition, $R^0_q = E^{\infty}_{0,q}$ and $T^0_q= (E')^{\infty}_{0,q}$, so we have that $$\theta^{\infty}_{0,q}: E^{\infty}_{0,q} \to (E')^{\infty}_{0,q}$$ is an isomorphism.

Since $Y_0 = \pt$, the map $\theta^1_{0,1}: E^{1}_{0,q} \to (E')^{1}_{0,q}$ is identified with $\sHf_{q} F \to \sHf_{q}  X$ by Proposition \ref{properties_of_SS_7.1} \eqref{E1_from_F} and Lemma \ref{E1_functoriality}. Since we have assumed that $\sHf_{q} F \to \sHf_{q} X$ is not an isomorphism, it follows that $$\theta^1_{0,q}: E^{1}_{0,q} \to (E')^{1}_{0,q}$$ is not an isomorphism. Since we have assumed that $Y_1=\pt$, Proposition \ref{properties_of_SS_7.1} \eqref{E1_from_F} implies that the domains of $d^1_{1,q}$ and $(d')^1_{1,q}$ are zero. Thus $d^1_{1,q}$ and $(d')^1_{1,q}$ are zero. Thus $\theta^2_{0,q}$ is not an isomorphism. Let $n$ be maximal such that $E^{n}_{0,q} \to (E')^{n}_{0,q}$ is not an isomorphism. Since $ d^n_{0,q} =0$, $ (d')^n_{0,q} =0$ , and $E^{n+1}_{0,q} \to (E')^{n+1}_{0,q}$ is an isomorphism, we conclude that $$\Image d^{n}_{n, q - (n-1)} \to \Image (d')^{n}_{n, q - (n-1)}$$ is not an isomorphism.

Note that $q+1 - i + (r-1) - 1 <q$ when $r < i+1$. Thus by Lemma \ref{map_first_quad_spectral_sequences_iso_below_row_q} \eqref{theta_iso_j + (r-1)-1 < q}, we have that $$\theta^r_{i, q+1-i}: E^{r}_{i, q+1 - i} \to (E')^{r}_{i, q+1 - i}$$ is an isomorphism for $i>1$ and $r =1,2, \ldots, i$. For $r=i+1$, $$E^{i+1}_{i, q+1 - i} \cong \Ker d^{i}_{i, q+ 1 -i}/ \Image d^{i}_{2i, q+1-i -(i-1)}$$ $$(E')^{i+1}_{i, q+ 1 -i} \cong \ker (d')^{i}_{i, q+ 1 -i}/ \Image (d')^{i}_{2i, q+1-i -(i-1)}.$$  

Furthermore, we have the isomorphism $ \Image d^{i}_{2i, q+1-i -(i-1)} \cong   \Image (d')^{i}_{2i, q+1-i -(i-1)}$ by Lemma \ref{map_first_quad_spectral_sequences_iso_below_row_q} \eqref{theta_iso_j + (r-1)-1 < q}. Thus $E^{i+1}_{i, q+1 - i} \stackrel{\subseteq}{\to} (E')^{i+1}_{i, q+1 - i}$ is an injection, and is an isomorphism if and only if $$ \Ker d^{i}_{i, q+ 1 -i} \to \Ker (d')^{i}_{i, q+ 1 -i} $$ is an isomorphism, which happens if and only if  $\Image d^{i}_{i, q+ 1 -i} \to \Image (d')^{i}_{i, q+ 1 -i}$ is an isomorphism.

By the above, we thus conclude that $E^{n+1}_{n, q+1 - n} \stackrel{\subseteq}{\to} (E')^{n+1}_{n, q+1 - n}$ is not surjective.

Note that by degree reasons, $d^m_{n, q+1 - n} = 0$ and $(d')^m_{n, q+1 - n} = 0$ for $m\geq n+1$.  Since by Lemma \ref{map_first_quad_spectral_sequences_iso_below_row_q} \eqref{theta_injective_j<q} and \eqref{theta_iso_j + (r-1)-1 < q} we have $\Image d^m_{n+m, q+1 - n -(m-1)} \cong \Image (d')^m_{n+m, q+1 - n -(m-1)}$, we conclude that $$E^{\infty}_{n, q+1 - n} \stackrel{\subseteq}{\to} (E')^{\infty}_{n, q+1 -n}$$ is not surjective. 

The same reasoning shows that $E^{\infty}_{i, q+1 - i} \stackrel{\subseteq}{\to} (E')^{\infty}_{i, q+1 -i}$ is an isomorphism for $i>n\geq 2$. Explicitly, for $i>n$, we have that $E^{i}_{0,q} \to (E')^{i}_{0,q}$ is an isomorphism by choice of $n$. As above, note that since $ d^i_{0,q} =0$, $ (d')^i_{0,q} =0$ , and $E^{i+1}_{0,q} \to (E')^{i+1}_{0,q}$ is an isomorphism, we conclude that $\Image d^{i}_{i, q - (i-1)} \to \Image (d')^{i}_{i, q - (i-1)}$ is an isomorphism. Continuing with the same reasoning, note that $\theta^i_{i, q+1-i}: E^{i}_{i, q+1 - i} \to (E')^{i}_{i, q+1 - i}$ is an isomorphism by Lemma \ref{map_first_quad_spectral_sequences_iso_below_row_q} \eqref{theta_iso_j + (r-1)-1 < q}, whence $\Ker d^{i}_{i, q - (i-1)} \to \Ker (d')^{i}_{i, q - (i-1)}$ is an isomorphism. Since $ \Image d^{i}_{2i, q+1-i -(i-1)} \cong   \Image (d')^{i}_{2i, q+1-i -(i-1)}$ by Lemma \ref{map_first_quad_spectral_sequences_iso_below_row_q} \eqref{theta_injective_j<q} and \eqref{theta_iso_j + (r-1)-1 < q}, we have that $\theta^{i+1}_{i, q+1-i}: E^{i+1}_{i, q+1 - i} \to (E')^{i+1}_{i, q+1 - i}$ is an isomorphism. By degree reasons $d^m_{i, q+1 - i} = 0$ and $(d')^m_{i, q+1 - i} = 0$ for $m\geq i+1$.  Since by Lemma \ref{map_first_quad_spectral_sequences_iso_below_row_q} \eqref{theta_injective_j<q} and \eqref{theta_iso_j + (r-1)-1 < q} we have $\Image d^m_{i+m, q+1 - i -(m-1)} \cong \Image (d')^m_{i+m, q+1 - i -(m-1)}$, we have that $E^{\infty}_{i, q+1 - i} \stackrel{\subseteq}{\to} (E')^{\infty}_{i, q+1 -i}$ is an isomorphism for $i>n$ as claimed.

Since $\Sigma^{-1} \mathfrak{b}$ induces an isomorphism on $\sHf_{q+1}$, we have that $R^{q+1}_{q+1} \to T^{q+1}_{q+1}$ is an isomorphism. Combining this with the commutative diagram \begin{equation}\label{Tiq+1_Riq+1_CD}\xymatrix{0 \ar[r] & T^i_{q+1} \ar[r] & T^{i+1}_{q+1} \ar[r] & (E')^{\infty}_{i+1,q+1-(i+1)} \ar[r] &0 \\ 
0 \ar[r] &R^i_{q+1} \ar[u] \ar[r] & R^{i+1}_{q+1} \ar[r] \ar[u] & E^{\infty}_{i+1,q+1-(i+1)} \ar[r]  \ar[u]^{\theta^{\infty}_{i+1, q+1-(i+1)}} &0} \end{equation} for $i=q,q-1,\ldots,n$ and the five lemma, we see that $R^{i}_{q+1} \to T^{i}_{q+1}$ is an isomorphism for $i=q,q-1,\ldots,n$.

In particular, $R^{n}_{n+1} \to T^{i}_{q+1}$ is surjective, which by \eqref{Tiq+1_Riq+1_CD} with $i=n-1$  implies that $E^{\infty}_{n, q+1 - n} \stackrel{\subseteq}{\to} (E')^{\infty}_{n, q+1 -n}$ is surjective, giving the desired contradiction. 
\end{proof}

Here is a verbal description of the proof of Proposition \ref{bsmashf_we_implies_ba_we}. Choose $q$ minimal such that $\pi_{q}^{s,\Aone}(F) \to \pi_{q}^{s,\Aone}(X)$ is not an isomorphism. The failure to be an isomorphism is necessarily a failure of injectivity. In terms of the map of spectral sequences, this implies that $E^1_{0,1} \to (E')^1_{0,q}$ is not injective. Since $q$ is minimal, degree arguments with first quadrant spectral sequences imply that $E^{\infty}_{i,q-i} \to (E')^{\infty}_{i, q-i}$ are isomorphisms for $i>0$. Since $\pi_q^{s,\Aone}((b \wedge \Sigma^{\infty} f_+) \circ \Sigma^{\infty} \Delta_+)$ is an isomorphism, and since $\oplus_i E^{\infty}_{i,q-i}$ is the associated graded of $\pi_q^{\Aone,s}(E)$ and the analogous statement for $E'$ and $X \times Y$ holds, it follows that $E^{\infty}_{0,q} \to  (E')^{\infty}_{i, q-i} $ is also an isomorphism. Thus we can choose a maximal $n \geq 2$ for which $E^n_{0,q} \to (E')^n_{0,q}$ is not an isomorphism. For degree reasons, the failure of this map to be an isomorphism is a failure of injectivity. Thus the image of a $d$ must be larger than the image of a $(d')$. Since the domains of these differentials have smaller second index, these domains actually have to be isomorphic. This then implies that the kernel of the $d$ is smaller than the kernel of the $(d')$. This leads to a contradiction with the surjectivity of $\pi_{q+1}^{s,\Aone}(F) \to \pi_{q+1}^{s,\Aone}(X)$.

\section{\texorpdfstring{$\Aone$}{A1} simplicial EHP fiber sequence}\label{Section_Aone_simplicial_EHP_fiber_sequence}

\begin{df}\label{Def:fiber_sequence_up_to_homotopy}
Say that the sequence $X \to Y \to Z$ in $\Spaces_{\ast}$ is a $\ms$ fiber sequence up to homotopy if there is a diagram $$\xymatrix{  X \ar[d] \ar[r] & \ar[d] Y \ar[r] & \ar[d] Z \\ F \ar[r] & E \ar[r]^f & B }$$ which commutes up to homotopy with $B$ fibrant, $f$ a fibration with fiber $F = \ast \times_B E$, and all the vertical maps weak equivalences. 
\end{df}

\begin{remark}\label{Morel_fibration_sequence_remark}
Morel defines $X \to Y \to Z$ to be a simplicial fibration sequence if the composition $X \to Z$ is the constant map and the induced map from $X$ to the homotopy fiber of $Y \to Z$ in the injective local model structure is a simplicial weak equivalence.

He then defines $X \to Y \to Z$ to be an $\Aone$-fibration sequence if $\Laone X \to \Laone Y \to \Laone Z$ is a simplicial fibration sequence. See \cite[Definition 6.44]{morel2012}. 

In this vein, it is natural to define $X \to Y \to Z$ to be a $P-\Aone$ fibration sequence if $$L_{P} \Laone X \to L_{P} \Laone Y \to L_{P} \Laone Z$$ is a simplicial fibration sequence. 

Note that if $X \to Y \to Z$ in $\Spaces_{\ast}$ is a $P-\Aone$ fiber sequence up to homotopy as in \ref{Def:fiber_sequence_up_to_homotopy}, then $B$, $E$, and $F$ are $P$-$\Aone$ local, from which it follows that they can be identified with $L_P \Laone X$, $L_P \Laone E$, and $L_P \Laone F$ respectively. Since $P$-$\Aone$ fibrations are simplicial fibrations, we have that a $P-\Aone$ fiber sequence up to homotopy is a $P-\Aone$ fibration sequence.
\end{remark}

Note that for $P$ the set of all primes, the $P$-$\Aone$ injective model structures on $\Spaces$ and $\Spt(\Sm_k)$ are the $\Aone$ injective model structures, and $X \to L_P X$ is the identity map.

For $X \in \sPreK_{\ast}$, let $j: J(X) \to J(X^{\wedge 2})$ be $j_2$ from Definition \ref{presheaf_combinatorial_extension_and_j_n}. Recall from Section \ref{subsection_homotopy_spheres} and before the notation $S^{n+q\alpha} = S^n \wedge \Gm^{\wedge q}$. Recall from Sections \ref{Section:introduction} and \ref{Section:stable_isomorphism} the notation $ - \langle -1 \rangle$ in $\GW(k)$. Recall from Section \ref{sec:localization} that for a set of primes $P$, $\ZZ_P$ denotes $\ZZ$ with all primes not in $P$ inverted.

\begin{tm} \label{p-local_fiber_sequence_J(X)}
Let $X = S^{n+q \alpha}$ with $n>1$, and let $e = (-1)^{n+q} \langle -1 \rangle^q$. Let $P$ be a set of primes, and suppose that for all $m \in \ZZ_{>0}$, the element $(m+1)+me$ is a unit in $\GW(k) \otimes \ZZ_P$. Then \begin{equation}\label{sequence_X_J(X)_J(X2)_sectionA1fiber_sequence} X \to J(X) \stackrel{j}{\to} J(X^{\wedge 2})\end{equation} is a fiber sequence up to homotopy in the $P$-$\Aone$ injective model structure on $\Spaces$.
\end{tm}

By Proposition \ref{J(X)=ho-nis=Omega_SigmaX}, Theorem \ref{p-local_fiber_sequence_J(X)} proves Theorem \ref{EHP-fiber-sequence_without_localization_introduction}.

\begin{proof}
Let $\ms$ denote the $P$-$\Aone$ injective model structure on $\Spaces$ and $\Spt (\Sm_k)$.

Recall the notation $D(X) = \vee_{j=0}^{\infty} X^{\wedge j}$ from Definition \ref{J_etc_def}.  From Section \eqref{Hilton-Milnor_Snaith_section}, we have the zig-zag  \eqref{zigzag_for_SigmaJ}$$\xymatrix{ \Sigma J(X)_+ \ar[r] & \ssimp \vert \Sigma D(X) \vert & \ar[l] \Sigma D(X)}$$ of weak equivalences in the global model structure. 

Let $b_1: D(X) \to L_P \Laone X$ denote the composition of the map $D(X) \to X$ which for $j \neq 1$ crushes the summands $X^{\wedge j}$ to the base point with the canonical map $X \to L_P \Laone X$. 

Replace $J(X^{\wedge 2})$ by a fibrant simplicial presheaf $J(X^{\wedge 2}) \to L_P \Laone J(X^{\wedge 2})$. Let $f': E' \to L_P \Laone J(X^{\wedge 2})$ be a fibrant replacement of $J(X)\stackrel{j}{\to} J(X^{\wedge 2}) \to L_P \Laone J(X^{\wedge 2})$ is the $\ms$ model structure. Let $$E = J(X^{\wedge 2}) \times_{L_P \Laone J(X^{\wedge 2})} E'$$ be the pull-back of $E'$ to $J(X^{\wedge 2})$, and let $f: E \to J(X^{\wedge 2})$ be the canonical projection. (The motivation for defining $E'$ is to obtain the diagram \eqref{XJJX2toFE'L} below which fits precisely into Definition \ref{Def:fiber_sequence_up_to_homotopy}. If this is overly pedantic, the reader may consider a fibrant replacement of $j$.) Note that $f$ is an $\ms$ fibration, and that there is a canonical map $J(X) \to E$. Since $\ms$ is a proper model structure and $J(X^{\wedge 2}) \to L_P \Laone J(X^{\wedge 2})$ is a $\ms$ weak equivalence, we have that $E \to E'$ is a weak equivalence. Since $J(X) \to E'$ is a $\ms$-weak equivalence by construction, by the $2$-out-of-$3$ property, it follows that the canonical map $J(X) \to E$ is a $\ms$ weak equivalence. We thus have a commutative diagram $$\xymatrix{ J(X) \ar[rd]_j \ar[r]^{\cong_{\ms}} & E \ar[d]^{f}\\ & J(X^{\wedge 2})}$$ with the map $J(X) \stackrel{\cong}{\to} E$ a $\ms$-weak equivalence.

It follows that we have the zig-zag $$ \xymatrix{ \Sigma E & \ar[l]_{\cong} \Sigma J(X)_+ \ar[r]^{\cong} & \ssimp \vert \Sigma D(X) \vert & \ar[l]_{\cong} \Sigma D(X) \ar[r]^{b_1}& \Sigma L_P \Laone X }$$ which determines a map $b: \Sigma^{\infty} E \to \Sigma^{\infty} L_P \Laone X$ in $\ho_{\ms}\Spt(\Sm_k)$. 

As in Section \ref{subsection:stable_weak_equivalence} and \ref{subsection_functoriality}, we obtain a map $$(b_+ \wedge \Sigma^{\infty} f_+) \circ \Sigma^{\infty} \Delta_+ :  \Sigma^{\infty} E_+ \to \Sigma^{\infty} (L_P \Laone X \times J(X^{\wedge 2}))_+.$$ in $\ho_{\ms}\Spt(\Sm_k)$. The hypothesis that for all $m \in \ZZ_{>0}$, the element $((m+1)+me)$ is a unit in $\GW(k) \otimes \ZZ_P$ implies that $(b_+ \wedge \Sigma^{\infty} f_+) \circ \Sigma^{\infty} \Delta_+$ is an isomorphism by Proposition \ref{pr:stableIsomorphism}. By Section \ref{section:low-dimensional_simplices_J}, $J(X^{\wedge 2})$ is $1$-reduced and $L_P \Laone X$ is fibrant, so we may apply Proposition \ref{bsmashf_we_implies_ba_we}. It follows that $b \circ \Sigma^{\infty} a: \Sigma^{\infty} F \to \Sigma^{\infty} X$ is an isomorphism in $\ho_{\ms}\Spt (\Sm_k)$, where $a: F \to E$ is the inclusion of the fiber of $f$ into $E$.

Since $E$ is the pull-back of $E'$ we have the diagram in $\Spaces$ $$\xymatrix{ F \ar[r] & E' \ar[r] & L_P \Laone J(X^{\wedge 2}) \\ F \ar[u] \ar[r] & E \ar[u]^{\cong_{\ms}} \ar[r] & J(X^{\wedge 2}) \ar[u]^{\cong_{\ms}}}.$$

We conclude that $b \circ \Sigma^{\infty} a': \Sigma^{\infty} F \to \Sigma^{\infty} X$ is an isomorphism in $\ho_{\ms}\Spt (\Sm_k)$, where $a'$ is the composition  in $\ho_{\ms}\Spt (\Sm_k)$ corresponding to the zig-zag  $F \to E' \leftarrow E$.

Since the composition of the two maps in the sequence \eqref{sequence_X_J(X)_J(X2)_sectionA1fiber_sequence} is constant, the composition $$X \to J(X) \to L_P \Laone J(X^{\wedge 2})$$ is also constant. We therefore have an induced map $$h: X \to F.$$ By construction of $b_1$, the composition $\Sigma^{\infty}X \to \Sigma^{\infty} E \stackrel{b}{\to} \Sigma^{\infty} L_P \Laone X$ is the canonical map associated to the identity in $\ho_{\ms}\Spt(\Sm_k)$, and in particular is an isomorphism. Thus the composition $\Sigma^{\infty} X \stackrel{\Sigma^{\infty} h}{\to} \Sigma^{\infty} F \stackrel{\Sigma^{\infty} a'}{\to} \Sigma^{\infty} E \stackrel{b}{\to} \Sigma^{\infty} X$ is an isomorphism. Since $b \circ \Sigma^{\infty} a'$ is an isomorphism in $\ho_{\ms}\Spt(\Sm_k)$, so is $\Sigma^{\infty} h$.

Note that $L_P h: L_P X \to L_P F$ is a map of $\Aone$-simply connected objects. $\Sinf L_P h$ is an $\Aone$ weak equivalence as follows. By Corollary \ref{co:sspLoc_P=Loc_Pssp}, we have that $\Sinf L_P h \weq L_P \Sinf h$. Since $\Sinf h$ is a $P$-$\Aone$ weak equivalence, we have that $L_P \Laone \Sinf h$ is a simplicial weak equivalence. By Proposition \ref{pr:LPLaone=LaoneLP},  we have $L_P \Laone \Sinf h \weq \Laone L_P  \Sinf h$, whence $\Laone L_P  \Sinf h$ is a simplicial weak equivalence so $L_P \Sinf h$ is an $\Aone$ weak equivalence as claimed. We now apply Corollary \ref{co:destabilization} to conclude that $L_P h$ is an $\Aone$ weak equivalence, whence $h$ is an $\ms$-weak equivalence.

The diagram \begin{equation}\label{XJJX2toFE'L}\xymatrix{  X \ar[r] \ar[d] & J(X) \ar[r]^j \ar[d]& J(X^{\wedge 2}) \ar[d] \\ F \ar[r] & E' \ar[r] & L_P \Laone J(X^{\wedge 2}) }\end{equation} shows that \eqref{sequence_X_J(X)_J(X2)_sectionA1fiber_sequence} is an $\ms$ fiber sequence up to homotopy.
\end{proof}

\begin{proof} (Corollary \ref{co_fiber_sequence_specific-P_introduction}) 
\begin{itemize}
\item By Theorem \ref{p-local_fiber_sequence_J(X)}, it is sufficient to show that $(m+1)+me$ is a unit in $\GW(k) \otimes \ZZ_{(2)}$ for all positive integers $m$. This was shown in Corollary \ref{co:2localUnits}.
\item Apply Theorem \ref{p-local_fiber_sequence_J(X)} and Corollary \ref{co:GWUnits}. 
\item When $n$ is odd and $q$ is even, we have $e = -1$, whence $((m+1) + m e) = 1$ which is a unit in $\GW(k)$.
\item When $2 \eta = 0$, we have $2 \eta \rho = 0$, whence $\eta \rho$ is torsion in $\GW(k)$. All torsion elements of $\GW(k)$ are nilpotent \cite[Chapter VIII, \S 8.1]{Lam2005} (or because the Grothendieck-Witt ring is a $\lambda$-ring, and all torsion elements of $\lambda$ rings are nilpotent by a result of Graeme Segal), so $\eta \rho$ is nilpotent. When $n+q$ is odd, $$e = (-1)^n (-1)^q (1+ \rho \eta)^q = - (1+ \rho \eta)^q  = \begin{cases} - 1 -  \rho \eta &\mbox{if $q$ odd} \\ 
-1 & \mbox{if $q$ even.} \end{cases} $$ Thus, $1+e$ is nilpotent, from which it follows that $((m+1)+me) = 1 + (e+1)m$ is a unit in $\GW(k)$.

The fact that $2 \eta = 0$ when $k= \CC$ is shown \cite[Remark 6.3.5, Lemma 6.3.7]{morel2004}.
\end{itemize}
\end{proof}

\begin{co}\label{LES_of_Aone_simplicial_EHP_fiber_sequence}
Let $X = S^{n+q \alpha}$ with $n>1$. Choose $v\in \ZZ$.There is a functorial long exact sequence $$\ldots \to \ZZ_{(2)} \otimes \pi_{i+ v \alpha}^{\Aone} X \to  \ZZ_{(2)} \otimes \pi_{i+1+v \alpha}^{\Aone} \Sigma X \to  \ZZ_{(2)} \otimes \pi_{i+1+v \alpha}^{\Aone} \Sigma (X \wedge X) \to \ZZ_{(2)} \otimes \pi_{i-1+ v \alpha}^{\Aone} X \to  \ldots  .$$
\end{co}

\begin{proof}
Combining Theorem \ref{p-local_fiber_sequence_J(X)} and Corollary \ref{co_fiber_sequence_specific-P_introduction}, we have that $$X \to J(X) \to J(X^{\wedge 2}) $$ is a $2$-$\Aone$ fiber sequence up to homotopy. It follows that there is an associated long exact sequence in $\pi_{\ast+v \alpha}^{2,\Aone}$. See Proposition \ref{pr:lesPAone}. By Proposition \ref{pr:pAoneGmHsheaves} and \cite[Theorem 6.13]{morel2012}, we may replace $\pi_{\ast + v \alpha}^{2,\Aone}$ with $\ZZ_{(2)} \otimes \pi_{\ast + v \alpha}^{\Aone}$. By Corollary \ref{co:pi_iAoneJ=pi_i+1AoneSigma}, we can identify the homotopy groups $\pi_{\ast + v \alpha}^{\Aone} J(X)$ with $ \pi_{\ast + 1 + v \alpha}^{\Aone} \Sigma X$. This yields the claimed long exact sequence.
\end{proof}

The long exact sequences of Corollary \ref{LES_of_Aone_simplicial_EHP_fiber_sequence} form an exact couple, which in turn gives rise to an $\Aone$ simplicial EHP spectral sequence. Here the adjective ``simplicial" refers to the suspension with respect to the simplicial circle $S^1$.

\begin{tm}\label{EHP_spectral_sequence}
Choose $q,v \in \ZZ_{\geq 0}$ and $n \in \ZZ$ such that $n \geq 2$. There is a spectral sequence $$(E^{r}_{i,m}, d_r: E^{r}_{i,m} \to E^{r}_{i-1,m-r}) \Rightarrow \ZZ_{(2)} \otimes \pi_{i-n, v-q}^{s,\Aone} =  \ZZ_{(2)} \otimes \pi^{s,\Aone}_{i+ v \alpha} S^{n + q \alpha}$$ with $E^1_{i,m} =   \ZZ_{(2)} \otimes \pi^{\Aone}_{m+1+i + v \alpha} (S^{2m+2n+1 + 2q \alpha})$ if $i \geq 2n-1 +m$ and otherwise $E^1_{i,m} = 0$. 
\end{tm}

\begin{proof}
The $E^1$ page is as claimed by the construction of the exact couple and by the fact that $\pi_{i+ v \alpha}^{\Aone} S^{n+q \alpha} = 0$ for $i<n$. The latter fact follows from Morel's connectivity theorem\cite[Theorem 6.38]{morel2012}. 

The spectral sequence converges for degree reasons; for all $(i,m)$ there are only finitely many $r$ with a non-zero differential leaving or entering $E^r_{i,m}$.

Since by definition we have $\colim_m \pi^{\Aone}_{m+i + v \alpha} \Sigma^m S^{n + q \alpha} =  \pi^{s,\Aone}_{i+ v \alpha} S^{n + q \alpha}$ (see Section \ref{subsection:Spectra}), we also have $$\colim_m (\ZZ_{(2)} \otimes \pi^{\Aone}_{i+m + v \alpha} \Sigma^m S^{n + q \alpha}) = \ZZ_{(2)} \otimes \pi^{s,\Aone}_{i+ v \alpha} S^{n + q \alpha}$$ and it follows that the spectral sequence converges to $ \ZZ_{(2)} \otimes \pi^{s,\Aone}_{i + v \alpha} S^{n + q \alpha}$.
\end{proof}

As in the setting of algebraic topology, one obtains truncated EHP spectral sequences converging to unstable homotopy groups of spheres.

\begin{tm}\label{truncated_EHP_spectral_sequence}
Choose $q,v \in \ZZ_{\geq 0}$ and $n_1, n_2 \in \ZZ$ such that $n_2 \geq n_1 \geq 2$. There is a spectral sequence $(E^{r}_{i,m}, d_r: E^{r}_{i,m} \to E^{r}_{i-1,m-r}) \Rightarrow \ZZ_{(2)} \otimes \pi^{\Aone}_{i+ v \alpha} S^{n_2 + q \alpha}$ with $$E^1_{i,m} =   \begin{cases} \ZZ_{(2)} \otimes \pi^{\Aone}_{m+1+i+ v \alpha} (S^{2m+2n_1+1 + 2q \alpha}) &\mbox{if $i \geq 2n-1 +m$ and $0 \leq m< n_2  -n_1$} \\ 
0 & \mbox{otherwise. }  \end{cases} $$ 
\end{tm}

\begin{proof}

The long exact sequences of Corollary \ref{LES_of_Aone_simplicial_EHP_fiber_sequence} for $X = S^{n_1 + m+q \alpha}$ with $0 \leq m < n_2 - n_1$ can be combined with the long exact sequence $$ \to 0 \to \ZZ_{(2)} \otimes \pi_{i+ v \alpha}^{\Aone} S^{n_2 +q \alpha} \to  \ZZ_{(2)} \otimes \pi_{i+ v \alpha}^{\Aone} S^{n_2 +q \alpha} \to 0 \to \ldots $$ associated to the $2$-$\Aone$ fiber sequence $$S^{n_2 +q \alpha} \to S^{n_2 +q \alpha} \to \ast ,$$ (which replaces in Theorem \ref{EHP_spectral_sequence} the long exact sequences of Corollary \ref{LES_of_Aone_simplicial_EHP_fiber_sequence} for $X = S^{n_1 + m +q \alpha}$ with $n_2 - n_1 \leq m$) to form an exact couple.  The $E^1$-page equals to $E^1$ page of the EHP sequence constructed in Theorem \ref{EHP_spectral_sequence} for $m <n_2 -n_1$, and $E^1_{i,m} = 0$ for $m \geq n_2 -n_1$. The convergence is clear.
 
\end{proof}

\bibliographystyle{amsalpha}
\bibliography{EHP}

\end{document}